\documentclass[pdflatex,sn-mathphys-num]{sn-jnl}

\usepackage{graphicx}%
\usepackage{aliascnt}
\usepackage{amsmath,amssymb,amsfonts}%
\usepackage{amsthm}%
\usepackage[capitalize]{cleveref}
\usepackage[title]{appendix}%
\usepackage{xcolor}%

\usepackage{booktabs}%
\usepackage{orcidlink}
\usepackage[T1]{fontenc}

\usepackage{float}
\usepackage{enumitem}
\usepackage{todonotes}
\usepackage{xspace}
\usepackage[ruled]{algorithm2e}
\usepackage{blkarray}
\usepackage{subcaption}
\usepackage{pgfplots}
\usepackage{bbm}
\usepackage{color}
\usepackage{tikz} 
\usepackage{xparse}

\usepackage{siunitx}
\usepackage{placeins} 
\usepackage{mathtools}
\usepackage{longtable}

\crefname{algocf}{Algorithm}{Algorithms}
\Crefname{algocf}{Algorithm}{Algorithms}

\theoremstyle{thmstyleone}%

\newtheorem{definition}{Definition}

\newaliascnt{proposition}{definition}
\newtheorem{proposition}[proposition]{Proposition}
\aliascntresetthe{proposition}
\crefname{proposition}{Proposition}{Propositions}
\Crefname{proposition}{Proposition}{Propositions}

\newaliascnt{problem}{definition}

\aliascntresetthe{problem}
\crefname{problem}{Problem}{Problems}
\Crefname{problem}{Problem}{Problems}

\newaliascnt{theorem}{definition}
\newtheorem{theorem}[theorem]{Theorem}
\aliascntresetthe{theorem}
\crefname{theorem}{Theorem}{Theorems}
\Crefname{theorem}{Theorem}{Theorems}

\newaliascnt{lemma}{definition}
\newtheorem{lemma}[lemma]{Lemma}
\aliascntresetthe{lemma}
\crefname{lemma}{Lemma}{Lemmata}
\Crefname{lemma}{Lemma}{Lemmata}

\theoremstyle{thmstylethree}%
\newaliascnt{corollary}{definition}
\newtheorem{corollary}[corollary]{Corollary}
\aliascntresetthe{corollary}
\crefname{corollary}{Corollary}{Corollaries}
\Crefname{corollary}{Corollary}{Corollaries}

\theoremstyle{thmstyletwo}%
\newaliascnt{example}{definition}
\newtheorem{example}[example]{Example}
\aliascntresetthe{example}
\crefname{example}{Example}{Examples}
\Crefname{example}{Example}{Examples}

\newcommand{\conv}{\mathrm{conv}}
\newcommand{\proj}{\mathrm{proj}}

\newcommand{\diag}{\mathrm{diag}}

\newcommand{\bracket}[2]{[#1,#2]}

\DeclareDocumentCommand\transpose{m}{#1^{\intercal}}
\DeclareDocumentCommand\Z{}{\mathbb{Z}}

\DeclareDocumentCommand\Q{}{\mathbb{Q}}
\DeclareDocumentCommand\R{}{\mathbb{R}}
\DeclareDocumentCommand\M{}{\mathbb{MI}}

\DeclareDocumentCommand\MP{}{\mathcal{P}}
\DeclareDocumentCommand\MQ{}{\mathcal{Q}}

\DeclareDocumentCommand\zerovec{}{\mathbbm{O}}
\DeclareDocumentCommand\onevec{}{\mathbbm{1}}

\DeclareDocumentCommand\cplxNP{}{\mathsf{NP}}

\DeclareDocumentCommand\orderO{o}{\mathcal{O}\IfValueTF{#1}{\left(#1\right)}{}}

\DeclareDocumentCommand\Rolf{m}{}

\begin{document}
\title[Folding Mixed-Integer Linear Programs and Reflection Symmetries]{Folding Mixed-Integer Linear Programs and
Reflection Symmetries}

\author{\fnm{Rolf} \sur{van der Hulst} \orcidlink{0000-0002-5941-3016}
}

\affil{\orgname{University of Twente}, 
\city{Enschede}, \country{The Netherlands}}

\abstract{For mixed-integer linear programming and linear programming it is well known that symmetries can have a negative impact on the performance of branch-and-bound and linear optimization algorithms.
A common strategy to handle symmetries in linear programs is to reduce the dimension of the linear program by aggregating symmetric variables and solving a linear program of reduced dimension. In their work \emph{Dimension Reduction via Color Refinement} (DRCR), Grohe, Kersting, Mladenov and Selman show that it is sufficient to run a fast color refinement algorithm to detect permutation symmetries and reduce the dimension of the linear program. 
We extend DRCR in two directions. First, we show that DRCR can be extended to reflection symmetries, which generalize permutation symmetries. Second, we show the folklore result that DRCR can be applied to the continuous columns of mixed-integer linear programs. In order to derive additional reductions on the integer variables we use \emph{affine totally unimodular decompositions} to reformulate mixed-integer linear programs into mixed-integer linear programs with fewer integer variables.
Computational experiments on MIPLIB 2017 collection set using SCIP 10 show that DRCR is an effective tool for handling symmetries.
For the linear programming relaxations, DRCR with reflection symmetries yields a modest reduction in running time compared to the original DRCR procedure. For mixed-integer linear programming models, DRCR is very effective at reducing the solution time compared to the default configuration of SCIP. Moreover, the developed DRCR detection algorithms are fast and scale well to large problem instances.
}
\keywords{linear programming, reflection symmetry, mixed-integer linear programming, symmetry, total unimodularity}
\maketitle
\section{Introduction}
In many problems in mathematical optimization,  symmetries arise naturally. For the representative MIPLIB 2017 collection~\cite{Gleixner2021}, at least 47\% of the instances in the collection contains some permutation or reflection symmetry~\cite{Hojny2025-si}.
Symmetry can pose a serious challenge for (spatial) branch-and-cut algorithms in mixed-integer linear programming (MILP) and mixed-integer nonlinear programming (MINLP). If symmetries are not handled adequately, they may cause a branch-and-cut algorithm to repeatedly explore symmetric parts of the search space~\cite{Margot2010}. In the worst case, all symmetric solutions to a problem are enumerated, which can cause a major performance degradation in the branch-and-cut algorithm, and make it difficult to solve the given optimization problem efficiently.\\

Previous works that develop symmetry handling methods for mathematical optimization can loosely be categorized into two categories. 
The first category consists of symmetry handling methods that remove symmetries from the problem formulation by \emph{projection}, which remove symmetric parts of the feasible region by projecting onto the fixed space of the symmetry group. By doing so, one obtains a symmetry-free formulation with fewer variables and constraints, which is attractive from a computational perspective. This method is used primarily for convex optimization problems such as linear programming~\cite{Bodi2013} and semidefinite programs~\cite{Gatermann2004}.\\

In contrast to the symmetry handling methods for convex optimization problems, mixed-integer linear programming solvers typically handle symmetry by \emph{breaking} it, either through specialized branching rules~\cite{Ostrowski2011} or by adding valid inequalities~\cite{Liberti2012}. Due to the non-convexity of the integrality constraints, the projection-based methods used for convex optimization problems no longer apply as they do not necessarily preserve integrality. Although breaking symmetries is still very popular for mixed-integer linear programming, a few recent works, such as the work on Orbital Shrinking~\cite{Fischetti2012,Salvagnin2013, Fischetti2017}, challenge the idea that symmetry handling methods for mixed-integer linear programming must necessarily break symmetry. This work falls into the same category and formulates symmetry-handling methods that bridge the gap between the strong symmetry handling methods for linear programs to the more challenging setting of mixed-integer linear programs. As our main result, we show that the strong dimension reduction method for linear programs from~\cite{Grohe2014} can be extended to mixed-integer linear programs. Secondly, we show that the dimension reduction method from~\cite{Grohe2014} can be extended to include the class of reflection symmetries, which generalizes the permutation symmetries that are handled by their dimension reduction method.
Next, we consider some of the relevant previous works in further detail.\\

\noindent\textbf{Symmetry handling for linear programs}\\
For linear programs, B{\"o}di, Herr and Joswig~\cite{Bodi2013} show that one can exploit the convexity of the linear program by projecting the linear program onto the fixed space of its symmetry group. This method is computationally attractive, as it both removes the symmetries from the linear program and reduces its dimension. Furthermore, B{\"o}di, Herr and Joswig present an algorithm that exploits the projection onto the fixed space to solve highly symmetric integer programs whose symmetry group is (almost) transitive.
Their work was inspired by the symmetry handling for Semidefinite Programs~\cite{Gatermann2004,Bachoc2012,Permenter2020}, where symmetries of Semidefinite Programs are exploited by applying results from invariant theory and representation theory to reduce the dimensions of the Semidefinite Programs. 

One downside of projecting onto the fixed space of the symmetry group is that the computation of the symmetry group of the Linear Program may be too expensive compared to simply solving the symmetric original program. Typically, the symmetry group is computed by detecting automorphisms of an auxilliary graph representing the (Mixed-Integer) Linear Program~\cite{Margot2010}. The exact complexity of the graph automorphism problem is unknown, and no polynomial time algorithms are known~\cite{Read1977}. However, there are several software packages such as \texttt{bliss}~\cite{Junttila}, \texttt{nauty}~\cite{McKay2025} and \texttt{saucy}~\cite{Darga} that can handle large graphs in practice. \\

Grohe, Kersting, Mladenov and Selman~\cite{Grohe2014} observed that instead of computing the symmetry group via graph automorphism on an auxilliary graph, it is sufficient to run a color refinement method to reduce the dimension of the underlying LP, which they call \emph{Dimension Reduction via Color Refinement} (DRCR). Some sources also refer to their technique as \emph{LP folding} or \emph{weak symmetry}. Their approach has two key advantages compared to the projection onto the fixed space of the symmetry group. First of all, it generalizes the handling of permutation symmetries and may provide additional reductions. Secondly, the color refinement approach is fast in theory and fast in practice since it runs in linearithmic time. One downside of DRCR is that the color refinement algorithm presented in~\cite{Grohe2014} only detects permutation symmetries, whereas B{\"o}di, Herr and Joswig~\cite{Bodi2013} show that projecting to the fixed space of the symmetry group works for the more general linear symmetry group. However, this is not a major limitation in practice, as the majority of symmetries in (Mixed-Integer) Linear Programs are permutation symmetries~\cite{Hojny2025-si}. The developers of the FICO Xpress solver report that after implementing DRCR that they  `observed improvements of an average of up to 30\% on standard LP benchmark sets'~\cite{Berthold2018}. In this case, the average is a somewhat skewed measure; for the affected linear programming models that can be reformulated, applying DRCR even leads to solving times that are `often more than 5 times as fast on average'. However, the technique is only applicable to a limited set of models, which make up roughly 10\% of their test set. 
Similar results are reported by Gurobi \cite{GurobiOptimization2020}. 
To the best of our knowledge, DRCR has currently not yet been implemented in any open source solver. This is somewhat surprising, as a speedup of up to 30\% for linear programming solvers from a single technique is very rare. DRCR can be considered one of the most important algorithmic advances for linear programming software in the last 15 years.\\

\noindent\textbf{Symmetry handling for mixed-integer linear programming}\\
Although the symmetry handling methods for Linear Programming are very powerful, they are unfortunately not directly applicable to Mixed-Integer Linear Programming, as the dimension reduction by projecting to the fixed space relies on the convexity of the problem. For Mixed-Integer Linear Programs, the non-convex integrality constraints invalidate the use of the reduction. Instead, there is a vast collection of symmetry handling methods, where most methods handle symmetry by breaking it in a static or dynamic fashion. Some methods attempt to eliminate the symmetry by adding symmetry-breaking constraints~\cite{Liberti2012,Liberti2014}. Other methods, such as Isomorphism Pruning~\cite{Margot2001,Margot2003} or Orbital Branching~\cite{Ostrowski2008,Ostrowski2011} handle symmetry dynamically by pruning symmetric branches from the branch-and-bound tree.
For a more complete overview, we refer the reader to the surveys by Margot~\cite{Margot2010} and Pfetsch and Rehn~\cite{Pfetsch2019}.
One symmetry handling method that does not break symmetry is discussed by Pfetsch and Rehn. They consider the fixing of continuous variables to the fixed space of the symmetry group whose automorphisms only permute the continuous variables. From their experiments, they conclude that ``fixing the continuous variables slightly speeds up LP-solving, but does not result in a significant overall speed-up'' compared to other symmetry handling methods. \\

One method for symmetry handling for MILP which is relevant to the current work is \emph{Orbital Shrinking}, which was introduced by Fischetti and Liberti~\cite{Fischetti2012}, and whose theory was later solidified in~\cite{Fischetti2017}. The key idea of Orbital Shrinking is to aggregate variables in each orbit of the symmetry group to a single variable per orbit to obtain a smaller mixed-integer program. This is similar to the projection to the fixed space of the symmetry group for linear programs, which also aggregates variables in each orbit. However, due to the non-convexity of the integrality constraints, the reduced mixed-integer problem defines a relaxation of the original problem. Thus, solutions to the relaxed problem are checked by solving a subproblem to check feasibility in the original model. Frequently, this feasibility problem is solved using Constraint Programming techniques~\cite{ Salvagnin2013,Salvagnin2012}. Because orbital shrinking defines a relaxation, any dual bound generated for the relaxed model is also valid for the original model. Due to the smaller size of the relaxed model and the fact that the reduced model contains fewer/no symmetries, the computed dual bounds from the reduced model are typically quite strong for highly symmetric instances~\cite{Salvagnin2013,Salvagnin2012}.

Another innovative method for MILP that preserves symmetry is the core point method~\cite{Herr2013}, which extend the more theoretical work by B{\"o}di, Herr and Joswig~\cite{Bodi2013}. Although they do show that their algorithm is effective for integer programs with (nearly) transitive symmetry groups, their algorithm does not perform as well as orbital branching on arbitrary MILP instances. For a careful comparison between Orbital Shrinking and the core point method, we recommend~\cite[section 3]{Fischetti2017}.

Finally, most modern MIP solvers have fast presolving procedures for detecting simple symmetries arising from parallel columns and parallel rows~\cite{Achterberg2016}, which are typically implemented using 2-level hashing schemes \cite{Achterberg2016,Gemander2020}. This can also be considered a form of variable aggregation to remove symmetries, although the treated symmetries are very simple. In~\cite[Section 7.1]{Achterberg2016} a special structure is presented where non-overlapping symmetries  can be aggregated. In particular, if there is an automorphism of the ILP such that variables contained in the same orbit appear in disjoint constraints, then one can aggregate all the variables in the orbit into one variable without breaking linear and integer feasibility.

\paragraph{Contribution}
We extend the DRCR procedure from~\cite{Grohe2014} in several directions. First of all, we show that DRCR for Linear Programs can be extended to reflection symmetries, which arise from the action of signed permutation matrices on linear programs. Our main idea for realizing this uses two transformations. First, we reformulate the original LP by centering the origin at the center of the variable domains using an affine transformation. Then, we split each variable into two variables, where one variable corresponds to the positive domain and one variable corresponds to the negative domain. We show that by applying DRCR to this reformulated linear program, one can recover the variable complementation and row scaling that achieves the best possible reduction among all linear programs obtained by complementing variables and negating rows. The procedure can be implemented efficiently with only minor modifications to the original DRCR procedure.
\\

For our second contribution, we aim to close the gap between the strong symmetry handling methods for linear programs and the symmetry handling methods for mixed-integer linear programs by specializing DRCR to mixed-integer linear programs.  
In particular, we show the `folklore' result that DRCR can also be applied to the continuous columns of mixed-integer linear programs. 
Additionally, we further investigate the application of DRCR in MILP to also aggregate integer variables. We use \emph{implied integrality}~\cite{VanDerHulst2025} and \emph{affine TU decompositions}~\cite{BaderHWZ18} to infer redundant integrality constraints on a subset of variables. Then, the implied convexity of these variables is exploited by applying DRCR to aggregate these variables. Our approach can also be thought of as \emph{Exact Orbital Shrinking}, where one shows that the orbital shrinking relaxation has the same strength as the original problem, and the feasibility subproblem is an integer program that is given by a perfect formulation and can be solved using linear programming. We formulate an algorithm that specializes DRCR for mixed-integer linear programs. 
One of its crucial components are fast detection algorithms for \emph{network matrices}~\cite{BixbyWagner1988,RowNetworkMatrixPaper}, which are a subclass of totally unimodular matrices~\cite{Schrijver86}. The developed algorithm does not handle the complete permutation symmetry group, but only aggregates variables that appear in automorphisms where all the permuted variables have no integrality constraints or have integrality constraints that can be reformulated using affine TU decompositions or implied integrality. \\

We conduct experiments on the MIPLIB 2017 collection set~\cite{Gleixner2021}. For the linear programs, we consider the linear relaxations, and show that taking reflection symmetries into account in DRCR can lead to additional reductions and a reduced model size and soution time. For the DRCR algorithm for MILP, we show that it has two favorable characteristics. First, the aggregation of variables to remove symmetries from the model is effective in reducing the running time of affected mixed-integer linear programs. In a computational comparison using SCIP 10~\cite{hojny2025scipoptimizationsuite100}, the affected mixed-integer linear programming instances in MIPLIB 2017 that can be solved by SCIP or our method are solved, on average, more than twice as fast compared to the default configuration of SCIP 10. 
Second, we observe that our detection algorithm for DRCR in MILP is fast and scales well to large problems, much like in the linear programming case.

\paragraph{Outline}
In \cref{sec_notation} we introduce DRCR in further detail and define relevant notation.
\Cref{sec_reflection_symmetry} extends DRCR to also handle reflection symmetries that arise from signed permutation matrices. In \cref{sec_dimred_milp}, we explore how DRCR can be extended and specialized for mixed-integer linear programs. In \cref{sec_algorithm}, we formulate a DRCR algorithm that detects reflection symmetries in MILP problems. 
\Cref{sec_results} presents computational experiments on MIPLIB 2017.

\section{Dimension Reduction via Color Refinement}
\label{sec_notation}

Let $V$ and $W$ be two disjoint finite sets. Throughout this work, we will index matrices with the sets $V$ and $W$ for the rows and columns, respectively, and denote any matrix $A$ indexed by $V$ and $W$ using $A\in \R^{V\times W}$. Its elements are indexed using $A_{v,w}$ for $v\in V$, $w\in W$. For $P\subseteq V$, $Q\subseteq W$ we use $A_{P,Q}$ to denote the submatrix induced by $P$ and $Q$ that contains all elements $A_{v,w}$ for $(v,w) \in P\times Q$. For vectors, we use similar notation such as $b\in \R^V$ and $b_v$ for $v\in V$ to denote the vector and its elements, respectively. The order of the elements of matrices will not be relevant, unless noted otherwise.
For two vectors $y,z \in \R^W$, we use $\max(y,z)$ to be the element-wise maximum of $y$ and $z$, such that for all $w\in W$ we have $\max(y,z)_w \coloneqq \max(y_w,z_w)$. Note that for any vector $x\in \R^W$, we have the identity $x \coloneqq \max(x,\zerovec) - \max(-x,\zerovec)$.
For a partition $\MP$ of a set $V$ and some set $V'\subseteq V$, we say that \emph{$\MP$ is compatible with $V'$} if for all $P\in \MP$ either $P\subseteq V'$ or $P\cap V' = \emptyset$ holds.

\subsection{Equitable partitions and fractional automorphisms}

First, we introduce equitable partitions of linear programs as presented in~\cite{Grohe2014}. 
Given a matrix $A\in \R^{V\times W}$, DRCR iteratively computes partitions $\MP_i$ and $\MQ_i$ of $V$ and $W$, the rows and columns of $A$. Let $\MP_0 = \{ V\}$ and $\MQ_0 = \{ W\}$. 
A partition $(\MP,\MQ)$ of the rows and columns of $A$ is said to be \emph{equitable} if for all $P\in \MP,Q\in\MQ$ it holds that
\begin{align}
    \label{eq_folding_row_sums}
    \sum_{w\in Q} A_{v,w} &= \sum_{w\in Q} A_{v',w}\text{ for all } v,v'\in P\\
    \label{eq_folding_col_sums}
    \sum_{v\in P} A_{v,w} &= \sum_{v\in P} A_{v,w'}\text{ for all } w,w'\in Q.
\end{align} In words, for every block of $A$ defined by $P,Q$, the row sums and columns sums restricted to elements in the block must be equal. For a vector $b\in \R^V$, we similarly say that $\MP$ is an \emph{equitable partition} of $b$ if $b_v = b_{v'}$ holds for all $v,v'\in P$ for all $P\in \MP$. If for $(P,Q)\in \MP \times \MQ$ both \eqref{eq_folding_row_sums} and \eqref{eq_folding_col_sums} hold, we also say that $(P,Q)$ defines an \emph{equitable block} of $A$. A partition $\MP_1$ is said to \emph{refine} a partition $\MP_2$ if for every $P_1\in \MP_1$ there exists some $P_2\in \MP_2$ such that $P_1\subseteq P_2$. Color refinement works by iteratively refining $\MP_i$ and $\MQ_i$, by splitting any classes $P\in \MP_i$ or $Q\in\MQ_i$ into multiple classes in $\MP_{i+1}$ and $\MQ_{i+1}$ if they do not satisfy the conditions \eqref{eq_folding_row_sums} and \eqref{eq_folding_col_sums} respectively. For some iteration $i$, this iterative process satisfies $\MP_i = \MP_{i+1}$ and $\MQ_i = \MQ_{i+1}$. For this $i$,
the partition $(\MP_\infty,\MQ_\infty) \coloneqq (\MP_i,\MQ_i)$ is known as the \emph{coarsest equitable partition} of $A$. The authors of~\cite{Grohe2014} adapt the well-known fast color refinement algorithm by Paige and Tarjan~\cite{Paige1987}, that is based on ideas by Hopcroft~\cite{Hopcroft1971}, to compute the coarsest equitable partition. For $n = |V| + |W|$ and the total bitlength $m$ of the entries of $A$, they show that one can compute the coarsest equitable partition in $\orderO((n+m)\log n)$ time. \\

Given a partition $\MP$ of a set $V$, its \emph{partition matrix} is denoted by $\Pi_{\MP}\in \{0,1\}^{V\times\MP}$, for which $(\Pi)_{vP} = 1 \iff v\in P$ holds for all $v\in V$ and $P\in \MP$. A \emph{scaled partition matrix} $\widetilde{\Pi_\MP}$ is a normalized partition matrix defined such that $(\widetilde{\Pi_\MP})_{vP} = \begin{cases} \frac{1}{|P|} \text{ if $v\in P$} \\
0\text{ otherwise}\end{cases}$ holds for all $v\in V$ and $P\in \MP$. \Cref{thm_part_id} summarizes a few identities for partition matrices and scaled partition matrices that are established and used in~\cite{Grohe2014}.
\begin{proposition}[\cite{Grohe2014}]
\label{thm_part_id}
Given a partition $\MP$ of a set $V$, the following hold:

\begin{enumerate}[label=(\roman*)]
    \item \label{thm_part_id_id}$\transpose{\widetilde{\Pi_\MP}} \Pi_\MP = \transpose{\Pi}_\MP \widetilde{\Pi_\MP} = I_{\MP\times \MP} $.
    \item \label{thm_part_id_ds}For $v,v'\in V$, $(\widetilde{\Pi_\MP}\transpose{\Pi_\MP})_{v,v'} = (\Pi_\MP \transpose{\widetilde{\Pi_\MP}})_{v',v} = \begin{cases}
        \frac{1}{|P|} \text{ if $v,v'\in P$ for some $P\in \MP$}\\
        0 \text{ otherwise.}
    \end{cases}$
    \item \label{thm_part_id_equitable} If $\MP$ is an equitable partition of $b\in \R^V$, then $b= \widetilde{\Pi_{\MP}} \transpose{\Pi_{\MP}} b = \Pi_\MP \transpose{\widetilde{\Pi_\MP}}b$ holds. 
\end{enumerate}
\end{proposition}
\begin{proof}
For~\ref{thm_part_id_id}, consider any $P,P'\in \MP$ and note that $(\transpose{\widetilde{\Pi_\MP}} \Pi_\MP)_{P,P'} = \begin{cases}
    \sum_{v\in P} \frac{1}{|P|} \text{ if $P=P'$}\\
    0 \text{ otherwise}
\end{cases} = I_{\MP\times\MP}$ holds. The identity in \ref{thm_part_id_id} follows since $I_{\MP\times \MP} = \transpose{(\transpose{\widetilde{\Pi_\MP}} \Pi_\MP)} = \transpose{\Pi}_\MP \widetilde{\Pi_\MP}$ holds.
The proof of \ref{thm_part_id_ds} follows by the definition of matrix transpose and multiplication. For \ref{thm_part_id_equitable} we have for $v\in V$ and $P\in \MP$ such that $v\in P$ that $(\transpose{\widetilde{\Pi_\MP}} \Pi_\MP b)_v = \sum_{v'\in P} \frac{1}{|P|} b_{v'} = b_v$, where the last equality holds since $\MP$ is an equitable partition of $b$, which implies that $b_v = b_{v'}$ holds for all $v'\in P$.
\end{proof}

A matrix $A\in \R^{V\times W}$ is said to be \emph{stochastic} if it is nonnegative and $\sum_{w\in W} A_{v,w} = 1$ for all $v\in V$, and it is \emph{doubly stochastic} if both $A$ and $\transpose{A}$ are stochastic. For a matrix $A\in \R^{V\times W}$, a pair $(X,Y) \in \R^{V\times V} \times \R^{W\times W}$ of doubly stochastic matrices is said to be a \emph{fractional automorphism} of $A$ if $XA = AY$ holds. One of the key insights made by the authors of~\cite{Grohe2014} is given in \cref{thm_frac_aut_equitable_partition}, and relates fractional automorphisms of $A$ to its equitable partitions.
\begin{proposition}(\cite{Grohe2014})
\label{thm_frac_aut_equitable_partition}
    If $(\MP,\MQ)$ is an equitable partition of a matrix $A\in \R^{V\times W}$, then $(\Pi_\MP \transpose{\widetilde{\Pi_\MP}}, \Pi_\MQ \transpose{\widetilde{\Pi_\MQ}})$ is a fractional automorphism of $A$.
\end{proposition}
In \cref{thm_frac_aut_equitable_partition}, it can be easily checked that $\Pi_\MP \transpose{\widetilde{\Pi_\MP}}$ and $\Pi_\MQ \transpose{\widetilde{\Pi_\MQ}}$ are both doubly stochastic due to \cref{thm_part_id}\ref{thm_part_id_ds}. From \cref{thm_part_id}\ref{thm_part_id_ds} and \cref{thm_frac_aut_equitable_partition} it also follows that if $(\MP,\MQ)$ is an equitable partition, that $(\widetilde{\Pi_\MP} \transpose{\Pi_\MP}, \widetilde{\Pi_\MQ} \transpose{\Pi_\MQ})$ is a fractional automorphism of $A$. 

\subsection{Dimension Reduction of a Linear Program}

Throughout this work, we consider the following linear program in standard form.

\begin{equation}
    \label{lp_orig_signed}
    \min \,\transpose{c} x \text{  s.t. } A x = b,~\ell \leq x \leq u  \tag{$F(A,b,\ell,u,c)$}
\end{equation}
with $A\in \R^{V\times W}$, $b\in \R^V$, $\ell_w \in \R \cup \{-\infty\}$ and $u\in \R\cup \{\infty\}$ such that $\ell_w \leq u_w$ for all $w\in W$.
We use $F(A,b,\ell,u,c)$ to denote the linear program, and we define its feasible region as $P(A,b,\ell, u) \coloneqq \{ x\in \R^W \mid A x = b, \ell \leq x \leq u\}$.
For the linear program $F(A,b,\ell, u, c)$ and a partition $\MP$ of $V$ and a partition $\MQ$ of $W$, $(\MP,\MQ)$ is an \emph{equitable partition of $F(A,b,\ell,u,c)$} if $(\MP,\MQ)$ is an equitable partition of $A$, $\MP$ is an equitable partition of $b$, and $\MQ$ is an equitable partition of $\ell$, $u$, and $c$. The key result necessary for DRCR is given in \cite[Lemma 7.1]{Grohe2014} and shows that a linear program that has an equitable partition $(\MP,\MQ)$ can be reduced to a linear program with a constraint matrix with dimension $|\MP|\times |\MQ|$. We show a variation of this result which has essentially the same proof, but which produces a slightly different but equivalent linear program whose variables and constraints are scaled differently. The motivation for deviating from \cite{Grohe2014} in \cref{thm_lp_folding} will become clear in \cref{sec_dimred_milp}, where we consider mixed-integer linear programs.

\begin{proposition}[Grohe, Kersting, Mladenov and Selman~\cite{Grohe2014}
]
    \label{thm_lp_folding}
    For given $A\in \R^{V\times W}$, $b\in \R^{V}$, $c\in \R^{W}$ and vectors $\ell, u$ with $\ell_w \in \R \cup \{-\infty\}$ and $u_w\in \R \cup \{\infty\} $ such that $\ell_w \leq u_w$ holds for all $w\in W$, consider the linear program $F(A,b,\ell, u, c)$. Let $(\MP,\MQ)$ be an equitable partition of $F(A,b,\ell,u,c)$.  Define 
    $A' \coloneqq \transpose{\Pi_{\MP}} A \widetilde{\Pi_{\MQ}}$, $b'\coloneqq \transpose{\Pi_{\MP}} b$, $\ell' \coloneqq \transpose{\Pi_{\MQ}} \ell$, $u' \coloneqq \transpose{\Pi_\MQ} u$ and $c' \coloneqq  \transpose{\widetilde{\Pi_\MQ}} c$, and consider the reduced linear program $F(A',b',\ell',u',c')$. Then, the following hold:
    \begin{enumerate}[label=(\roman*)]
        \item \label{thm_lp_folding_to_red} If $x\in \R^W$ is feasible for $P(A,b,\ell, u)$, then $y \coloneqq \transpose{\Pi_{\MQ}} x$ is feasible for $P(A',b',\ell', u')$ and $\transpose{c'}y = \transpose{c} x$ holds.
        \item \label{thm_lp_folding_from_red} If $y\in \R^\MQ$ is feasible for $P(A',b',\ell', u')$, then $x \coloneqq \widetilde{\Pi_{\MQ}} y$ is feasible for $P(A,b,\ell,u)$ and $\transpose{c} x = \transpose{c'} y$ holds
    \end{enumerate}
\end{proposition}
\begin{proof}
    Let us prove the first point. Let $x\in P(A,b,\ell,u)$ be given and consider $y\coloneqq \transpose{\Pi_{\MQ}} x$. First of all, we show that for each $Q\in \MQ$, $\ell'_Q \leq y_Q \leq u'_Q$ holds. Note that  $\ell'_Q = (\transpose{\Pi_\MP} \ell)_Q = \sum_{w\in Q} \ell_w$ and $u'_Q = (\transpose{\Pi_\MP} u)_Q = \sum_{w\in Q} u_w$ and $y_Q  = \sum_{w\in Q} x_w$ hold. Then, $\ell'_Q \leq y_Q \leq u'_Q$ is equivalent to $\sum_{w\in Q}\ell_w \leq \sum_{w\in Q} x_w \leq  \sum_{w\in Q} u_w$, which is clearly valid by summing the inequalities $\ell_w \leq x_w \leq u_w$ for all $w\in Q$. Then, we show that $y$ is feasible by showing that the equation $A' y = b'$ holds. The following equations hold.
    \begin{equation*}
        A' y = \transpose{\Pi_{\MP}} A \widetilde{\Pi_{\MQ}} \transpose{\Pi_{\MQ}} x \stackrel{\textup{(a)}}{=} \transpose{\Pi_{\MP}} \widetilde{\Pi_{\MP}} \transpose{\Pi_{\MP}}  A x \stackrel{\textup{(b)}}{=} \transpose{\Pi_{\MP}}  A x = \transpose{\Pi_{\MP}}  b = b' 
    \end{equation*}

    The equality in \textup{(a)} follows from the fact that $(\MP,\MQ)$ is an equitable partition of $A$, which implies that $(\widetilde{\Pi_{\MP}} \transpose{\Pi_{\MP}}, \widetilde{\Pi_{\MQ}} \transpose{\Pi_{\MQ}})$ is a fractional automorphism of $A$ using \cref{thm_frac_aut_equitable_partition}. The equality in \textup{(b)} follows from \cref{thm_part_id}\ref{thm_part_id_id}. Thus, $y\in P(A',b',\ell',u')$ holds, and $y$ is indeed feasible. Finally, since $\MQ$ is an equitable partition of $c$ it follows that $\transpose{c'} y = \transpose{c} \widetilde{\Pi_\MQ} \transpose{\Pi_\MQ} x = \transpose{c} x$, where we use \cref{thm_part_id}\ref{thm_part_id_equitable} in the final step.

    Next, we proceed with the proof of second point. Let $y\in P(A',b',\ell',u')$ be a given feasible point, and consider $x\coloneqq \widetilde{\Pi_\MQ} y$.  
    First, note that since $\MQ$ is an equitable partition of $\ell$ and $u$, that for any $Q\in \MQ$ we have $\ell'_Q = |Q| \ell_w$ for all $w\in Q$ and $u'_Q = |Q| u_w$, since the $l_w$ and $u_w$ entries must be identical for all $w\in Q$ and are summed by definition of $\transpose{\Pi_\MQ}$. Additionally, note that for each $w\in W$, there exists exactly one $Q\in \MQ$ such that $x_w = \frac{1}{|Q|} y_Q$. 
    Then, we have for every $w\in W$ and $Q\in \MQ$ with $w\in Q$ that $\ell'_Q \leq y_Q \leq u'_Q$ holds, which implies that $\frac{1}{|Q|} \ell'_Q \leq \frac{1}{|Q|} y_Q \leq \frac{1}{|Q|} u'_Q$ holds. By substituting using $\ell'_Q = |Q| \ell_w$, $u'_Q = |Q| u_w$ and $x_w = \frac{1}{|Q|} y_Q$, we see that $\ell_w \leq x_w \leq u_w$ holds for all $w\in W$.
    
    Then, similar to the proof of the first point, we consider $Ax$ and show that it equals $b$.
    \begin{equation*}
        A x = A \widetilde{\Pi_\MQ} y \stackrel{\textup{(c)}}{=} A \widetilde{\Pi_\MQ} \transpose{\Pi_\MQ} \widetilde{\Pi_\MQ} y \stackrel{\textup{(d)}}{=} \widetilde{\Pi_\MP} \transpose{\Pi_\MP} A  \widetilde{\Pi_\MQ} y = \widetilde{\Pi_\MP} A' y = \widetilde{\Pi_\MP} b' = \widetilde{\Pi_\MP} \transpose{\Pi_\MP} b \stackrel{\textup{(e)}}{=} b 
    \end{equation*}
    Here, \textup{(c)} holds by \cref{thm_part_id}\ref{thm_part_id_id} and \textup{(d)}holds since $(\MP,\MQ)$ is an equitable partition of $A$, which implies a fractional automorphism using \cref{thm_frac_aut_equitable_partition}. Additionally, \textup{(e)} follows by \cref{thm_part_id}\ref{thm_part_id_equitable} and the fact $\MP$ is an equitable partition of $b$. Thus, $x$ is feasible for $P(A,b,\ell,u)$. 
    Finally, note that $\transpose{c} x = \transpose{c} \widetilde{\Pi_\MQ} \transpose{\Pi_\MQ} x = \transpose{c'} y$ holds, where the first equality follows from \cref{thm_part_id}\ref{thm_part_id_equitable} and the fact that $\MQ$ is an equitable partition of $c$. 
\end{proof}

Occasionally, we will use the notation
$R^{\MP,\MQ}(F(A,b,\ell,u,c)) \coloneqq F(\transpose{\Pi_{\MP}} A \widetilde{\Pi_\MQ}, \transpose{\Pi_{\MP}} b, \transpose{\Pi_{\MQ}} \ell, \transpose{\Pi_{\MQ}} u, 
\transpose{\widetilde{\Pi_{\MQ}}} c)$
to denote the reduced linear program in \cref{thm_lp_folding}.

The key difference in the proof of \cref{thm_lp_folding} compared to the proof in \cite[Lemma 7.1]{Grohe2014} is that we use the fractional automorphism $(\widetilde{\Pi_{\MP}} \transpose{\Pi_{\MP}}, \widetilde{\Pi_{\MQ}} \transpose{\Pi_{\MQ}})$, instead of $({\Pi_{\MP}} \transpose{\widetilde{\Pi_{\MP}}}, {\Pi_{\MQ}} \transpose{\widetilde{\Pi_{\MQ}}})$, which leads to different variable and constraint scaling in the reduced linear programs and in the mappings from and to the original linear program. 

Although we present our results in standard form, equitable partitions can also be applied to linear programs in inequality form by first introducing slack variables and turning them into equality form. In \cref{thm_ineq_equivalent} we show that this procedure does not affect equitability of the constraint matrix. 
\begin{proposition}
    \label{thm_ineq_equivalent}
    Let $(\MP,\MQ)$ be an equitable partition of $A\in \R^{V\times W}$. Then $(\MP, \MQ \cup \MP)$ is an equitable partition of $\begin{bmatrix}
        A & I
    \end{bmatrix}$.
\end{proposition}
\begin{proof}
Consider any block $(P,Q) \in \MP\times (\MQ \cup \MP)$. If $Q \in \MQ$, then $\begin{bmatrix}
    A & I 
\end{bmatrix}_{PQ} = A_{PQ}$ holds, and $A_{PQ}$ satisfies the equitability conditions \eqref{eq_folding_row_sums} and \eqref{eq_folding_col_sums} since $(\MP,\MQ)$ is an equitable partition of $A$. Otherwise, for $Q\in \MP$, we have that $\begin{bmatrix}
    A & I 
\end{bmatrix}_{PQ} = I_{P,P'}$ for some $P'\in \MP$. If $P= P'$, then $I_{P,P}$ is a $P\times P$ identity matrix, which satisfies \eqref{eq_folding_row_sums} and \eqref{eq_folding_col_sums}. Otherwise, if $P\neq P'$, then $I_{P,P'}$ is an all-zero matrix, which also satisfies \eqref{eq_folding_row_sums} and \eqref{eq_folding_col_sums}. Thus,  $\MP\times (\MQ \cup \MP)$ is an equitable partition of $\begin{bmatrix}
    A & I
\end{bmatrix}$. 
\end{proof}
Note that the introduced slack variables will all have identical bounds and do not participate in the objective, so the equitability conditions for the variable bounds and objective are also satisfied under the addition of slack variables. However, the standard form may have additional reductions compared to inequality form if the slack variables can be aggregated with non-slack variables. In~\cite{Geyer2019} the authors observe that additional reductions on the slack variables in equality form lead to large running time improvements for algorithms that enumerate all non-symmetric solutions of certain MILP. Intuitively, this can be explained by the fact that certain relationships between the slack of the constraints and the variables become more explicit upon adding the slack variables.
Thus, we consider the more general standard form in the remainder of this work.\\

The authors of~\cite{Grohe2014} show that DRCR is advantageous primarily because it is fast to compute and because it generalizes permutation symmetries. We have not yet discussed one of its more relevant downsides, which is that a basic optimal solution to the reduced LP is not necessarily mapped to a basic solution of the original LP. This is particularly relevant for branch-and-bound approaches that utilize an LP relaxation, as for these the reduction for the LP may not be valid for the MILP, and thus we need to obtain a solution to the original LP. For MILP, basic linear programming solutions are necessary in an efficient branch-and-cut implementation, which relies heavily on the dual simplex algorithm to branch quickly and to efficiently find cutting planes such as those formulated by Gomory~\cite{Gomory60} from the simplex tableau.
For DRCR, the original solution obtained from the the reduced solution may lie in the interior of the optimal face of the original LP. Then, to obtain a basic solution, one has to run a crossover procedure for the original problem to obtain a feasible vertex from the interior point~\cite{Megiddo1991}. In some cases, the crossover procedure has a runtime that may be orders of magnitude larger than solving the reduced linear program. For example, for problems with a transitive symmetry group, the linear program computed by DRCR consists of a single variable and is easily solved, but the crossover procedure may be as expensive as solving the original linear program. 

\section{Extending DRCR to reflection symmetries}

\label{sec_reflection_symmetry}
As mentioned in the introduction, the DRCR algorithm presented in~\cite{Grohe2014} only detects automorphisms that arise from permutation matrices, whereas the theory they formulate supports more general linear symmetries of the constraint matrix. In~\cite{Bodi2013} it is shown that for linear programs the projection to the fixed space of the symmetry group works for any linear group, rather than just the symmetries arising from permutation matrices. The algorithm presented in \cite{Bodi2013} uses reflection symmetries generated by the signed permutation group to solve highly symmetric integer programs. Reflection symmetries generalize permutation symmetries to signed permutations and additionally take the action of complementing a variable $x$ to $u-x$ into account, where $u$ is its upper bound. \\

Most MILP solvers and LP solvers only perform symmetry handling for permutation symmetries. SCIP is a notable exception, as it handles reflection symmetries since version 9.0~\cite{Bolusani2024}. More information on how reflection symmetries are detected and handled in SCIP can be found in~\cite{Hojny2025-si}. As far as we are aware, the only other work that treats reflection symmetries computationally for MILP solvers is by Christophel, G{\"u}zelsoy and P{\'o}lik~\cite{Christophel2014} and explores how reflection symmetries may be used in primal heuristics. Reflection symmetries that are not permutation symmetries are somewhat uncommon in Mixed Integer Programs; Hojny~\cite{Hojny2025-si} reports that they were detected for only 60 out of 1055 tested MIPLIB 2017 instances and 6 out of 486 tested MINLPLIB instances, and that handling of reflection symmetries for mixed-integer linear programs does not produce large speedups. In fact, the default settings for symmetry handling in SCIP 9.0 disable the detection and handling of reflection symmetries. This can perhaps be explained by the interesting result by Geyer, Bulutoglu and Ryan~\cite{Geyer2019}, who show that for feasible linear programs that are full-dimensional and do not contain redundant inequalities, permutation symmetries can be used to detect the complete linear symmetry group of the linear program. To do so, they first normalize the linear program using a singular-value decomposition and then formulate a complicated algorithm to detect the full symmetry group. They also detect reflection symmetries and exploit them in an enumeration procedure for an integer linear program used for classifying orthogonal arrays.\\

As our first contribution, we will show that DRCR can be extended to detect reductions that take the complements of every variable into account. We will consider the original linear program given by $F(A,b,\ell,u,c)$, and transform it to obtain a representation where DRCR is more effective. 
Our approach somewhat resembles the detection strategies in \cite{Christophel2014,Hojny2025-si}, which both reduce the problem of finding reflection symmetries to finding permutation symmetries by adding both the original and the complemented variables to the auxilliary graph used for symmetry detection. We will use two transformations to achieve something similar using DRCR.
First, the origin is moved to the center of the variable domains using an  affine transformation, so that $\ell \leq x\leq u$ is transformed into $-\nu\leq x' \leq \nu $, where $\nu = \frac{u-\ell}{2}$. Then, each variable $x'_w$ for $w\in W$ is split into two variables $x'^+_w$ and $x'^-_w$ such that $x'_w = x'^+_w - x'^-_w$ holds, where $0\leq x'^+ \leq \nu$ and $0\leq x'^-\leq \nu$ hold. Crucially, we show that by computing an equitable partition on this linear program with $x'^+$ and $x'^-$, one can recover the best equitable partition possible for the original linear program under the operation of complementing the variables.

We illustrate the main ideas of our method with an example linear program given in~\eqref{eq_example_1}.
\begin{subequations}
\label{eq_example_1}
\begin{align}
    \min\,  x_1 - x_2 & & \\
    \begin{bmatrix}
        2 & 1 & 1 \\
        -1 & -2 & -1 \\
        -1 & 1 & 0
    \end{bmatrix}
    \begin{bmatrix}
        x_1\\
        x_2\\
        x_3
    \end{bmatrix}
    &\leq \begin{bmatrix}
    5\\
    -3\\
    1\\
\end{bmatrix}\\
0 \leq x_i &\leq 2 & \forall i\in \{1,2,3\}
\end{align}    
\end{subequations}

The polyhedron that defines the feasible region of~\eqref{eq_example_1} has two reflection symmetries, one reflection in the hyperplane $x_1 + x_2 = 2$ and one reflection in the hyperplane $x_3 = 1$. We will show how both symmetries can be projected out by considering a reformulation of~\eqref{eq_example_1}.
First, we apply an affine transformation to~\eqref{eq_example_1} that yields variables that have identical negative and positive domain sizes by substituting $x = x' + \onevec$. Furthermore, we add the slack variables $s$ to turn the inequalities into equalities.

\begin{subequations}
\label{eq_example_2}
\begin{align}
    \min\,  x'_1 - x'_2 & \\
    \begin{bmatrix}
        2 & 1 & 1 & 1 & 0 & 0 \\
        -1 & -2 & -1 & 0 & 1 & 0\\
        -1 & 1 & 0 & 0 & 0 & 1
    \end{bmatrix}
    \begin{bmatrix}
        x'_1\\
        x'_2\\
        x'_3 \\
        s_1\\
        s_2\\
        s_3
    \end{bmatrix}
    & = \begin{bmatrix}
    1\\
    1\\
    1\\
\end{bmatrix}\\
-1 \leq x'_i &\leq 1 & \forall i\in \{1,2,3\}\\
s_i &\geq 0 &\forall i \in \{1,2,3\}
\end{align}    
\end{subequations}

Next, we introduce redundant variables and constraints in the following manner. First, we split each variable $x'_w$ using $x'_w = x'^+_w - x'^-_w$ for $w\in \{1,2,3\}$. After doing so, we additionally add a negated copy of each equation. Then, the mappings defined by $x'_w \gets x'^+_w - x'^-_w$ and $(x'^+_w,x'^-_w) \gets (\max(x_w,0), \max(-x_w,0))$ can be used to show that feasibility and optimality are preserved. This yields the following equivalent linear program, which we will call the \emph{split reformulation} more generally.
\begin{subequations}
\label{eq_example_3}
\begin{align}
    \min\,  x'^+_1 - x'^-_1 - x'^+_2  + x'^-_2 & \\
    \begin{bmatrix}
        2 & 1 & 1 & -2 & -1 & -1 & 1 & 0 & 0 \\
        -1 & -2 & -1 & 1 & 2 & 1 & 0 & 1 & 0\\
        -1 & 1 & 0 & 1 & -1 & 0 & 0 & 0 & 1 \\
        -2 & -1 & -1 & 2 & 1 & 1 & -1 & 0 & 0 \\
        1 & 2 & 1 & -1 & -2 & -1 & 0 & -1 & 0\\
        1 & -1 & 0 & -1 & 1 & 0 & 0 & 0 & -1 
    \end{bmatrix}
    \begin{bmatrix}
        x'^+_1\\
        x'^+_2\\
        x'^+_3 \\
        x'^-_1\\
        x'^-_2\\
        x'^-_3 \\
        s_1\\
        s_2\\
        s_3
    \end{bmatrix}
    & = \begin{bmatrix}
    1\\
    1\\
    1\\
    -1\\
    -1\\
    -1
\end{bmatrix}\\
0 \leq x'^j_i \leq 1& \quad\forall i\in \{1,2,3\}, \forall j\in \{+,-\}\\
s_i \geq 0 & \quad\forall i \in \{1,2,3\}
\end{align}    
\end{subequations}

Next, we compute the equitable partition of the split reformulation ~\eqref{eq_example_3}. We reorder the variables to clarify the partition.

\begin{equation*}
   \begin{blockarray}{rrr|rr|rr|rr|rrr}
    & x'^+_1 & x'^-_2 & x'^-_1 & x'^+_2 & x'^+_3 & x'^-_3 & s_1 & s_2 & s_3 & & \\
\begin{block}{c[rr|rr|rr|rr|r]r[r]}
  & 2 & -1 & -2 & 1 & 1 & -1 & 1 & 0 & 0 & & 1\\
  & -1 & 2 & 1 & -2 & -1 & 1 & 0 & 1 & 0 & & 1 \\
  \BAhline
  & -1 & -1 & 1 & 1 & 0 & 0 & 0 & 0 & 1 & & 1\\
  \BAhline
  & -2 & 1 & 2 & -1 & -1 & 1 & -1 & 0 & 0 & & -1\\
  & 1 & -2 & -1 & 2 & 1 & -1 & 0 & -1 & 0 & & -1 \\
  \BAhline
  & 1 & 1 &  -1 & -1 & 0 & 0 & 0 & 0 & -1 & & -1\\
\end{block}
\end{blockarray} 
\end{equation*}

Then, we apply~\cref{thm_lp_folding} to the given equitable partition. Doing so aggregates some of the positive and negative index variables with each other.
\begin{subequations}
\label{eq_example_4}
\begin{align}
    \min y_{1^+,2^-} - y_{1^-,2^+} & \\
    \begin{bmatrix}
        1 & -1 & 0 & 1 & 0 \\
        -1 & 1 & 0 & 0 & 1\\
        -1 & 1 & 0 & -1 & 0\\
        1 & -1 & 0 & 0 & -1\\
    \end{bmatrix} \begin{bmatrix}
        y_{1^+,2^-}\\
        y_{1^-,2^+} \\
        y_{3^+,3^-} \\
        s_{1,2}\\
        s_3
    \end{bmatrix} &= \begin{bmatrix}
        2 \\
        1\\
        -2\\
        -1
    \end{bmatrix}\\
    0 \leq y_{1^+,2^-}, y_{1^-,2^+}, y_{3^+,3^-}&\leq 2 \\
    s_{1,2}, s_3 &\geq 0\\
\end{align}    
\end{subequations}

We can notice two things about~\eqref{eq_example_4}. First of all, we note that $y_{3^+,3^-}$ vanishes from the given equations and thus any $0 \leq y_{3^+,3^-} \leq 2$ is always feasible, so we may remove it from the Linear Program. 
This corresponds to removing the reflection symmetry around the hyperplane $x_3 = 1$, which implies that $x'_{3^+}$ and $x'_{3^-}$ are symmetric to each other. Secondly, we can observe that the redundancy that we introduced by adding negated variables and rows is still present in~\eqref{eq_example_4}. By removing the duplicated rows and setting $z_{1^\pm,2^{\mp}} = y_{1^+,2^-} - y_{1^-,2^+}$, we obtain the following one-dimensional linear program. 
\begin{align*}
    \min z_{1^\pm, 2^\mp} & \\
    z_{1^\pm, 2^\mp} + s_{1,2} &= 2\\
    -z_{1^\pm, 2^\mp} + s_3 &= 1\\
    -2 \leq z_{1^\pm, 2^\mp} &\leq 2\\
    s_{1,2},s_3 &\geq 0
\end{align*}
In this final linear program, we have additionally removed the reflection symmetry around the hyperplane $x_1 + x_2 = 2$ in the original problem. Finally, we can conclude that applying the equitable partition to the split reformulation was sufficient to remove the reflection symmetries from our problem. In the next sections, we show more formally that the illustrated procedure can be used to aggregate the reflection symmetries of any linear program. 

\subsection{Centering the linear program}
Initially, we consider the linear program $F(A,b,\ell,u,c)$. In order to detect reflection symmetries, we would like to center each variable at the center of its domain. This can be achieved using an affine transformation that offsets the variables.

\begin{proposition}
\label{thm_affine_equivalent}
    For any scalar $\delta \in \R$, $x \in P(A,b,\ell,u)$ holds if and only if $x'\coloneqq x - \delta \in P(A,b- A \delta,\ell - \delta,u- \delta)$ holds. Moreover, $\transpose{c} x' = \transpose{c} x - \transpose{c} \delta$ holds.
\end{proposition}
\begin{proof}
    It is sufficient to observe that $A x = b \iff A x - A \delta = b - A\delta \iff Ax' = b - A \delta$ and that $\ell \leq x \leq u \iff \ell - \delta \leq x - \delta \leq u -\delta \iff \ell - \delta \leq x' \leq u -\delta$. Moreover, $\transpose{c} x' = \transpose{c} x - \transpose{c} \delta$ follows by definition of $x'$. 
\end{proof}

The main idea that we use \cref{thm_affine_equivalent} for in the context of equitable partitions is that that $F(A,b-A \delta, \ell-\delta,u-\delta,c)$ may admit a coarser equitable partition than $F(A,b,\ell,u,c)$. 
We note that there are infinitely many choices for $\delta$ to do so, so it is not clear how $\delta$ may be chosen. Frequently, we use \cref{thm_affine_equivalent} to center the variables at their domains by using $\delta_w \coloneqq \frac{u_w + \ell_w}{2}$. 
For variables with infinite lower or upper bounds, such a centered $\delta_w$ value is not well-defined. Later on, we will clarify how one should choose $\delta_w$ in such cases. For now, we analyze the equitable partitions of $F(A,b-A \delta, \ell-\delta,u-\delta,c)$ for any finite $\delta \in \R$.\\

Furthermore, we note that $\delta$ corresponds to an offset of the primal variables. It is also possible to expand the results in this work by using an offset for the dual variables, which corresponds to adding weighted multiples of the rows to the objective. One could hope that doing so may help to construct additional variables that have identical objective value, which would lead to potentially coarser equitable partitions. In order to simplify the presentation, we do not consider this extension here.

\subsection{Split reformulations of linear programs}

In both papers~\cite{Hojny2025-si} and \cite{Christophel2014}, the authors reduce the problem of finding reflection symmetries to finding permutation symmetries by adding both the original and the complemented variables to the auxiliary graph used for symmetry detection. For linear programs, a similar effect can be achieved by splitting each variable $x$ as $x=x^+ - x^-$, where $x^+$ represents the positive domain of $x$ and $x^-$ corresponds to the negative domain of $x$. We introduce the notion of a \emph{split reformulation} to achieve such a reformulation.
\begin{definition}
    \label{def_split_reform}
    Let $F(A,b,\ell,u,c)$ be any linear program with $A\in \R^{V\times W}$, $b\in \R^V$, $c\in \R^W$ and vectors $\ell,u$ with $\ell_w\in \R\cup \{-\infty\}$ and $u \in \R\cup \{\infty\}$ for all $w\in W$. 
    If $\ell \leq 0$ and $u \geq 0$ hold, we define the \emph{split reformulation} of $F(A,b,\ell,u,c)$ to be the following linear program with variables $x^+\in \R^W$ and $x^-\in \R^W$:
    \begin{equation}
        \begin{split}
            \min\,  \transpose{c} x^+ - \transpose{c} x^-\\
            A x^+ - A x^- &= b\\
            -A x^+ + A x^- &= -b\\
            0\leq x^+ &\leq u \\
            0 \leq x^- &\leq -\ell
        \end{split}
        \tag{$F_S(A,b,\ell,u,c))$}
    \end{equation}
    We use $F_S(A,b,\ell,u,c)$ to denote the linear program that is the split reformulation of $F(A,b,\ell,u,c)$, and we use $P_S(A,b,\ell,u)$ to denote its feasible region. 
    Furthermore, we use $\widehat{M} = \begin{bmatrix}
        A & -A\\
        -A & A
    \end{bmatrix}\in \R^{(\widehat{V}^+ \cup \widehat{V}^-) \times (\widehat{W}^+ \cup \widehat{W}^-)}$, to denote the constraint matrix and $\widehat{V} \coloneqq \widehat{V}^+ \cup \widehat{V}^-$ and $\widehat{W} \coloneqq \widehat{W}^+ \cup \widehat{W}^-$ to denote the row and column index sets.
\end{definition}

Note that the split reformulation  $F_S(A,b,\ell,u)$ includes the constraints $-Ax^+ + A x^- = -b$, indexed by $\widehat{V}_2$, which are redundant as they are directly implied by negating $Ax^+ - Ax^- = b$. Thus, they may be left out without changing the feasible region. Previous works do not consider the duplication of constraints. However, we will see that these redundant equations may be useful for the derivation of coarse equitable partitions. In particular, they may help to establish equitable partitions that arise if a subset of the rows are multiplied by $-1$.\\

In \cref{thm_splitting_equivalent}, we show that the split reformulation of any linear program is equivalent to its original linear program.
\begin{lemma}
    \label{thm_splitting_equivalent}
    Let $A\in \R^{V\times W}$, $b\in \R^V$, $c\in \R^W$ and vectors $\ell,u,$ with $l_w \in \R_{\leq0} \cup \{-\infty\}$ and $u_w\in \R_{\geq 0} \cup \{\infty\}$ for all $w\in W$. Then, the following hold:
    \begin{enumerate}[label=(\roman*)]
        \item \label{thm_splitting_equivalent_to_red} If $x\in P(A,b,\ell,u)$, then for $x^+\coloneqq \max(x,\zerovec)$ and $x^-\coloneqq \max(-x,\zerovec)$ it holds that $(x^+,x^-) \in P_S(A,b,\ell,u)$ and $\transpose{c} x = \transpose{c}x^+ - \transpose{c} x^-$.
        \item \label{thm_splitting_equivalent_from_red} If $(x^+,x^-)\in P_S(A,b,\ell,u)$ then for $x\coloneqq x^+ - x^-$ it holds that $x\in P(A,b,\ell,u)$ and $\transpose{c} x = \transpose{c}x^+ - \transpose{c} x^-$.  
    \end{enumerate}
\end{lemma}
\begin{proof}
To prove \ref{thm_splitting_equivalent_to_red}, let $x\in P(A,b,\ell,u)$ be given and consider $x^+=\max(x,\zerovec)$ and $x^-=\max(-x,\zerovec)$. Then, we have:
\begin{equation*}
 Ax^+ - Ax^- = A(x^+-x^-) = A(\max(x,\zerovec)- \max(-x,\zerovec)) = A x \stackrel{\textup{(a)}}{=} b,   
\end{equation*} where \textup{(a)} follows since $x\in P(A,b,\ell,u)$ holds. It follows directly that also $-Ax^+ + Ax^- = -(Ax^+ - Ax^-) = -b$ holds.
Furthermore, note that $x^+ \geq 0$ and $x^-\geq 0$ hold by definition of the $\max$ operator. The upper bounds $x^+\leq u$ and $x^- \leq -\ell$ follow from $x\leq u$ and $-x \leq - \ell$ which are satisfied since $x\in P(A,b,\ell,u)$ holds. Then, since $\ell \leq 0 $ and $u\geq 0$, it follows that $(x^+,x^-)$ is feasible for $ P_S(A,b,\ell,u)$. Finally, note that $\transpose{c} x = \transpose{c} (\max(x,\zerovec)- \max(-x,\zerovec)) = \transpose{c} x^+ - \transpose{c} x^-$ shows that the objectives are equal.

To prove \ref{thm_splitting_equivalent_from_red}, let $(x^+,x^-)\in P_S(A,b,\ell,u)$ be given and consider $x\coloneqq x^+-x^-$. Then, $Ax = A x^+ - Ax^- = b$ holds since $(x^+,x^-) \in P_S(A,b,\ell,u)$. Since $x^+\geq 0$ and $x^-\geq 0$ hold in $P_S(A,b,\ell,u)$, we also have that $x = x^+ - x^- \leq u + \zerovec = u$ and $x = x^+ - x^- \geq \zerovec - -\ell = \ell$, which shows that $\ell \leq x \leq u$ holds, and thus it follows that $x\in P(A,b,\ell,u)$. Finally, we have by definition of $x$ that $\transpose{c} x = \transpose{c} (x^+ - x^-) = \transpose{c} x^+ - \transpose{c} x^-$. 
\end{proof}

 First of all, we note that $\ell \leq 0$ and $u \geq 0$ are necessary conditions to guarantee that the split reformulation exists in the proof of \cref{thm_splitting_equivalent}. For any linear programs with $\ell_w \leq u_w$ for variable $x_w$, one can achieve this by transforming $x_w$ into $x'_w$ using $x_w' = x_w + \delta_w$ for any $\delta_w$ that satisfies $\ell_w \leq \delta_w \leq u_w$. If $\ell_w \leq u_w$ does not hold, then the linear program is clearly infeasible.\\
 
For the detection of reflection symmetries, previous approaches detect reflection symmetries by detecting permutation symmetries on a mixed-integer linear program obtained from a split reformulation.
Thus, a logical next step is to consider the application of DRCR to a split reformulation $F_S(A,b,\ell,u,c)$ for any linear program $F(A,b,\ell,u,c)$ with $\ell \leq \zerovec$ and $u\geq \zerovec$. To do so, we first need to understand some of the special properties of equitable partitions of split reformulations.

\subsection{Equitable partitions of split reformulations}
Throughout this section, we investigate the structure of equitable partitions of the split reformulation $F_S(A,b,\ell,u,c)$ for an arbitrary linear program $F(A,b,\ell,u,c)$ with $\ell \leq \zerovec$ and $u\geq \zerovec$.

Recall that for a split reformulation, its variable index sets $\widehat{W} = \widehat{W}^+ \cup \widehat{W}^-$ and $\widehat{V} = \widehat{V}^+ \cup \widehat{V}^-$ correspond to the original and negated variables and constraints, respectively. For $x\in \{-1,1\}$, we will use the synonyms $\widehat{V}^x\coloneqq \begin{cases}
    \widehat{V}^+ \text{ if $x= 1$}\\
    \widehat{V}^- \text{ if $x=-1$}\\
\end{cases}$ and $\widehat{W}^x\coloneqq  \begin{cases}
    \widehat{W}^+ \text{ if $x= 1$}\\
    \widehat{W}^- \text{ if $x=-1$}\\
\end{cases}$ to denote the corresponding signed sets and for $v\in V$ ($w\in W$) we similarly use $v^x$ ($w^x$) to indicate that $v$ lies in $\widehat{V}^x$ ($\widehat{W}^x$).
For any $P\subseteq V$ and  $\gamma \in \{-1,1\}^{P}$, and, we use $P^\gamma \coloneqq \{ v^{\gamma_v} \in \widehat{V}^{\gamma_v} \mid v\in P \}$ to denote the signed set of indices corresponding in $\widehat{V}^+$ (if $\gamma_v = 1)$ and $\widehat{V}^-$ (if $\gamma_v = -1$).
For example, for $P=\{v_1,v_2,v_3\}$ and $\gamma \coloneqq \begin{bmatrix}
    \gamma_{v_1} & \gamma_{v_2} & \gamma_{v_3}
\end{bmatrix} = \begin{bmatrix}
    1 & -1 & -1
\end{bmatrix}$, we would have that $P^\gamma = \{v^+_1, v^-_2,v^-_3\}$ and that $P^{-\gamma} = \{v^-_1,v^+_2,v^+_3\}$, where $v^+_1,v^+_2,v^+_3 \in \widehat{V}^+_2$ and $v^-_1,v^-_2,v^-_3\in \widehat{V}^-_2$. Similarly, for any $Q\subseteq W$ and $\lambda \in \{-1,1\}^{Q}$, we define $Q^\lambda \coloneqq \{w^{\lambda_w} \in \widehat{W}^{\lambda_w} \mid w\in Q\}$.\\

Then, we consider the structure of equitable partitions $(\widehat{\MP},\widehat{\MQ})$ of $F_S(A,b,\ell,u,c)$ in further detail. 
These equitable partitions have many symmetries that are a result of the redundancies introduced by the split reformulation. One property that will be important in this context, is the \emph{polarity} of parts $\widehat{P}\in \widehat{\MP}$. If for some $v\in V$, both $v^+\in \widehat{P}$ and $v^-\in \widehat{P}$ hold then $\widehat{P}$ is said to be \emph{bipolar}. If $\widehat{P}\subseteq \widehat{V}$ is not bipolar, then $|\widehat{P}\cap \{v^+,v^-\}|\leq 1$ holds for all $v\in V$ and we say it is \emph{unipolar}. Note that for each unipolar part $\widehat{P}$ there exists a part $P\subseteq V$ and $\gamma\in\{-1,1\}^V$ such that $\widehat{P} = P^\gamma$ holds. 
Similarly, $\widehat{Q}\subseteq \widehat{W}$ is \emph{bipolar} if there exists some $w\in W$ such that $w^+\in \widehat{Q}$ and $w^-\in Q$ hold, and $\widehat{Q}$ is said to be \emph{unipolar} if $|\widehat{Q}\cap \{w^+,w^-\}|\leq 1$ holds for all $w\in W$. If $\widehat{Q}$ is unipolar, then there exists $Q\subseteq W$ and $\lambda\in \{-1,1\}^W$ such that $\widehat{Q} = Q^\lambda$ holds. In \cref{def_symmetric}, we define the symmetric structure of equitable partitions of the split reformulation $F_S(A,b,\ell,u,c)$ of any linear program $F(A,b,\ell,u,c)$ with row and column indices $V$ and $W$, respectively. The sets $\widehat{V}$ and $\widehat{W}$ are the row and column indices of $F_S(A,b,\ell,u,c)$ as defined in \cref{def_split_reform}.

\begin{definition}
    \label{def_symmetric}
    Let $(\widehat{\MP},\widehat{\MQ})$ be a partition of $\widehat{V}\times \widehat{W}$. We say that $(\widehat{\MP},\widehat{\MQ})$ is \emph{symmetric} if the following hold.

    \begin{enumerate}[label=(\roman*)]
        \item for every unipolar part $P\in \widehat{\MP}$ such that $P =P^\gamma $ holds, $P^{-\gamma} \in \widehat{\MP}$ holds too. \label{def_symmetric_row_unipolar}
        \item for every bipolar part $P\in \widehat{\MP}$ and any $v\in V$, $P$ either contains both $v^+$ and $v^-$ or none. \label{def_symmetric_row_bipolar} 
        \item for every unipolar part $Q\in \widehat{\MQ}$ such that $Q =Q^\lambda$ holds, $Q^{-\lambda} \in \widehat{\MQ}$ holds too. \label{def_symmetric_col_unipolar}
        \item for every bipolar part $Q\in \widehat{\MQ}$ and any $w\in W$, $Q$ either contains both $w^+$ and $w^-$ or none. \label{def_symmetric_col_bipolar}
    \end{enumerate}
\end{definition}
For a symmetric equitable partition $(\widehat{\MP},\widehat{\MQ})$ of $F_S(A,b,\ell,u,c)$, we use $\widehat{\MP}_B \coloneqq \{P \in \widehat{\MP} \mid P \text{ is bipolar}\}$ and $\widehat{\MP}_U \coloneqq \{P\in \widehat{\MP} \mid P \text{ is unipolar}\}$ and note that $\widehat{\MP}_U$ and $\widehat{\MP}_B$ are a disjoint partition of $\widehat{\MP}$. 
Similarly, we define that $\widehat{\MQ}_B \coloneqq \{Q \in \widehat{\MQ} \mid Q \text{ is bipolar}\}$ and $\widehat{\MQ}_U \coloneqq \{Q \in\widehat{\MQ} \mid Q \text{ is unipolar}\}$, which also partitions $\widehat{Q}$ . Furthermore, we use $\widehat{V}_U \coloneqq \bigcup_{P\in \widehat{\MP}_U} P$, $\widehat{V}_B \coloneqq \bigcup_{P\in \widehat{\MP}_B} P$, $\widehat{W}_U \coloneqq \bigcup_{Q\in \widehat{\MQ}_U} Q$ and $\widehat{W}_B \coloneqq \bigcup_{Q\in \widehat{\MQ}_B} Q$ for the corresponding partitioned subsets of $\widehat{V}$ and $\widehat{W}$. To denote the corresponding sets in $V$ and $W$ we use $V_U \coloneqq \{v \in V \mid \{v^+,v^-\}\cap \widehat{V}_U\neq\emptyset\}$, $V_B \coloneqq \{v \in V \mid \{v^+,v^-\}\subseteq \widehat{V}_B\}$, $W_U \coloneqq \{w \in W \mid \{w^+,w^-\}\cap \widehat{W}_U\neq\emptyset\}$ and $W_B \coloneqq \{w \in W \mid \{w^+,w^-\}\subseteq\widehat{W}_B\}$ .
    
The notion of bipolar and unipolar parts in \cref{def_symmetric} is very similar to an observation by B{\"o}di, Herr and Joswig~\cite{Bodi2013}, who show that the orbits of reflection symmetry groups are either unipolar or bipolar, where bipolar orbits must consist of only pairs $w^+,w^-$ for some subset of $W$. Similarly, we will show that the equitable partitions of $F_S(A,b,\ell,u,c)$, that the associated equitable partitions are indeed symmetric and satisfy \cref{def_symmetric}. \\

Next, we show one important property of symmetric partitions: bipolar parts correspond to zero row and column sums.
\begin{lemma}
    \label{thm_bipolar_symmetric_sum}
    Let $(\widehat{\MP},\widehat{\MQ})$ be any partition of $F_S(A,b,\ell,u,c)$ and let $(P',Q') \in \widehat{\MP} \times \widehat{\MQ}$. Then, the following hold:
    \begin{enumerate}[label=(\roman*)]
        \item \label{thm_bipolar_symmetric_sum_column} If $P' = P^\onevec \cup P^{-\onevec}$ holds for some set $P\subseteq V$, then $\sum_{v\in P'} \widehat{M}_{v,w} = 0$ holds for all $w\in W$. 
        \item \label{thm_bipolar_symmetric_sum_row} If $Q' = Q^\onevec \cup Q^{-\onevec}$ holds for some set $Q\subseteq W$, then $\sum_{w\in Q'} \widehat{M}_{v,w} = 0$ holds for all $v\in V$.
    \end{enumerate}
\end{lemma}
\begin{proof}
    To prove~\ref{thm_bipolar_symmetric_sum_column}, consider any $w\in W$. Then, we have that
    \begin{equation*}
     \sum_{v\in P'} \widehat{M}_{v,w} = \sum_{v^+\in P^{\onevec}} \widehat{M}_{v^+,w} + \sum_{v^-\in P^{-\onevec}} \widehat{M}_{v^-,w} = \sum_{v\in P} A_{v,w} + \sum_{v\in P} -A_{v,w} = 0,    
    \end{equation*}
    where the first equality is given by the condition. 
    The second equality can be observed through case distinction whether $w\in W^+$ or $w\in W^-$ holds. If $w\in W^+$ holds then by definition of $\widehat{M}$ we have $\widehat{M}_{v^+,w} = A_{v,w}$ and $\widehat{M}_{v^-,w} = -A_{v,w}$. For $w\in W^-$ we have  $\widehat{M}_{v^+,w} = -A_{v,w}$ and $\widehat{M}_{v^-,w} = A_{v,w}$. In both cases, the equality follows.
    
    The proof for~\ref{thm_bipolar_symmetric_sum_row} is similar, where we have for all $v\in V$ that:
    \begin{equation*}
     \sum_{w\in Q'} \widehat{M}_{v,w} = \sum_{w^+\in Q^{\onevec}} \widehat{M}_{v,w^+} + \sum_{w^-\in Q^{-\onevec}} \widehat{M}_{v,w^-} = \sum_{w\in Q} A_{v,w} + \sum_{w\in Q} -A_{v,w} = 0.    
    \end{equation*}
\end{proof}
\begin{corollary}
    \label{thm_bipolar_symmetric_equitable_sum}
    Let $(\widehat{\MP},\widehat{\MQ})$ be any symmetric equitable partition of $F_S(A,b,\ell,u,c)$ and let $(P,Q) \in (\widehat{\MP},\widehat{\MQ})$. If at least one of $P$ and $Q$ is bipolar, then $\sum_{v\in P} \widehat{M}_{v,w} = 0$ holds for all $w\in Q$ and $\sum_{w\in Q} \widehat{M}_{v,w} = 0$ holds for all $v\in P$. 
\end{corollary}
\begin{proof}
    We have $\sum_{v\in P} \sum_{w\in Q} \widehat{M}_{v,w} = \sum_{w\in Q} \sum_{v\in P} \widehat{M}_{v,w} = 0$,
    where the final equality holds by \cref{thm_bipolar_symmetric_sum}, since bipolarity of either $Q$ or $P$ implies that $\sum_{w\in Q} \widehat{M}_{v,w} = 0$ holds for all $v\in P$ or $\sum_{v\in P} \widehat{M}_{v,w} = 0$ holds for all $w\in Q$. By equitability of the block $(P,Q)$, the row and column sums are identical, and it follows that $\sum_{w\in Q} \widehat{M}_{v,w} = \frac{1}{|P|} \sum_{v'\in P} \sum_{w\in Q} \widehat{M}_{v',w} = 0$ holds for all $v\in P$ and $\sum_{v\in P} \widehat{M}_{v,w} = \frac{1}{|Q|} \sum_{w'\in Q} \sum_{v\in P} \widehat{M}_{v,w'} = 0$ holds for all $w'\in Q$.
\end{proof}
In order to investigate symmetric equitable partitions of split reformulations, we need a key assumption that limits the structure of the equitable partitions for the split reformulation. One disadvantage of the split reformulation is that the $x^+$- and $x^-$-variables may interact in non-identical ways, which makes it tough to interpret the computed partition of $\widehat{V}$ and $\widehat{W}$ in the context of the original linear program $F(A,b,\ell,u,c)$. For example, for variables $w_1,w_2\in W$, it may be the case that there exists a part $\{w_1^+,w_1^-,w_2^+\}$ which acts identically on the whole domain of $w_1$ but only on the positive domain of $w_2$. 
To avoid such issues, we consider only partitions of the split reformulation such that for each part $Q\in \widehat{\MQ}$, the associated variables in $Q$ of the linear program $F(A,b,\ell,u,c)$ have identical variable bounds. 
\begin{definition}
    \label{def_part_domain_respecting}
    A partition $(\widehat{\MP},\widehat{\MQ})$ of $\widehat{V} \times \widehat{W}$ is said to be \emph{bound-respecting} if for all $Q\in \widehat{\MQ}$ and any $w^{s_1}_1,w^{s_2}_2\in Q$ we have:
    \begin{enumerate}[label=(\roman*)]
        \item $\ell_{w_1} = \ell_{w_2}$ and $u_{w_1} = u_{w_2}$ if $s_1 = s_2$ holds
        \item $\ell_{w_1} = -u_{w_2}$ and $u_{w_1} = - \ell_{w_2}$ if $s_1 \neq s_2$ holds
    \end{enumerate}
\end{definition}

Then, we show that the split reformulation $F_S(A,b,\ell,u,c)$ admits a symmetric equitable partition in two steps. First, we show in \cref{thm_initial_partition_symmetric} that the initial coarsest bound-respecting equitable partition is symmetric. Second, we show in \cref{thm_refinement_symmetric} that any refinements of the initial partition from the constraint matrix preserve the symmetry of the partition. 

\begin{lemma}
\label{thm_initial_partition_symmetric}
Let $(\widehat{\MP},\widehat{\MQ})$ be the coarsest bound-respecting partition of $F_S(A,b,\ell,u,c)$ such that $\widehat{\MP}$ is equitable for $\widehat{b} \coloneqq \begin{bmatrix}
    b & - b
\end{bmatrix}$ and $\widehat{\MQ}$ is equitable for the $\widehat{c} \coloneqq \transpose{\begin{bmatrix}
    c & -c
\end{bmatrix}}$ and $\widehat{u} \coloneqq \transpose{\begin{bmatrix}
    u & -\ell
\end{bmatrix}}$.
Then, $(\widehat{\MP},\widehat{\MQ})$ is symmetric. Moreover, $b_v = 0$ holds for all $v\in \widehat{V}_B$ and $c_w =0$ and $u_w = -\ell_w$ hold for all $w\in W_B$.
\end{lemma}
\begin{proof}
First, consider the coarsest equitable partition $\widehat{\MP}$ of the right-hand side $\widehat{b} \in \R^{\widehat{V}^+\cup \widehat{V}^-}$. Consider any bipolar part $P \in \widehat{\MP}$ such that for some $v \in V$, $\{v^+,v^-\} \subseteq P$ holds. Then, since $\widehat{\MP}$ is an equitable partition of $\widehat{b}$, it follows that $b_v = \widehat{b}_{v^+} = \widehat{b}_{v^-} = - b_v$ holds, which implies that $b_v = 0$. As this holds for all $v\in V$ with $b_{v} = 0$, it follows that the block $P$ in the coarsest equitable partition is exactly formed by those rows of $b$ that have right-hand side $0$, i.e. $P = \{ v^+,v^- \mid \text{ such that $b_v = 0$ for $v\in V$} \}$ is the only bipolar part in $\widehat{\MP}$. Clearly, $P$ satisfies \cref{def_symmetric}\ref{def_symmetric_row_bipolar}. 

For any unipolar part $\widehat{P}\in \widehat{\MP}$ holds such that $\widehat{P} = P^\gamma$ for suitable $P \subseteq V$ and $\gamma \in \{-1,1\}^V$, it follows from equitability of $\widehat{\MP}$ for $\widehat{b}$ that for all $v^{\gamma}_1,v^\gamma_2\in \widehat{P}$ that $\widehat{b}_{v^\gamma_1} = \gamma_{v_1} b_{v_1} = \gamma_{v_2} b_{v_2} = \widehat{b}_{v^\gamma_2}$ holds. However, this directly implies for the set $P^{-\gamma} \subseteq \widehat{V}$ that $\widehat{b}_{v^{-\gamma}_1} = -\gamma_{v_1} b_{v_1} = -\gamma_{v_2} b_{v_2} = \widehat{b}_{v^{-\gamma}_1}$ holds, which shows that $P^{-\gamma}$ is also an equitable block. Thus, $\widehat{\MP}$ can only be the coarsest equitable partition of $\widehat{b}$ if there exists some block $\widehat{P}'$ such that $P^{-\gamma} \subseteq \widehat{P}'$. One can show using a similar argument that $P^{-\gamma} = \widehat{P}'$ holds: existence of any $v^{-\gamma}\in \widehat{P}' \setminus P^{-\gamma}$ contradicts that $\widehat{\MP}$ is the coarsest equitable partition of $\widehat{b}$ as then the part containing $v^\gamma$ can be merged with $\widehat{P}$ to form a coarser partition. Since the above reasoning holds for all unipolar blocks, it follows that \cref{def_symmetric}\ref{def_symmetric_row_unipolar} is satisfied.

Second, we consider the coarsest bound-respecting equitable partition $\widehat{\MQ}$ of $\widehat{c}$ and $\widehat{u}$. Then, for any $Q\in \widehat{\MQ}$ we have for $\widehat{w}_1, \widehat{w}_2 \in Q$ that $\widehat{c}_{\widehat{w}_1} = \widehat{c}_{\widehat{w}_2}$ and $\widehat{u}_{\widehat{w}_1} = \widehat{u}_{\widehat{w}_2}$ hold.
First, consider the case where $Q\in \widehat{\MQ}$ is bipolar, and let $w_1\in W$ be some index such that $\{w_1^+,w_1^-\}\subseteq Q$ holds. Then, it follows from equitability for $\widehat{c}$ that $c_{w_1} = \widehat{c}_{w_1^+} = \widehat{c}_{w_1^-} = -c_{w_1}$ holds, which implies that $c_{w_1} = 0$. Furthermore, equitability for $\widehat{u}$ implies that $u_{w_1} = \widehat{u}_{w_1^+} = \widehat{u}_{w_1^-} = -\ell_{w_1}$ holds. Then, assume for the sake of contradiction that there exists any $w_2\in W$ such that $|\{w_2^+,w_2^-\} \cap Q| = 1$ holds. 
If $w_2^+\in Q$ holds, then $\widehat{u}_{w^-_2} = - \ell_{w_2} = - \ell_{w_1} = \widehat{u}_{w^-_1} =  \widehat{u}_{w^+_2}$ shows that $w_2^-$ has the same bound for $\widehat{u}$, where the second equality follows since $\widehat{\MQ}$ is a bound-respecting partition and $w_2^+$ and $w_1^+$ are both included in $Q$. Similarly, if $w_2^- \in Q$ holds, then $\widehat{u}_{w^+_2} = u_{w_2} = u_{w_1} = \widehat{u}_{w^+_1} = \widehat{u}_{w^-_2}$ holds since $\widehat{\MQ}$ is a bound-respecting partition and $Q$ contains both $w^-_1$ and $w^-_2$. In either case, it follows that $\widehat{u}_{w_2^+} = \widehat{u}_{w_2^-}$ holds. As one of $w_2^+$ and $w_2^-$ is in $Q$, it follows that $\widehat{c}_{w^+_2} = \widehat{c}_{w^-_2} = c_{w_2} = 0$ also holds, which contradicts that $\widehat{\MQ}$ is the coarsest bound-respecting partition as then both $w_2^+$ and $w_2^-$ may be included in $Q$. Thus, bipolar parts must always contain pairs of variables and  \cref{def_symmetric}\ref{def_symmetric_col_bipolar} follows.

Finally, consider the case where $\widehat{Q}\in \widehat{\MQ}$ is unipolar and given by $\widehat{Q} = Q^{\lambda}$ with $Q\subseteq W$ and $\lambda \in \{-1,1\}^W$. By equitability of $\widehat{\MQ}$, we must have for all  $w_1^\lambda,w_2^\lambda \in Q^\lambda$ that $\widehat{c}_{w_1^\lambda} = \widehat{c}_{w_2^\lambda}$ and $\widehat{u}_{w_1^\lambda} = \widehat{u}_{w_2^\lambda}$ hold. Then, consider $Q^{-\lambda}\subseteq \widehat{W}$. First of all, note that for any $w_1^{-\lambda},w_2^{-\lambda} \in Q^{-\lambda}$ that we have: $\widehat{c}_{w_1^{-\lambda}} = -\widehat{c}_{w_1^{\lambda}} = -\widehat{c}_{w_2^{\lambda}} =\widehat{c}_{w_2^{-\lambda}}$, where the first and last equalities follow by definition of $\widehat{c}$ and the second equality follows since $w_1^\lambda$ and $w_2^\lambda$ lie in the equitable part $Q^\lambda$. Now, let us show that for the bounds $\widehat{u}_{w_1^{-\lambda}} = \widehat{u}_{w_2^{-\lambda}}$ holds too:
\begin{align*}
    \widehat{u}_{w_1^{-\lambda}} &= \begin{cases}
        -\ell_{w_1} &\text{ if $\lambda_{w_1} =1$}\\
        u_{w_1} &\text{ if $\lambda_{w_1} = -1$}
    \end{cases} = \begin{cases}
        -\ell_{w_1} &\text{ if $\lambda_{w_1} =1$ and $\lambda_{w_2} = 1$}\\
        -\ell_{w_1} &\text{ if $\lambda_{w_1} =1$ and $\lambda_{w_2} = -1$}\\
        u_{w_1} &\text{ if $\lambda_{w_1} = -1$ and $\lambda_{w_2} = 1$} \\
        u_{w_1} &\text{ if $\lambda_{w_1} = -1$ and $\lambda_{w_2} = -1$}
    \end{cases}\\
    &\stackrel{\textup{(a)}}{=} \begin{cases}
        -\ell_{w_2} &\text{ if $\lambda_{w_1} =1$ and $\lambda_{w_2} = 1$}\\
        u_{w_2} &\text{ if $\lambda_{w_1} =1$ and $\lambda_{w_2} = -1$}\\
        -\ell_{w_2} &\text{ if $\lambda_{w_1} = -1$ and $\lambda_{w_2} = 1$} \\
        u_{w_2} &\text{ if $\lambda_{w_1} = -1$ and $\lambda_{w_2} = -1$}
    \end{cases} =
    \begin{cases}
        -\ell_{w_2} &\text{ if $\lambda_{w_2} =1$}\\
        u_{w_2} &\text{ if $\lambda_{w_2} = -1$}
    \end{cases} = \widehat{u}_{w_2^{-\lambda}}.
\end{align*}
Here, \textup{(a)} follows since $\widehat{\MQ}$ is a bound-respecting partition and $w_1^\lambda$ and $w_2^\lambda$ are in $Q\in \widehat{\MQ}$. Thus $Q^{-\lambda}\subseteq Q'$ must hold for some $Q'\in \widehat{\MQ}$. Using a similar argument, one can show that existence of any $w^{-\lambda}_x \in Q' \setminus Q^{-\lambda}$ shows that the part containing $w^{\lambda}_x$ can be merged with $Q^\lambda$, which contradicts that $\widehat{\MQ}$ is the coarsest partition. Thus $Q^{-\lambda}$ belongs to $\widehat{\MQ}$, which shows that \cref{def_symmetric}\ref{def_symmetric_col_unipolar} is satisfied. 
\end{proof}

Note in \cref{thm_initial_partition_symmetric} that the lower bounds of variables $x^+$ and $x^-$ are not considered as they are all zero and thus do not affect the equitable partition of $\widehat{W}$. Next, let us show that refinement of symmetric partitions through the constraint matrix again yields a symmetric partition.

\begin{lemma}
\label{thm_refinement_symmetric}
    Consider the split reformulation $F_S(A,b,\ell,u,c)$ of a linear program, where the constraint matrix of $F_S(A,b,\ell,u,c)$ is given by $\widehat{M} = \begin{bmatrix}
        A & - A\\
        -A & A
    \end{bmatrix}$, and let $(\widehat{\MP}_0,\widehat{\MQ}_0)$ be a symmetric partition of $F(A,b,\ell,u,c)$. Then, the coarsest equitable partition $(\widehat{\MP}_\infty,\widehat{\MQ}_\infty)$ that refines $(\widehat{\MP}_0,\widehat{\MQ}_0)$ with respect to $\widehat{M}$ is symmetric.
\end{lemma}
\begin{proof}

    By Theorem 2 in~\cite{Paige1987}, the coarsest equitable partition is unique. Thus, it suffices to show that there exists a sequence of refinements such that the coarsest equitable partition is symmetric. 
    We prove the statement by showing that there exists a refinement algorithm that maintains a symmetric partition at every step. In particular, we assume that $(\widehat{\MP}_i,\widehat{\MQ}_i)$ is symmetric, and show that existence of any refinement of $(\widehat{\MP}_i,\widehat{\MQ}_i)$ implies that there exists a refinement into a partition $(\widehat{\MP}_{i+1},\widehat{\MQ}_{i+1})$ that is symmetric. The result then follows inductively, where the base case is satisfied by assumption.

    Then, assume that $(\widehat{\MP}_i,\widehat{\MQ}_i)$ is a symmetric partition of $F_S(A,b,\ell,u,c)$ such that there exists some $(P',Q') \in (\widehat{\MP}_i,\widehat{\MQ}_i)$ such that $\widehat{M}_{P',Q'}$ is not an equitable block. 
    We consider the case where $\widehat{M}_{P',Q'}$ does not satisfy \eqref{eq_folding_row_sums} for some $v,v'\in P'$ such that $\sum_{w\in Q'} \widehat{M}_{v,w} \neq \sum_{w\in Q'} \widehat{M}_{v',w}$, and show that then $\widehat{\MP}_i$ can be refined symmetrically so that $(\widehat{\MP}_{i+1},\widehat{\MQ}_i)$ is a symmetric partition that refines $(\widehat{\MP}_i,\widehat{\MQ}_i)$. The argumentation for the case where $\widehat{M}_{P',Q'}$ does not satisfy \eqref{eq_folding_col_sums} is identical up to transposition. 

First, note that if $Q'$ is bipolar, that then it follows by symmetry of $\widehat{P}_i$ and \cref{thm_bipolar_symmetric_sum} that $\sum_{v,w} \widehat{M}_{v,w} = 0$ holds for all $v\in P'$. Hence~\eqref{eq_folding_row_sums} is satisfied in this case, and it is sufficient to consider the case where $Q'$ is unipolar in the following. We let $Q' = Q^\lambda$ for some $Q\subseteq W$ and $\lambda\in \{-1,1\}^W$ and use $Q^+ \coloneqq \{w\in Q \mid w^+\in Q^\lambda\}$ and $Q^- \coloneqq \{w\in Q \mid w^-\in Q^\lambda\}$ to denote the positive and negative variables in $Q$, respectively.
Next, we show that refinement of the block $\widehat{M}_{P',Q'}$ always produces a symmetric partition. To do so, we distinguish the cases whether $P'$ is bipolar or unipolar.\\
\noindent\textbf{Case 1: $P'$ is bipolar.}\\
    If $P'$ is bipolar, then there exists $P\subseteq V$ such that $P' = P^\onevec \cup P^{-\onevec}$. Then, since $Q'$ is unipolar, we have that:
        \begin{equation*}
            \widehat{M}_{P',Q'} = \widehat{M}_{P',Q^\lambda} = \begin{bmatrix}
            \widehat{M}_{P^\onevec,Q^\lambda \cap W^+_2} & \widehat{M}_{P^\onevec,Q^\lambda \cap W^-_2} \\
            \widehat{M}_{P^{-\onevec},Q^\lambda \cap W^+_2} & \widehat{M}_{P^{-\onevec},Q^\lambda \cap W^-_2} 
        \end{bmatrix} = \begin{bmatrix}
            A_{P,Q^+} & -A_{P,Q^-} \\
            -A_{P,Q^+} & A_{P,Q^-} 
        \end{bmatrix} = \begin{bmatrix}
            G \\
            -G
        \end{bmatrix},
        \end{equation*}
        where $G\coloneqq \begin{bmatrix}
            A_{P,Q^+} & -A_{P,Q^-}
        \end{bmatrix}$.
        Then, consider the refinement of $\widehat{\MP}_i$ into $\widehat{\MP}_{i+1}$ obtained by splitting $P'$ into $k$ different classes based on the different row sums of $\widehat{M}_{P',Q'}$. Let $P'_j$ for $j=1,\cdots,k$ denote these parts, where we have for all $v,v'\in P'_j$ that $\sum_{w\in Q'} \widehat{M}_{v,w} = \sum_{w\in Q'} \widehat{M}_{v',w}$, and define $\widehat{\MP}_{i+1} \coloneqq (\widehat{\MP}_i \setminus \{P'\} )\cup \bigcup_{j=1}^k P'_j$. We show that \ref{def_symmetric_row_unipolar} and \ref{def_symmetric_row_bipolar} in \cref{def_symmetric} are satisfied for all $P'_j$, which is sufficient since they are satisfied by assumption for all $P''\in \widehat{\MP}_i \setminus \{P'\}$ .

        First, consider the case where $P'_j$ is bipolar. Since $P'_j$ contains at least one pair $v^+,v^-\in P'_j$ for some $v\in V_2$, we can deduce that $\sum_{w\in Q'} G_{v,w} = \sum_{w\in Q'} \widehat{M}_{v,w} = \sum_{w\in Q'} \widehat{M}_{v',w} = \sum_{w\in Q'} -G_{v,w} = 0$ holds. Then, note that for any $\gamma\in \{-1,1\}$ such that $v^\gamma \in P'_j$, the rows of $v^\gamma$ and $v^{-\gamma}$ both sum to zero, which implies that also $v^{-\gamma}\in P'_j$ holds by definition of $P'_j$, where we use the assumption that \cref{def_symmetric}\ref{def_symmetric_row_bipolar} holds for $\widehat{\MP}_i$ to infer that $v^{-\gamma} \in P'$ holds. Thus, this shows that \cref{def_symmetric}\ref{def_symmetric_row_bipolar} is satisfied by $\widehat{\MP}_{i+1}$ in the case where $P'$ is bipolar. 

        If $P'_j$ is unipolar, then there exists some $\gamma\in \{-1,1\}^V$ and $P\subseteq V$ such that $P'_j = P^\gamma$, and we have for all $v\in P'_j$ that $\sum_{w\in Q'} \widehat{M}_{v,w} = \gamma_v \sum_{w\in Q'} G_{v,w} = \alpha $ for some $\alpha \in \R\setminus\{0\}$, where we can exclude $\alpha = 0$ because this corresponds to the bipolar parts treated above. Since $\MP_i'$ satisfies \cref{def_symmetric}\ref{def_symmetric_row_bipolar} by assumption, all elements of $P^{-\gamma}$ are also contained in $P'$, and must be partitioned according to their sums. Then we have for all $v\in P^{-\gamma}$ that $\sum_{w\in Q'} \widehat{M}_{v,w} = -\gamma_v \sum_{w\in Q'} G_{v,w} = -\alpha$, which shows that there exists some $j'$ such that $P'_{j'} = P^{-\gamma}$. Thus, we have shown that \cref{def_symmetric}\ref{def_symmetric_row_unipolar} holds for $\widehat{\MP}_{i+1}$ in the case where $P'$ is bipolar. Thus, we have shown that $(\widehat{\MP}_{i+1}, \widehat{\MQ}_i)$ is symmetric and refines $(\widehat{\MP}_i,\widehat{\MQ}_i)$.\\
        
    \noindent\textbf{Case 2: $P'$ is unipolar.}\\
    If $P'$ is unipolar, then there exists $P\subseteq V$ and $\gamma\in \{-1,1\}^V$ such that $P' = P^\gamma$ holds. By the assumption that $\widehat{\MP}_i$ is symmetric, it also holds that $P^{-\gamma} \in \widehat{\MP}_i$. Then, we show that for $P^\gamma$ and $P^{-\gamma}$ the corresponding blocks $\widehat{M}_{P^\gamma,Q'}$ and $\widehat{M}_{P^{-\gamma},Q'}$ can be simultaneously refined to preserve symmetry in $\widehat{\MP}_{i+1}$. For $x\in \{\gamma,-\gamma\}$, let $P^{x}_+ \coloneqq \{v\in P \mid v^+ \in P^x\}$ and $P^{x}_- \coloneqq \{ v\in P \mid v^- \in P^{x}\}$ denote the positive and negative variables associated to $P^x$, respectively. Note that $P^{\gamma}_+ = P^{-\gamma}_-$ and $P^{\gamma}_- = P^{-\gamma}_+$ hold. 
    Then, for the considered blocks $P^\gamma$ and $P^{-\gamma}$, we have that:
    \begin{equation*}
    \begin{bmatrix}
        \widehat{M}_{P^\gamma,Q'} \\
        \widehat{M}_{P^{-\gamma},Q'} 
    \end{bmatrix} = \begin{bmatrix}
        \widehat{M}_{P^\gamma,Q^\lambda} \\
        \widehat{M}_{P^{-\gamma},Q^\lambda} 
    \end{bmatrix} = 
    \begin{bmatrix}
        A_{P^\gamma_+,Q^+} & -A_{P^\gamma_+,Q^-} \\
        -A_{P^\gamma_-,Q^+} & A_{P^\gamma_-,Q^-} \\
        -A_{P^{-\gamma}_-,Q^+} & A_{P^{-\gamma}_-,Q^-} \\
        A_{P^{-\gamma}_+,Q^+} & A_{P^{-\gamma}_+,Q^-}
    \end{bmatrix} = \begin{bmatrix}
        A_{P^\gamma_+,Q^+} & -A_{P^\gamma_+,Q^-} \\
        -A_{P^\gamma_-,Q^+} & A_{P^\gamma_-,Q^-} \\
        -A_{P^{\gamma}_+,Q^+} & A_{P^{\gamma}_+,Q^-} \\
        A_{P^{\gamma}_-,Q^+} & A_{P^{\gamma}_-,Q^-}
    \end{bmatrix} = \begin{bmatrix}
        G\\
        -G
    \end{bmatrix},
    \end{equation*} where $G=\widehat{M}_{P^\gamma,Q'} = \begin{bmatrix}
        A_{P^\gamma_+,Q^+} & -A_{P^\gamma_+,Q^-} \\
        -A_{P^\gamma_-,Q^+} & A_{P^\gamma_-,Q^-}
    \end{bmatrix}$. Thus, the block corresponding to $(P^{-\gamma},Q')$ is the negated block of $(P^\gamma,Q')$.
    
    For $j=1,\dotsc,k$, let $P_j$ denote the partition of $P$ into $P_j$ so that the row sums of all $v\in P^\gamma_j$ are identical, and for $j\neq j'$ and $v\in P_j$ and $v\in P_{j'}$, the row sums of $v$ and $v'$ differ. Then, we similarly define the partition $P^{-\gamma}_j$ for $j=1,\dotsc,k$. For $v,v'\in P^{-\gamma}$, note that
    \begin{equation*}
        \sum_{w\in Q'} \widehat{M}_{v,w} = \sum_{w\in Q'} -G_{v,w} = - \sum_{w\in Q'} G_{v,w} \stackrel{\textup{(a)}}{=} - \sum_{w\in Q'} G_{v',w} =  \sum_{w\in Q'} -G_{v',w} = \sum_{w\in Q'} \widehat{M}_{v',w}
    \end{equation*} where \textup{(a)} holds because $P^\gamma_j$ has the same row sums by definition. Thus, the partitions $P^\gamma_j$ and $P^{-\gamma}_j$ are both valid refinements. Note that for each $j=1,\cdots,k$, the pair $P^\gamma_j$ and $P^{-\gamma}_j$ is a unipolar pair as in \cref{def_symmetric}\ref{def_symmetric_row_unipolar}. Then, the partition defined by $\widehat{\MP}_{i+1} \coloneqq \widehat{\MP}_{i}\setminus \{P^\gamma, P^{-\gamma}\} \cup \bigcup_{j=1}^k (P^\gamma_j\cup P^{-\gamma}_j)$ satisfies \cref{def_symmetric}\ref{def_symmetric_row_unipolar}, and \ref{def_symmetric_row_bipolar} follows since we did not alter any bipolar part in $\widehat{\MP}_i$. Thus, $(\widehat{\MP}_{i+1},\widehat{\MQ}_i)$ is symmetric and refines $(\widehat{\MP}_i,\widehat{\MQ}_i)$.\\

    Thus, we have shown that there exists a symmetric refinement $(\widehat{\MP}_{i+1}, \widehat{\MQ}_{i+1})$ in all cases, where $\widehat{\MQ}_{i+1} = \widehat{\MQ}_i$ holds. Using essentially the same proof, one can show that there exists a symmetric refinement $(\widehat{\MP}_{i+1},\widehat{\MQ}_{i+1})$ with $\widehat{\MP}_i = \widehat{\MP}_{i+1}$ if there exists some $(P',Q')\in \widehat{\MP}_i\times \widehat{\MQ}_i$ such that \eqref{eq_folding_col_sums} is not satisfied for $\widehat{M}$. This completes the proof.
\end{proof}

\begin{theorem}
    \label{thm_lp_symmetric}
    The coarsest bound-respecting and equitable partition $(\widehat{\MP},\widehat{\MQ})$ of $F_S(A,b,\ell,u,c)$ is symmetric.
\end{theorem}
\begin{proof}
    \Cref{thm_initial_partition_symmetric} shows that the initial partition is symmetric, and it follows from \cref{thm_refinement_symmetric} shows that applying color refinement on the constraint matrix produces a symmetric equitable partition, given the symmetric initial partition.
\end{proof}

\Cref{thm_lp_symmetric} shows that symmetric equitable partitions of split reformulations can be obtained by using the assumption that the partition is bound-respecting.
Although this assumption seems a bit arbitrary for now, we will motivate it in further detail later on.\\

A simple but crucial observation for symmetric partitions is that symmetric partition of the split reformulation induce partitions and signings for the original linear program.
\begin{proposition}
\label{thm_unipolar_original_mapping}
    Consider any symmetric partition $(\widehat{\MP},\widehat{\MQ})$. For its unipolar parts $\widehat{\MP}_U$ and $\widehat{\MQ}_U$, there exist $\gamma\in \{-1,1\}^{V_U}$ and a partition $\MP_U$ of $V_U$ such that $\widehat{\MP}_U = \bigcup_{P\in \MP}\{ P^\gamma, P^{-\gamma}\}$, and there exists $\lambda\in \{-1,1\}^{W_U}$ and a partition $\MQ_U$ of $W_U$ such that $\widehat{\MQ}_U = \bigcup_{Q\in \MQ} \{ Q^\lambda, Q^{-\lambda}\}$. For its bipolar parts $\widehat{\MP}_B$ and $\widehat{\MQ}_B$, there exist partitions $\MP_B$ of $V_B$ and $\MQ_B$ of $W_B$ such that $\widehat{\MP}_B = \{P^\onevec \cup P^{-\onevec} \mid P \in \MP_B\}$ and $\widehat{\MQ}_B = \{Q^\onevec \cup Q^{-\onevec} \mid  Q \in \MQ_B\}$ hold. 
\end{proposition}
\begin{proof}
    Consider any unipolar part $\widehat{P}\in \widehat{\MP}$. By unipolarity, there exists $P\subseteq V$ such that $\widehat{P} = P^\gamma$. Then, since $(\widehat{\MP},\widehat{\MQ})$ is symmetric,   \cref{def_symmetric}\ref{def_symmetric_row_unipolar} holds and shows $P^{-\gamma}\in \widehat{\MP}$ holds too. Since this holds for all elements and $\widehat{\MP}$ is a partition, it follows that $\widehat{\MP}$ consists of pairs $P^\gamma$, $P^{-\gamma}$. Then, $\MP_U$ is generated by all such $P$ and $\gamma$ is defined by picking a representative part for each pair $P^\gamma$, $P^{-\gamma}$. A similar argument using \cref{def_symmetric}\ref{def_symmetric_col_unipolar}shows the result for $\lambda$ and $\MQ_U$.
    For the bipolar parts, the statement follows directly from \cref{def_symmetric}\ref{def_symmetric_row_bipolar}\ref{def_symmetric_col_bipolar}.
\end{proof}

For any $\gamma\in \{-1,1\}^{V_U}$, $\lambda\in \{-1,1\}^{W_U}$, $\MP \coloneqq \MP_U \cup \MP_B$ and $\MQ\coloneqq \MQ_U \cup \MQ_B$ such as those in \cref{thm_unipolar_original_mapping} we say that they are \emph{generated by} $(\widehat{\MP},\widehat{\MQ})$.
We use $\widehat{P}^\gamma_U \coloneqq \{ P^\gamma \mid P \in \MP_U\}$ and $\widehat{Q}^\lambda_U \coloneqq \{ Q^\lambda \mid Q \in \MQ_U\}$, and define $\widehat{P}^{-\gamma}_U$ and $\widehat{Q}^{-\lambda}_U$ similarly.
Note that the choice for $\gamma$ is not unique as for each $P\in \MP_U$ we can exchange $P^{\gamma}$ and $P^{-\gamma}$ by negating $\gamma_v$ for all $v\in P$, i.e. by using $\gamma' \coloneqq (-\gamma_P,\gamma_{V\setminus P})$. Similarly, one can also choose for each $Q\in \MQ$ to simultaneously negate $\lambda_{w}$ for all $w\in Q$. We note that these actions only affect the signs $\gamma$ and $\lambda$ and not the sets $\MP_U$ and $\MQ_U$. 

\subsection{Reducing the split reformulation}

Now that we have a good understanding of the equitable partitions of split reformulations, we can consider how they can be used to reduce their dimension. 
In this section, we consider a linear program $F(A,b,\ell,u,c)$ and its split reformulation $F_S(A,b,\ell,u,c)$. We apply \cref{thm_lp_folding} to the split reformulation using a symmetric equitable partition $(\widehat{\MP},\widehat{\MQ})$ to obtain a smaller linear program.\\

To clarify the effect of the symmetric equitable partition on $F_S(A,b,\ell,u,c)$, we partition the row and variable sets into 3 sets each. Let $\gamma\in \{-1,1\}^{V_U}$, $\lambda \in \{-1,1\}^{W_U}$ and $\MP_U$ and $\MQ_U$ be generated by the symmetric equitable partition $(\widehat{\MP},\widehat{\MQ})$.
We consider the corresponding partition of $\widehat{V}$ and $\widehat{W}$ into $\widehat{V}^\gamma_U \coloneqq \bigcup_{P\in \widehat{\MP}^\gamma_U} P$, $\widehat{V}^{-\gamma}_U \coloneqq \bigcup_{P\in \widehat{\MP}^{-\gamma}_U} P$ and $\widehat{V}_B$. Similarly, the column partition is defined by $\widehat{W}^\lambda_U \coloneqq \bigcup_{Q\in \widehat{\MQ}^\lambda_U} Q$, $\widehat{W}^{-\lambda}_U \coloneqq \bigcup_{Q\in \widehat{\MQ}^{-\lambda}_U} Q$ and $\widehat{W}_B$. Then, we consider the variable partition $(\widehat{x}_{\widehat{W}^\lambda_U}, \widehat{x}_{\widehat{W}^{-\lambda}_U}, \widehat{x}_{\widehat{W}_B})$ for the split reformulation $F_S(A,b,\ell,u,c)$. Due to the symmetric definition of $\widehat{M}$, rearranging of the variables $\widehat{W}$ into $\widehat{W}^{\lambda}_U$ and $\widehat{W}^{-\lambda}_U$ can be achieved by multiplying by $\Lambda \coloneqq \diag(\lambda)$ from the right. Similarly, rearranging $\widehat{V}$ into $\widehat{V}^\gamma_U$ and $\widehat{V}^{-\gamma}_U$ can be achieved by multiplying by $\Gamma \coloneqq \diag(\gamma)$ from the left.
Furthermore, we remark that rows $v\in \widehat{V}_B$ have right-hand side $0$ and that variables $w\in W_B$ have objective $0$ and have identical bounds $u_w = -\ell_w$ for $w^+$ and $w^-$ by \cref{thm_initial_partition_symmetric}. 

To denote the variable bounds for the unipolar variables, we define the \emph{bound selector vector} $\mathrm{bs}(\ell,u,\lambda)\in \R^{W}$ such that that $\mathrm{bs}(\ell,u,\lambda)_w \coloneqq \begin{cases}
    u_w &\text{ if $\lambda_w = 1$ }\\
    -\ell_w &\text{ if $\lambda_w = -1$.}
\end{cases}$
By permuting the variables of $F_S(A,b,\ell,u,c)$ and applying the mentioned substitutions, we obtain the following linear program:
\begin{subequations}
\label{lp_centered_extended_reordered}
\begin{align}
    \min\, \transpose{c_{W_U}} \Lambda \widehat{x}_{\widehat{W}^\lambda_U} - \transpose{c_{W_U}} \Lambda \widehat{x}_{\widehat{W}^{-\lambda}_U} & \\
\begin{bmatrix}
    \Gamma A^1 \Lambda & -\Gamma A^1 \Lambda & \widehat{M}^1
    \\
    -\Gamma A^1 \Lambda & \Gamma A^1 \Lambda & -\widehat{M}^1
    \\
    \widehat{M}^2 & -\widehat{M}^2 & \widehat{M}^3 \\
\end{bmatrix}
\begin{bmatrix}
    \widehat{x}_{\widehat{W}^\lambda_U}\\
    \widehat{x}_{\widehat{W}^{-\lambda}_U}\\
    \widehat{x}_{\widehat{W}_B}
\end{bmatrix} &= \begin{bmatrix}
    \Gamma b_{W_U}\\
    -\Gamma b_{W_U}\\
    \zerovec
\end{bmatrix}\\
\zerovec \leq \widehat{x}_{\widehat{W}^\lambda_U} &\leq \mathrm{bs}(\ell,u,\lambda)_{W_U} \\
\zerovec \leq \widehat{x}_{\widehat{W}^{-\lambda}_U} &\leq \mathrm{bs}(\ell,u,-\lambda)_{W_U} \\
\zerovec \leq \widehat{x}_{\widehat{W}_B} &\leq \begin{bmatrix}
    u_{W_B}\\
    u_{W_B}
\end{bmatrix} 
\end{align}    
\end{subequations}
where we use $A^1 \coloneqq A_{V_U,W_U}$, $\widehat{M}^1 \coloneqq \Gamma \begin{bmatrix}
    A_{V_U,W_B} & -A_{V_U,W_B}
\end{bmatrix}$, $\widehat{M}^2 \coloneqq \begin{bmatrix}
    A_{V_B,W_U}\\
    -A_{V_B,W_U}
\end{bmatrix} \Lambda$ and\\ $\widehat{M}^3 \coloneqq \begin{bmatrix}
    A_{V_B,W_B} & - A_{V_B,W_B}\\
    -A_{V_B,W_B} & A_{V_B,W_B}
\end{bmatrix}$.

We use $\widehat{y}\in \R^{\widehat{\MQ}}$ to denote the transformed variables, where we let $(\widehat{y}_{\widehat{\MQ}^\lambda_U}, \widehat{y}_{\widehat{\MQ}^{-\lambda}_U}, \widehat{y}_{\widehat{\MQ}_B})$ be the reduced variable sets corresponding to the variables $(\widehat{x}_{\widehat{W}^\lambda_U}, \widehat{x}_{\widehat{W}^{-\lambda}_U}, \widehat{x}_{\widehat{\MQ}_B})$ in $F_S(A,b,\ell,u,c)$.

Then, we are almost ready to apply the equitable partition $(\widehat{\MP},\widehat{\MQ})$ to \eqref{lp_centered_extended_reordered}. Since we determined that $\widehat{\MP}$ consisted of the partitions $\widehat{\MP}^\gamma_U, \widehat{\MP}^{-\gamma}_U$ and $\widehat{\MP}_{B}$, $\Pi_{\widehat{\MP}}$ is a block diagonal matrix consisting of the blocks $\Pi_{\widehat{\MP}^\gamma_U}, \Pi_{\widehat{\MP}^{-\gamma}_U}$ and $\Pi_{\widehat{\MP}_{B}}$. Moreover, note that since $\widehat{\MP}^{\gamma}_U$ and $ \widehat{\MP}^{-\gamma}_U$ arise from the same underlying partition $\MP_U$ of the rows $V_U$, we have that $\Pi_{\widehat{\MP}^{\gamma}_U} = \Pi_{\widehat{\MP}^{-\gamma}_U} = \Pi_{\MP_U}$. Similarly, $\Pi_{\widehat{\MQ}}$ is a block diagonal matrix consisting of the blocks $\Pi_{\widehat{\MQ}^\lambda_U}$, $\Pi_{\widehat{\MQ}^{-\lambda}_U}$ and $\Pi_{\widehat{\MQ}_{W_B}}$, where $\Pi_{\widehat{\MQ}^\lambda_U} = \Pi_{\widehat{\MQ}^{-\lambda}_U} = \Pi_{\MQ_U}$ holds. In \cref{thm_bipolar_symmetric_equitable_sum} we observed that equitable bipolar parts have zero row and column sums. In \cref{thm_zero_sums_equitable}, we show that this implies that the matrices $\widehat{M}^i$ for $i=1,2,3$ transform into all-zero matrices once we apply \cref{thm_lp_folding}. Then, the reduction of $F_S(A,b,\ell,u,c)$ in the reordered form~\eqref{lp_centered_extended_reordered} is given by~\eqref{lp_extended_reduced}.

\begin{lemma} \label{thm_zero_sums_equitable}Let $(\widehat{\MP},\widehat{\MQ})$ be a symmetric equitable partition of $F_S(A,b,\ell,u,c)$, where $\widehat{M}$ is the constraint matrix of $F_S(A,b,\ell,u,c)$. Then,
     $\transpose{\Pi_{\widehat{\MP}_B}}\widehat{M}_{\widehat{V}_B,\star} = \zerovec$ and $\widehat{M}_{\star, \widehat{W}_B} \widetilde{\Pi_{\widehat{\MQ}_B}} = \zerovec$ hold.
\end{lemma}
\begin{proof}
    For any $P\in \widehat{\MP}_B$ and any $w\in W$ we have that: $(\transpose{\Pi_{\widehat{\MP}_B}}\widehat{M}_{\widehat{V}_B,\star})_{P,w} = \sum_{v\in P} \widehat{M}_{v,w} = 0$, where the final equality follows from \cref{thm_bipolar_symmetric_equitable_sum}. Similarly, for any $Q\in \widehat{\MQ}_B$ and any $v\in V$ we have that $(\widehat{M}_{\star, \widehat{W}_B} \widetilde{\Pi_{\widehat{\MQ}_B}} = \frac{1}{Q}\sum_{w\in Q} \widehat{M}_{v,w} = 0$ holds by \cref{thm_bipolar_symmetric_equitable_sum}.
\end{proof}

\begin{subequations}
\label{lp_extended_reduced}
\begin{align}
    \min\,\transpose{(\transpose{\widetilde{\Pi_{\MQ_u}}} \Lambda c_{W_U})} \widehat{y}_{\widehat{\MQ}^\lambda_U} - \transpose{(\transpose{\widetilde{\Pi_{\MQ_u}}} \Lambda c_{W_U})} \widehat{y}_{\widehat{\MQ}^{-\lambda}_U} & \label{lp_extended_reduced_obj} \\
\begin{bmatrix}
    \transpose{\Pi_{\MP_U}} \Gamma A^1 \Lambda \widetilde{\Pi_{\MQ_U}}& -\transpose{\Pi_{\MP_U}} \Gamma A^1 \Lambda \widetilde{\Pi_{\MQ_U}}&\zerovec
    \\
    -\transpose{\Pi_{\MP_U}}\Gamma A^1 \Lambda \widetilde{\Pi_{\MQ_U}} & \transpose{\Pi_{\MP_U}}\Gamma A^1 \Lambda \widetilde{\Pi_{\MQ_U}}& \zerovec
    \\
    \zerovec & \zerovec & \zerovec 
\end{bmatrix}
\begin{bmatrix}
    \widehat{y}_{\widehat{\MQ}^\lambda_U}\\
    \widehat{y}_{\widehat{\MQ}^{-\lambda}_U}\\
    \widehat{y}_{\widehat{\MQ}_B}
\end{bmatrix} &= \begin{bmatrix}
    \transpose{\Pi_{\MP_U}} \Gamma b_{W_U}\\
    -\transpose{\Pi_{\MP_U}} \Gamma b_{W_U}\\
    \zerovec
\end{bmatrix}\label{lp_extended_reduced_eqs}\\
\zerovec \leq \widehat{y}_{\widehat{\MQ}^\lambda_U} &\leq \transpose{\Pi_{\MQ_U}}\mathrm{bs}(\ell,u,\lambda)_{W_U} \\
\zerovec \leq \widehat{y}_{\widehat{\MQ}^{-\lambda}_U} &\leq \transpose{\Pi_{\MQ_U}}\mathrm{bs}(\ell,u,-\lambda)_{W_U} \label{lp_extended_reduced_unipolar_bounds} \\
0 \leq \widehat{y}_{\widehat{\MQ}_B} &\leq 2 \transpose{\Pi_{\MQ_B}} u_{W_B} \label{lp_extended_reduced_bipolar_bounds}
\end{align}    
\end{subequations}

There are a few things to note about the system~\eqref{lp_extended_reduced}. First of all, the rows corresponding to the bipolar rows clearly become redundant, and the columns $\widehat{y}_{\widehat{\MQ}_B}$ corresponding to the bipolar variables become independent of the other variables. Indeed,~\eqref{lp_extended_reduced_obj}-\eqref{lp_extended_reduced_unipolar_bounds} and~\eqref{lp_extended_reduced_bipolar_bounds} are the Cartesian product of two polyhedra.
Since $\widehat{y}_{\widehat{\MQ}_B}$ does not appear in the objective, we may arbitrarily fix $\widehat{y}_{\widehat{\MQ}_B} = 0$. Furthermore, the system~\eqref{lp_extended_reduced_obj}-\eqref{lp_extended_reduced_unipolar_bounds} still has the redundancies introduced by the split reformulation of $F(A,b,\ell,u,c)$. 
In fact, the linear program on the variables $\widehat{\MQ}^\lambda_U \cup \widehat{\MQ}^{-\lambda}_U$ in~\eqref{lp_extended_reduced_obj}-\eqref{lp_extended_reduced_unipolar_bounds} is exactly the split reformulation of another linear program!
One can verify that~\eqref{lp_extended_reduced_obj}-\eqref{lp_extended_reduced_unipolar_bounds} is equivalent to split reformulation $G_S$ of the following linear program:
\begin{equation*}
    G\coloneqq F(\transpose{\Pi_{\MP_U}} \Gamma A^1 \Lambda \widetilde{\Pi_{\MQ_U}},\transpose{\Pi_{\MP_U}} \Gamma b_{W_U},\transpose{\Pi_{\MQ_U}}\mathrm{bs}(\ell,u,\lambda)_{W_U}, - \transpose{\Pi_{\MQ_U}}\mathrm{bs}(\ell,u,-\lambda)_{W_U}, \transpose{\widetilde{\Pi_{\MQ_U}}} \Lambda c_{W_U}).
\end{equation*} 

\begin{proposition}
    \label{thm_split_reduce_equiv_reduce_split}
    Let $(\widehat{\MP},\widehat{\MQ})$ be a symmetric equitable partition of $F_S(A,b,\ell,u,c)$, and let $\lambda,\gamma,\MP$ and $\MQ$ be generated by $(\widehat{\MP},\widehat{\MQ})$. Then the reduced linear program can be defined as: 
    \begin{equation*}
        R^{\widehat{\MP},\widehat{\MQ}}(F_S(A,b,\ell,u,c)) \coloneqq G_S \times \{ \widehat{y}_{\widehat{\MQ}_B}\in \R^{\widehat{\MQ}_B} \mid \zerovec \leq \widehat{y}_{\widehat{\MQ}_B} \leq 2\transpose{\Pi_{\MQ_B}} u_{W_B}\}.
    \end{equation*}
\end{proposition}
\begin{proof}
    By reordering the variables and rows of $F_S(A,b,\ell,u,c)$ according to the signs $\lambda$ and $\gamma$ for the unipolar variables we can obtain the form~\eqref{lp_centered_extended_reordered}, where we implicitly use~\cref{thm_initial_partition_symmetric} to show that bipolar rows have zero right-hand side and bipolar variables have zero objective and identical upper and lower variable bounds.
    
    Then, the reduced form $R^{\widehat{\MP},\widehat{\MQ}}(F_S(A,b,\ell,u,c))$ is given by~\eqref{lp_extended_reduced}, where we use \cref{thm_zero_sums_equitable} to show that the bipolar rows and variables have only zero coefficients. Then, note that the feasible region of the linear program~\eqref{lp_extended_reduced} is exactly the Cartesian product of $G_S$ and $\{ \widehat{y}_{\widehat{\MQ}_B}\in \R^{\widehat{\MQ}_B} \mid \zerovec \leq \widehat{y}_{\widehat{\MQ}_B} \leq 2\transpose{\Pi_{\MQ_B}} u_{W_B}\}$ on the unipolar and bipolar parts, respectively, where we use that the all-zero rows with right-hand side zero in~\eqref{lp_extended_reduced_eqs} are redundant.
\end{proof}

Note that the linear program $G$ is obtained from $F(A,b,\ell,u,c)$ by taking the linear program induced by $V_U$ and $W_U$, scaling the rows and columns with $\Gamma$ and $\Lambda$, applying the equitable partition $(\MP_U,\MQ_U)$ and then computing the split reformulation. Thus, application of the symmetric equitable partition to the split reformulation can also be viewed as taking the unipolar rows and variables in the original linear program, scaling the rows and variables by $\pm 1$ factors, computing an equitable partition and then applying the split reformulation. The final step that computes the split reformulation is not helpful computationally, as $G$ has less variables than $G_S$ and is equivalent to $G_S$. Thus, we wish to reduce the linear programs into the form of $G$. We can obtain $G$ from~\eqref{lp_extended_reduced} by reversing the split reformulation as described in~\cref{thm_splitting_equivalent}. We use $y_{\MQ_U} \coloneqq \widehat{y}_{\widehat{\MQ}^\lambda_U} - \widehat{y}_{\widehat{\MQ}^{-\lambda}_U}$ to define the variables of $G$. Then, $G$ is given by the following linear program.
\begin{subequations}
\label{lp_original_refreduced}
\begin{align}
    \min \, \transpose{(\transpose{\widetilde{\Pi_{\MQ_u}}} \Lambda c_{W_U})} y_{\MQ_U} & \\
    \transpose{\Pi_{\MP_U}} \Gamma A_{V_U,W_U} \Lambda \widetilde{\Pi_{\MQ_U}} y_{\MQ_U} &= \transpose{\Pi_{\MP_U}} \Gamma b_{W_U}\\
-\transpose{\Pi_{\MQ_U}}\mathrm{bs}(\ell,u,-\lambda)_{W_U} \leq y_{\MQ_U} &\leq \transpose{\Pi_{\MQ_U}}\mathrm{bs}(\ell,u,\lambda)_{W_U} 
\end{align}    
\end{subequations}

Note that $G$ is obtained only through transformations that preserve the feasibility and optimality of the linear program. In particular, both the split reformulation and the dimension reduction via color refinement preserve feasibility and optimality.
Thus, we can reduce any linear program $F(A,b,\ell,u,c)$ to the form of $G$ in~\eqref{lp_original_refreduced} through a series of reformulations. \\

One small assumption in our analysis was that we required $\ell \leq 0$ and $u_w \geq 0$ to hold so that the split reformulations are well-defined. For any feasible linear program $\ell \leq u$ must hold, and it is then possible to choose $\delta \in \R^W$ such that $\ell \leq \delta \leq u$ holds. Then, one can apply an affine transformation $x = x' + \delta$ to the linear program $F(A,b,\ell,u,c)$ . By doing so, one obtains the equivalent linear program $F(A,b-A\delta, \ell - \delta, u - \delta,c)$, where the objective function has an additional constant offset of $\transpose{c} \delta$. Clearly, $\ell - \delta \leq 0$ and $u-\delta \geq 0$ hold, so the split reformulation is well defined.

In \cref{def_refeqpart}, we formulate the concept of a \emph{reflection reduction}, which represents the pair of linear programs given by $F(A,b,\ell,u,c)$ and its reduced form obtained by computing a symmetric equitable partition of $F_S(A,b-A\delta, \ell - \delta, u - \delta,c)$ for any $\ell \leq \delta \leq u$. In \cref{thm_refl_reduction}, we show that the symmetric equitable partitions of $F_S(A,b-A\delta, \ell - \delta, u - \delta,c)$ imply the existence of a reflection reduction.\\

In order to simplify the notation,  we
define $\mathrm{comp}(\lambda, \delta, \ell,u)\in \R^W$ to be a vector such that $\mathrm{comp}(\lambda, \delta, \ell,u)_w \coloneqq \begin{cases} 
\ell_w - \delta_w \text{ if } \lambda_w = 1\\
\delta_w - u_w \text{ if } \lambda_w = -1
\end{cases}$ holds for all $w\in W$, where we define that $\infty + k = \infty $ for all finite $k\in \R$ and similarly that $-\infty + k = -\infty$. The vector $\mathrm{comp}(\lambda, \delta, \ell,u)$ corresponds to the complemented variable lower bounds. Note that swapping $\ell$ and $u$ immediately also gives the complemented variable upper bounds using $\mathrm{comp}(\lambda, \delta, u, \ell)$.
 For the variable bound pair $(\ell,u)$, we will primarily use the shorthand notation $\ell^{\delta,\lambda} \coloneqq \mathrm{comp}(\lambda, \delta, \ell, u)$ and $u^{\delta,\lambda} \coloneqq \mathrm{comp}(\lambda, \delta, u, \ell)$. One can verify that $\ell^{\delta,\lambda} = -\mathrm{bs}(\ell-\delta,u-\delta,-\lambda)$ and $u^{\delta,\lambda} = \mathrm{bs}(\ell-\delta,u-\delta,\lambda)$ hold.

\begin{definition}[Reflection reduction]
    \label{def_refeqpart}
    Consider the linear program $F(A,b,\ell, u, c)$ with $A\in \R^{V\times W}$, $b\in \R^{V}$, $c\in \R^{W}$ and vectors $\ell, u$ with $\ell_w \in \R \cup \{-\infty\}$ and $u_w\in \R \cup \{\infty\} $ such that $\ell_w \leq u_w$ holds for all $w\in W$.
    Let $V_U\subseteq V$ and $W_U\subseteq W$, and define $W_B\coloneqq W\setminus W_U$.
    Consider a partition $(\MP,\MQ)$ where $\MP$ is a partition of $V_U$ and $\MQ$ is a partition of $W_U$, and let $\gamma\in \{-1,1\}^{V_U}$, $\lambda\in \{-1,1\}^{W_u}$ and let $\delta \in \R^W$ be such that $\ell \leq \delta \leq u$ holds.
    Then, define $A'\coloneqq \transpose{\Pi_{\MP}}\Gamma A_{V_U,W_U}\Lambda \widetilde{\Pi_{\MQ}}$, $b'\coloneqq \transpose{\Pi_{\MP}}\Gamma(b- A\delta)_{V_U}$, $\ell' \coloneqq \transpose{\Pi_{\MQ}} \ell^{\lambda,\delta}_{W_U}$, $u'\coloneqq  \transpose{\Pi_{\MQ}} u^{\lambda,\delta}_{W_U}$ and $c'\coloneqq \transpose{\widetilde{\Pi_\MQ}} \Lambda c_{W_U}$.
    We say that $(\MP,\MQ,\gamma,\lambda, \delta, V_U, W_U)$ is a \emph{reflection reduction} of $F(A,b,\ell,u,c)$ if the following hold.
    \begin{enumerate}[label=(\roman*)]
        \item \label{def_refeqpart_to_red} For every $x\in P(A,b,\ell,u)$, $y\coloneqq \transpose{\Pi_\MQ} \Lambda (x_{W_U}-\delta_{W_U})$ is feasible for $P(A',b',\ell',u')$ and $\transpose{c} x = \transpose{c'} y + \transpose{c} \delta$ holds.
        \item \label{def_refeqpart_from_red} For every $y\in P(A',b',\ell',u')$,  $x \coloneqq ( \Lambda \widetilde{\Pi_\MQ} y + \delta_{W_U},\delta_{W_B})$ is feasible for $P(A,b,\ell,u)$ and $\transpose{c'} y + \transpose{c} \delta = \transpose{c} x$ holds.
    \end{enumerate}
\end{definition}

The main use of reflection reductions is that they show that one can equivalently solve the linear program $F(A',b',\ell',u',c')$ which has reduced dimension compared to the original linear program $F(A,b,\ell,u,c)$. Then, our main result shows that symmetric equitable partitions imply reflection reductions. Its proof relies strongly on the equivalent linear programming relations defined by the split reformulation and the DRCR reduction.

\begin{theorem}
    \label{thm_refl_reduction}
     Consider the linear program $F(A,b,\ell, u, c)$ with $A\in \R^{V\times W}$, $b\in \R^{V}$, $c\in \R^{W}$ and vectors $\ell, u$ with $\ell_w \in \R \cup \{-\infty\}$ and $u_w\in \R \cup \{\infty\} $ such that $\ell_w \leq u_w$ holds for all $w\in W$, and let $\delta \in \R^W$ be given such that $\ell \leq \delta \leq u$ holds.
     Let $(\widehat{\MP},\widehat{\MQ})$ be a symmetric equitable partition of $F_S(A,b-A\delta, \ell - \delta, u - \delta,c)$, and let $\gamma \in \{-1,1\}^{V_U}$, $\lambda \in\{-1,1\}^{W_U}$, $\MP$ and $\MQ$ be generated by $(\widehat{\MP},\widehat{\MQ})$.
    Then, $(\MP,\MQ,\gamma,\lambda,\delta, V_U, W_U)$ is a reflection reduction of $F(A,b,\ell,u,c)$.
\end{theorem}
\begin{proof}
    We use $\Lambda \coloneqq \diag(\lambda)$, $\Gamma \coloneqq \diag(\gamma)$, $A'\coloneqq \transpose{\Pi_{\MP}}\Gamma A_{V_U,W_U}\Lambda \widetilde{\Pi_{\MQ}}$, $b'\coloneqq \transpose{\Pi_{\MP}}\Gamma(b- A\delta)_{V_U}$, $\ell' \coloneqq \transpose{\Pi_{\MQ}} \ell^{\lambda,\delta}_{W_U}$, $u'\coloneqq  \transpose{\Pi_{\MQ}} u^{\lambda,\delta}_{W_U}$ and $c'\coloneqq \transpose{\widetilde{\Pi_\MQ}} \Lambda c_{W_U}$ as in \cref{def_refeqpart}.
    
    First, let us prove \cref{def_refeqpart}\ref{def_refeqpart_to_red}. Let $x$ be a feasible solution for $F(A,b,\ell,u,c)$. Then, we have the following series of linear programs with feasible solutions:
    \begin{align*}
     & x \text{ is feasible for } F(A,b,\ell,u,c)\\ \stackrel{\textup{(a)}}\implies & x' \coloneqq x-\delta  \text{ is feasible for $F(A,b-A\delta, \ell - \delta, u - \delta,c)$} \\\stackrel{\textup{(b)}}\implies & \widehat{x} \coloneqq (\max(x',\zerovec),\max(-x',\zerovec)) \text{ is feasible for }F_S(A,b-A\delta, \ell - \delta, u - \delta,c) \\ \stackrel{\textup{(c)}}\implies & \widehat{y} \coloneqq \transpose{\Pi_{\widehat{\MQ}}} \widehat{x}\text{ is feasible for } R^{\widehat{\MP},\widehat{\MQ}}(F_S(A,b-A\delta, \ell - \delta, u - \delta,c))\\ \stackrel{\textup{(d)}}\implies & (\widehat{y}_{\widehat{\MQ}_U},\widehat{y}_{\widehat{\MQ}_B})\text{ is feasible for } P_S(A',b',\ell',u') \times \{\widehat{y}_{\widehat{\MQ}_B} \in \R^{\widehat{\MQ}_B} \mid \zerovec \leq \widehat{y}_{\widehat{\MQ}_B} \leq 2 \transpose{\Pi_{\MQ_B}} (u - \delta)_{W_B}\}\\
     \stackrel{\textup{(e)}}\implies & y_{\MQ_U} \coloneqq \widehat{y}_{\widehat{\MQ}^{\lambda}_U} - \widehat{y}_{\widehat{\MQ}^{-\lambda}_U} \text{ is feasible for } F(A',b',\ell',u',c')
    \end{align*}
    where \textup{(a)} a holds by \cref{thm_affine_equivalent}, \textup{(b)} holds by \cref{thm_splitting_equivalent}\ref{thm_splitting_equivalent_to_red} and \textup{(c)} holds by \cref{thm_lp_folding}\ref{thm_lp_folding_to_red}. The implication in \textup{(d)} holds by \cref{thm_split_reduce_equiv_reduce_split}, and 
    \textup{(e)} follows from \cref{thm_splitting_equivalent} and since the polyhedron for the bipolar variables is related only by the Cartesian product.
    For $Q\in \MQ_U$, consider that 
    \begin{equation*}
     y_Q = (\widehat{y}_{\MQ^\lambda_U})_Q - (\widehat{y}_{\MQ^{-\lambda}_U})_Q = \sum_{w\in Q} \widehat{x}_{w^\lambda} - \widehat{x}_{w^{-\lambda}} \stackrel{\textup{(f)}}= \sum_{w\in Q} \lambda_w x'_w,  
    \end{equation*}
    where $\textup{(f)}$ follows since 
    \begin{equation*}
       \widehat{x}_{w^\lambda} - \widehat{x}_{w^{-\lambda}} = \begin{cases}
        \max(x'_w,0)-\max(-x'_w,0)\text{ if $\lambda_w = 1$}\\
        \max(-x'_w,0)-\max(x'_w,0)\text{ if $\lambda_w=-1$}
    \end{cases} = \begin{cases}
        x'_w \text{ if $\lambda_w = 1$}\\
        -x'_w \text{ if $\lambda_w = 1$}
    \end{cases} = \lambda_w x'_w. 
    \end{equation*} Thus, $y_{\MQ_U} = \transpose{\Pi_{\MQ_U}} \Lambda x'_{W_U} = \transpose{\Pi_{\MQ_U}} \Lambda (x_{W_U}-\delta_{W_U})$ is feasible for $F(A',b',\ell',u',c')$, which establishes the feasibility of \cref{def_refeqpart}\ref{def_refeqpart_to_red} for $(\MP,\MQ,\gamma, \lambda, \delta,V_U,W_U)$. \\For the objective, we have
    \begin{align*}
        \transpose{c} x - \transpose{c} \delta &\stackrel{\textup{(g)}}= \transpose{c} x' \stackrel{\textup{(h)}}=
        \transpose{c} (\widehat{x}_{W_U}^+ - \widehat{x}_{W_U}^-) \stackrel{\textup{(i)}}= \transpose{c'}(\widehat{y}_{\widehat{\MQ}^{\lambda}_U} - \widehat{y}_{\widehat{\MQ}^{-\lambda}_U}) \stackrel{\textup{(j)}}= \transpose{c'}y
    \end{align*}
    where \textup{(g)} follows from \cref{thm_affine_equivalent}, \textup{(h)} follows from \cref{thm_splitting_equivalent}\ref{thm_splitting_equivalent_to_red} and the bipolar variables having zero objective by \cref{thm_initial_partition_symmetric}, \textup{(i)} follows from \cref{thm_lp_folding} and the reordering of the variables described in \cref{thm_split_reduce_equiv_reduce_split} and \textup{(j)} follows from \cref{thm_splitting_equivalent}\ref{thm_splitting_equivalent_from_red}. This establishes the objective in \cref{def_refeqpart}\ref{def_refeqpart_to_red}.

    To show \cref{def_refeqpart}\ref{def_refeqpart_from_red}, let $y \in P(A',b',\ell',u)$ be a feasible solution. Then, we have the following series of linear programs with feasible solutions.
    \begin{align*}
        &y\text{ is feasible for }F(A',b',\ell',u',c')\\
        \stackrel{\textup{(k)}}\implies & \widehat{y}_{\widehat{\MQ}_U} = (\widehat{y}_{\widehat{\MQ}^\lambda_U},\widehat{y}_{\widehat{\MQ}^{-\lambda}_U}) \coloneqq (\max(y_{\MQ_U},\zerovec),\max(-y_{\MQ_U},\zerovec))\text{ is feasible for } F_S(A',b',\ell',u',c')\\
        \stackrel{\textup{(l)}}{\implies} & \widehat{y} \coloneqq (\widehat{y}_{\widehat{\MQ}_U},\zerovec_{\widehat{\MQ}_B})\text{ is feasible for } P_S(A',b',\ell',u') \times \{\widehat{y}_{\widehat{\MQ}_B} \in \R^{\widehat{\MQ}_B} \mid \zerovec \leq \widehat{y}_{\widehat{\MQ}_B} \leq 2 \transpose{\Pi_{\MQ_B}} (u - \delta)_{W_B}\}\\
        \stackrel{\textup{(m)}}{\implies} & \widehat{y}\text{ is feasible for } R^{\widehat{\MP},\widehat{\MQ}}(F_S(A,b-A\delta, \ell - \delta, u - \delta,c)) \\
        \stackrel{\textup{(n)}}\implies&
        \widehat{x} \coloneqq \widetilde{\Pi_{\widehat{\MQ}}} \widehat{y} \text{ is feasible for $F_S(A,b-A\delta, \ell - \delta, u - \delta,c)$}\\
        \stackrel{\textup{(o)}} \implies & x'\coloneqq \widehat{x}^+ - \widehat{x}^-\text{ is feasible for } F(A,b-A\delta, \ell-\delta, u-\delta,c)\\
        \stackrel{\textup{(p)}} \implies & x \coloneqq x' + \delta \text{ is feasible for $F(A,b,\ell,u,c)$}
    \end{align*}

    Here, \textup{(k)} follows from \cref{thm_splitting_equivalent}\ref{thm_splitting_equivalent_to_red} and \textup{(l)} holds since taking the Cartesian product preserves feasibility and since $\zerovec$ is clearly feasible. Moreover, \textup{(m)} follows from \cref{thm_split_reduce_equiv_reduce_split}, \textup{(n)} follows from \cref{thm_lp_folding}\ref{thm_lp_folding_from_red}, \textup{(o)} follows from \cref{thm_splitting_equivalent}\ref{thm_splitting_equivalent_from_red} and \textup{(p)} follows from \cref{thm_affine_equivalent}.    

    Next, let us express $y$ in terms of $x$. If $w\in W_U$ holds, then we have the following for $Q\in \MQ_U$ such that $w\in Q$
    \begin{equation*}
        x'_w = \widehat{x}_{w^+} - \widehat{x}_{w^-} = (\widetilde{\Pi_{\widehat{\MQ}}} \widehat{y})_{w^+} - (\widetilde{\Pi_{\widehat{\MQ}}} \widehat{y})_{w^+} \stackrel{\textup{(q)}}{=} \frac{\lambda_w}{|Q|} (\widehat{y}_{Q^\lambda} - \widehat{y}_{Q^{-\lambda}}) = \frac{\lambda_w}{|Q|} y_Q = (\Lambda \widetilde{\Pi_{\MQ_U}} y)_w  
    \end{equation*}
    where \textup{(q)} holds because we have by the reordering of the variables that
    \begin{equation*}
        (\widetilde{\Pi_{\widehat{\MQ}}} \widehat{y})_{w^+} - (\widetilde{\Pi_{\widehat{\MQ}}} \widehat{y})_{w^-} = \begin{cases}
            \frac{1}{|Q|} (\widehat{y}_{Q^\lambda} - \widehat{y}_{Q^{-\lambda}}) &\text{ if $\lambda_w = 1$}\\
            \frac{1}{|Q|} (\widehat{y}_{Q^{-\lambda}} - \widehat{y}_{Q^{\lambda}}) &\text{ if $\lambda_w = -1$}
        \end{cases} = \frac{\lambda_w}{|Q|} (\widehat{y}_{Q^\lambda} - \widehat{y}_{Q^{-\lambda}}).
    \end{equation*}
    For $w\in W_B$, then for $Q \in \widehat{\MQ_B}$ such that $w\in Q$ we have the following:
    \begin{equation*}
        x'_w = \widehat{x}_{w^+} - \widehat{x}_{w^-} = (\widetilde{\Pi_{\widehat{\MQ}}} \widehat{y})_{w^+} - (\widetilde{\Pi_{\widehat{\MQ}}} \widehat{y})_{w^-} = \frac{1}{|Q|}(\widehat{y}_{Q} - \widehat{y}_{Q}) = 0.
    \end{equation*}
    Note that this equation does not depend on the value of $\widehat{y}_Q$, so in fact the choice to assign $y_{\MQ_B} = \zerovec$ is arbitrary as its values always cancel when they are transformed into $x'_{W_B}$.
    Then, we have that $x = x' + \delta = (\Lambda \widetilde{\Pi_{\MQ_U}} y,\zerovec) + \delta = (\Lambda \widetilde{\Pi_{\MQ_U}} y + \delta_{W_U},\delta_{W_B})$, which shows that the feasibility in \cref{def_refeqpart}\ref{def_refeqpart_from_red} is satisfied.
    For the objective, we have similar reductions as in the first part of the proof but in reverse order:
     \begin{align*}
        \transpose{c'}y \stackrel{\textup{(s)}}= \transpose{c'}(\widehat{y}_{\widehat{\MQ}^{\lambda}_U} - \widehat{y}_{\widehat{\MQ}^{-\lambda}_U}) \stackrel{\textup{(t)}}= \transpose{c'} x'_{W_U} \stackrel{\textup{(u)}}= \transpose{c} x' \stackrel{\textup{(v)}}=
        \transpose{c} x - \transpose{c} \delta,
    \end{align*}
    where \textup{(s)} follows from \cref{thm_splitting_equivalent}\ref{thm_splitting_equivalent_to_red}, \textup{(t)} holds by the relation between $\widehat{x}$ and $\widehat{y}$ that we explained above, \textup{(u)} holds since bipolar variables have zero objective by \cref{thm_initial_partition_symmetric} and \textup{(v)} holds by \cref{thm_affine_equivalent}. Then, the objective in \cref{def_refeqpart}\ref{def_refeqpart_from_red} is also satisfied.
    
    Thus, we have established that $(\MP, \MQ, \gamma, \lambda, \delta, V_U, W_U)$ satisfies \cref{def_refeqpart}, which shows that it is a reflection reduction of $F(A,b,\ell,u,c)$.
\end{proof}
\subsection{Choosing the affine offset}
\label{sec_choosing_affine_offset}
In order to apply reflection reductions to arbitrary linear programs, we had to introduce the affine offset $\delta$ to ensure that the lower bounds were nonpositive and the upper bounds were nonnegative. Given $\delta$, the linear transformation that is represented by the $\pm1$ scaling vectors $\lambda$ and $\gamma$ follows automatically from the symmetric equitable partition of the split reformulation. In this section, we argue how to choose $\delta$. In previous work, Hojny~\cite{Hojny2025-si} chooses $\delta$ to be the center of the variable domains such that $\delta_w \coloneqq \frac{\ell_w + u_w}{2}$. One issue is that such a center is not well-defined for variables with infinite lower or upper bounds. Similarly to Hojny, we use the offset $\delta^\mathcal{C}$ defined for all $w\in W$ as
\begin{equation*}
    \delta^\mathcal{C}_w \coloneqq \begin{cases}
        \frac{\ell_w + u_w}2 &\text{if $\ell_w$ and $u_w$ are finite.}\\
        0 &\text{if $\ell_w = -\infty$ and $u_w = \infty$}\\
        \ell_w &\text{if $\ell_w$ is finite and $u_w = \infty$.}\\
        u_w &\text{if $\ell_w = -\infty$ and $u_w$ is finite.}
    \end{cases}
\end{equation*}

Next, let us argue why $\delta^\mathcal{C}$ is a good choice for the offset. First of all, we note that for the most common case of variables $w$ with $\ell_w$ and $u_w$ finite, they can only belong to bipolar parts if $-(\ell_w - \delta_w) = u_w - \delta_w$ holds, which holds exactly if $\delta_w = \frac{\ell_w + u_w}2$ holds. Thus, variables with finite bounds can only be detected as bipolar variables if $\delta_w = \delta^\mathcal{C}_w = \frac{\ell_w + u_w}2$ holds. Note that in the case exactly one of $\ell_w$ and $u_w$ is finite, we can never detect bipolarity. If $\ell_w = -\infty$ and $u_w = \infty$ hold, then the choice of $\delta_w$ does not affect the transformed variable bounds and bipolarity may be detected in any case.\\

Second, recall that we assumed that the symmetric equitable partitions must be bound-respecting. For the linear program $F(A,b- A \delta,\ell- \delta,u- \delta,c)$, this condition can be stated as follows.
The partition $(\widehat{\MP},\widehat{\MQ})$ of the split reformulation must satisfy for all $Q\in \widehat{\MQ}$ and any $w_1^s,w_2^s\in Q$ that $\ell_{w_1} - \delta_{w_1} = \ell_{w_2} - \delta_{w_2}$ and $u_{w_1} - \delta_{w_1} = u_{w_2}- \delta_{w_2}$ hold if $s_1 = s_2$ holds, and if $s_1 \neq s_2$ holds that $\ell_{w_1} - \delta_{w_1} = -u_{w_2} + \delta_{w_2}$ and $u_{w_1} - \delta_{w_1} = - \ell_{w_2} + \delta_{w_2}$ hold. 
Next, we show that for $\delta = \delta^\mathcal{C}$ only the \emph{size} $u_w - \ell_w$ of the domain of $w$ matters.  
Assume for all $Q\in \MQ$ and any $w_1^s,w_2^s \in Q$ that $u_{w_1} - \ell_{w_1} = u_{w_2} - \ell_{w_2}$ holds. For the $s_1 = s_2$ case, we can use this assumption to show that: 
\begin{equation*}
    \ell_{w_1} - \delta^\mathcal{C}_{w_1} = \ell_{w_1} - \frac{\ell_{w_1} + u_{w_1}}{2} = 
    -\frac{u_{w_1} - \ell_{w_1}}{2} = \frac{u_{w_2} - \ell_{w_2}}{2} = \ell_{w_2} - \frac{\ell_{w_2} + u_{w_2}}{2} = \ell_{w_2} - \delta^\mathcal{C}_{w_2}
\end{equation*}
and similarly that
\begin{equation*}
    u_{w_1} - \delta^\mathcal{C}_{w_1} = u_{w_1} - \frac{\ell_{w_1} + u_{w_1}}{2} = 
    \frac{u_{w_1} - \ell_{w_1}}{2} = \frac{u_{w_2} - \ell_{w_2}}{2} = u_{w_2} - \frac{\ell_{w_2} + u_{w_2}}{2} = u_{w_2} - \delta^\mathcal{C}_{w_2}.
\end{equation*}
Also for the $s_1 \neq s_2$ case, we can use the assumption to show:
\begin{equation*}
\ell_{w_1} - \delta^\mathcal{C}_{w_1} =
-\frac{u_{w_1} - \ell_{w_1}}{2} = -\frac{u_{w_2} - \ell_{w_2}}{2} = - u_{w_2} + \frac{u_{w_2} + \ell_{w_2}}2 = -u_{w_2} + \delta^\mathcal{C}_{w_2}
\end{equation*}
and similarly:
\begin{equation*}
u_{w_1} - \delta^\mathcal{C}_{w_1} =
\frac{u_{w_1} - \ell_{w_1}}{2} = \frac{u_{w_2} - \ell_{w_2}}{2} = - \ell_{w_2} + \frac{u_{w_2} + \ell_{w_2}}2 = -\ell_{w_2} + \delta^\mathcal{C}_{w_2}.
\end{equation*}
Thus, the simpler condition $u_{w_1} - \ell_{w_1} = u_{w_2} - \ell_{w_2}$ is sufficient in the case where $\delta = \delta^{\mathcal{C}}$ holds to show that a partition is bound-respecting for finite $u_{w_1}$ and $\ell_{w_1}$. For infinite values, the bound-respecting condition is still necessary. In particular, the above bound argumentation does not prevent that free variables may be aggregated with variables with one infinite bound, as in both cases the domain size is infinite. Thus, the domain-respecting assumption is still necessary for linear programs with infinite variable bounds.\\

Although $\delta^{\mathcal{C}}$ is a natural choice, it does have a few problematic aspects. First of all, the product $A \delta^\mathcal{C}$ that appears in the right-hand side $b- A \delta^\mathcal{C}$ of the linear program after applying the transformation $x = x' + \delta$ can be problematic. If $A$ contains large nonzero-entries and/or $\delta^\mathcal{C}$ is large due to variables with large domains, then the product $A\delta^\mathcal{C}$ may result in large numerical values in the right-hand $b- A \delta^\mathcal{C}$ in the linear program obtained from the reflection reduction. This is undesirable, as it can lead to numerical problems that slow down the linear programming solver. Thus, we generally prefer offsets $\delta$ that are sparse and/or have small norms.\\

A second issue appears when we apply reflection reductions in the context of mixed-integer linear programs: the offset $\delta^{\mathcal{C}}$ is not necessarily integer. As a result, an integrality constraint $x_w \in \Z$ may be transformed into the form $x'_{w} + \delta^\mathcal{C}_w \in \Z$, which cannot be handled directly by MILP solvers if $\delta^\mathcal{C}_w$ is not integer.
Next, we formulate a result that helps to derive a family of reflection reductions for different $\delta$-values by detecting one for a single value of $\delta$. This result will be helpful to construct an offset $\delta$ that does not cause numerical issues and that also works well for mixed-integer linear programs.\\

\begin{theorem}
    \label{thm_delta_equivalent}
    Let $(\widehat{\MP},\widehat{\MQ})$ be a symmetric equitable partition of $F_S(A,b - A \delta, \ell - \delta, u - \delta, c)$ and let $\gamma$, $\lambda$, $\MP_U$ and $\MQ_U$ be generated by $(\widehat{\MP},\widehat{\MQ})$. Let $\nu \in \R^{W_U}$ be a vector that is equitably partitioned by $\MQ_U$ and define $\delta' \coloneqq (\delta_{W_U} + \Lambda \nu,\delta_{W_B})$ such that $\ell \leq \delta' \leq u$. Then $(\widehat{\MP},\widehat{\MQ})$ is a symmetric equitable partition of $F_S(A,b- A \delta', \ell - \delta', u - \delta', c)$.
\end{theorem}
\begin{proof}
    Clearly, $(\widehat{\MP},\widehat{\MQ})$ is also a symmetric partition of $F_S(A,b- A \delta', \ell - \delta', u - \delta', c)$. Since the objective and constraint matrix are unchanged compared to $F_S(A,b- A \delta, \ell - \delta, u - \delta, c)$ and remain so by the split reformulation, it is sufficient to show that $\widehat{\MP}$ is a symmetric equitable partition of the right-hand side and that $\widehat{\MQ}$ is an equitable partition of the variable upper bounds of $F_S(A,b- A \delta', \ell - \delta', u - \delta', c)$. Note that by definition of the split reformulation, all variables have lower bound zero, so we only consider the upper bounds. \\

    Let us state the reordered linear program of $F_S(A,b - A \delta', \ell-\delta',u-\delta',c)$, which is similar to the form~\eqref{lp_centered_extended_reordered}.

\begin{subequations}
\label{lp_centered_extended_reordered_delta}
\begin{align}
    \min\, \transpose{c_{W_U}} \Lambda \widehat{x}_{\widehat{W}^\lambda_U} - \transpose{c_{W_U}} \Lambda \widehat{x}_{\widehat{W}^{-\lambda}_U} & \\
\begin{bmatrix}
    \Gamma A_{V_U,W_U} \Lambda & -\Gamma A_{V_U,W_U} \Lambda & \widehat{M}^1
    \\
    -\Gamma A_{V_U,W_U} \Lambda & \Gamma A_{V_U,W_U} \Lambda & -\widehat{M}^1
    \\
    \widehat{M}^2 & -\widehat{M}^2 & \widehat{M}^3 \\
\end{bmatrix}
\begin{bmatrix}
    \widehat{x}_{\widehat{W}^\lambda_U}\\
    \widehat{x}_{\widehat{W}^{-\lambda}_U}\\
    \widehat{x}_{\widehat{W}_B}
\end{bmatrix} &= \begin{bmatrix}
    \Gamma (b - A \delta')_{V_U}\\
    -\Gamma (b- A \delta')_{V_U}\\
    \widehat{b}_{\widehat{V}_B}
\end{bmatrix}
\label{lp_centered_extended_reordered_delta_eqs}
\\
\zerovec \leq \widehat{x}_{\widehat{W}^\lambda_U} &\leq \mathrm{bs}(\ell-\delta',u-\delta',\lambda)_{W_U} \\
\zerovec \leq \widehat{x}_{\widehat{W}^{-\lambda}_U} &\leq \mathrm{bs}(\ell-\delta',u-\delta',-\lambda)_{W_U} \\
\zerovec \leq \widehat{x}_{\widehat{W}_B} &\leq \begin{bmatrix}
    -(\ell-\delta')_{W_B}\\
    (u-\delta')_{W_B}
\end{bmatrix} 
\end{align}    
\end{subequations}
We use $\widehat{b}_{\widehat{V}_B} \coloneqq \begin{bmatrix}
    (b - A \delta')_{V_B}\\
    -(b - A \delta')_{V_B}\\
\end{bmatrix}$ to denote the variable bounds of the bipolar rows.
For the variables $\widehat{W}^{\lambda}_U$, we can compute their upper bounds using:
    \begin{align*}
        \mathrm{bs}(\ell-\delta',u-\delta',\lambda)_w &= \begin{cases}
            u_w - \delta'_w&\text{if }\lambda_w=1\\
            -\ell_w + \delta'_w&\text{if }\lambda_w=-1
        \end{cases} = \begin{cases}
            u_w - \delta_w - \lambda_w \nu_w &\text{if }\lambda_w=1\\
            -\ell_w + \delta_w + \lambda_w \nu_w &\text{if }\lambda_w=-1
        \end{cases}\\ &= - \nu_w + \begin{cases}
            u_w - \delta_w&\text{if }\lambda_w=1\\
            -\ell_w + \delta_w&\text{if }\lambda_w=-1
        \end{cases}= - \nu_w +  \mathrm{bs}(\ell-\delta,u-\delta,\lambda)_w.
    \end{align*}
    Thus, we have that $\mathrm{bs}(\ell-\delta',u-\delta',\lambda) = -\nu + \mathrm{bs}(\ell-\delta,u-\delta,\lambda)$ holds.  
    Then, since $\widehat{\MQ}^\lambda_U$ is exactly the partition given by $\MQ_U$ and $\nu$ is $\MQ_U$ equitable, it follows that it follows that $\widehat{\MQ}^\lambda_U$ is an equitable partition of $\mathrm{bs}(\ell-\delta',u-\delta',\lambda)$ since it is the sum of the equitably partitioned vectors $-\nu$ and $\mathrm{bs}(\ell-\delta,u-\delta,\lambda)$.
    For $\widehat{W}^{-\lambda}_U$ a similar argument shows that $\mathrm{bs}(\ell-\delta',u-\delta', -\lambda) = \nu + \mathrm{bs}(\ell-\delta,u-\delta,-\lambda)$ holds, which implies that the upper bounds $\mathrm{bs}(\ell-\delta',u-\delta', -\lambda)$ are equitably partitioned by $\widehat{\MQ}^\lambda_U$. For the bipolar variables $\delta'_{W_B} = \delta_{W_B}$ holds so the variable bounds are unchanged, which directly implies that $\widehat{\MQ}_B$ is an equitable partition of the variable bounds of $\widehat{W}_B$. Combining these statements, it follows that $\widehat{\MQ}$ is an equitable partition of the variable upper bounds of $F_S(A,b-A\delta',\ell-\delta',u-\delta',c)$.\\

    Next, let us show that the right-hand sides of~\eqref{lp_centered_extended_reordered_delta_eqs} are equitably partitioned by $\widehat{\MP}$. Note that for any $v\in V$ that 
    \begin{equation*}
        (b- A \delta')_v = b_v - (A (\delta + (\Lambda\nu,\zerovec_{W_B})))_v =(b- A\delta)_v - (A_{\star,W_U} \Lambda \nu)_v
    \end{equation*} holds. Thus, it follows that 
    \begin{equation*}
        \Gamma(b- A \delta')_{V_U} = \Gamma(b- A \delta - A_{\star,W_U} \Lambda \nu )_{V_U} = \Gamma(b-A \delta)_{V_U} - \Gamma A_{V_U,W_U} \Lambda \nu.
    \end{equation*} 
    The term $\Gamma(b-A \delta)_{V_U}$ is equitably partitioned by the subpartition $\widehat{\MP}^\gamma_U$ of $\widehat{\MP}$. Then, we claim that 
    $\Gamma A_{V_U,W_U} \Lambda \nu$ is also a $\widehat{\MP}^\gamma_U$-equitable vector.
    
    First, note that $(\widehat{\MP}^\gamma_U,\MQ_U)$ is an equitable partition of $\Gamma A_{V_U,W_U} \Lambda$, which follows from the fact that $(\widehat{\MP},\widehat{\MQ})$ is an equitable partition of $\widehat{M}$ in the linear program $F_S(A,b- A \delta, \ell- \delta, u - \delta,c)$, where $(\widehat{\MP}^\gamma_U,\widehat{\MQ}^\lambda_U) = (\widehat{\MP}^\gamma_U,\MQ_U)$ induces the subpartition.

    Secondly, since $\nu$ is $\MQ_U$-equitable, we claim that multiplication of any $(\widehat{\MP}^\gamma_U,\MQ_U)$-matrix with a $\MQ_U$-equitable vector produces a $\widehat{\MP}^\gamma_U$-equitable vector. In particular note that for all $P\in \widehat{\MP}^\gamma_U$ and any $v\in P$, we have:
    \begin{align*}
        (\Gamma A_{V_U,W_U} \Lambda \nu)_v &= \sum_{w\in W} (\Gamma A_{V_U,W_U} \Lambda)_{v,w} \nu_{w} = \sum_{Q\in \MQ_U} \sum_{w\in Q} (\Gamma A_{V_U,W_U} \Lambda)_{v,w} \nu_{w} \\ &\stackrel{\textup{(a)}}{=} \sum_{Q\in \MQ} \alpha_{Q} \sum_{w\in Q} (\Gamma A_{V_U,W_U} \Lambda)_{v,w} \stackrel{\textup{(b)}} = \sum_{Q\in \MQ} \alpha_{Q} \beta_{P,Q}.
    \end{align*}
    Here, \textup{(a)} follows since $\MQ_U$ is an equitable partition of $\nu$, which implies that $\nu_w = \alpha_{Q}$ for $Q\in \MQ_U$ such that $w\in Q$. Then, \textup{(b)} follows since $(\widehat{\MP}^\gamma_U,\MQ_U)$ is an equitable partition of $\Gamma A \Lambda$, which implies that condition \eqref{eq_folding_row_sums} holds. Thus, the row sums of $\Gamma A \Lambda$ are identical for all $v\in P$, which we express using $\beta_{P,Q}$. Since the final expression is independent of $v$ for every part $P\in \widehat{\MP}^\gamma_U$, this shows that $\Gamma A \Lambda \nu$ is $\widehat{\MP}^\gamma_U$-equitable. Then, since the sum of two equitable vectors is equitable, it follows that $\widehat{\MP}^\gamma_U$ is an equitable partition of $\Gamma ( b- A \delta')_{V_U} = \Gamma(b- A \delta)_{V_U} - \Gamma A_{V_U,W_U} \Lambda \nu$.
    A very similar argument shows that $\widehat{\MP}^{-\gamma}_U$ is an equitable partition of $-\Gamma ( b- A \delta')_{V_U} = -\Gamma(b- A \delta)_{V_U} + \Gamma A_{V_U,W_U} \Lambda \nu$.

    For the bipolar rows, we consider $(b-A \delta')_v$ for $v\in V_B$ such that $\{v^+,v^-\} \subseteq P$ for $P \in \widehat{\MP}_B$. Then, we have that:
    \begin{align*}
        (b- A \delta')_{v} &= (b - A \delta)_v - (A_{\star,W_U} \Gamma \nu)_v \stackrel{\textup{(c)}}= - (A_{\star,W_U} \Gamma \nu)_v = - \sum_{w\in W_U} A_{v,w}\lambda_w \nu_w \\
        &\stackrel{\textup{(d)}}= - \sum_{Q\in \MQ_U} \alpha_Q \sum_{w\in Q} A_{v,w} \lambda_w \stackrel{\textup{(e)}}= - \sum_{Q\in \MQ_U} \alpha_Q \sum_{w^\lambda \in Q^\lambda} \widehat{M}_{v^+,w^\lambda} \stackrel{\textup{(f)}}= 0 
    \end{align*}
    Here, \textup{(c)} follows since $(\widehat{\MP},\widehat{\MQ})$ being a symmetric equitable partition for $F_S(A,b- A \delta, \ell-\delta, u - \delta,c)$ implies that $(b-A \delta)_v = 0$ for all $v\in V_B$. The relation in \textup{(d)} follows since $\nu$ is an $\MQ_U$-equitable vector. In \textup{(e)}, we use that $(A \Lambda)_{v,Q} = \widehat{M}_{v^+,Q^\lambda}$ holds by construction of $\widehat{M}$. In \textup{(f)}, we use that $\sum_{w^\lambda \in Q^\lambda} \widehat{M}_{v^+,w^\lambda} = 0$ holds by \cref{thm_bipolar_symmetric_equitable_sum}, using that $v$ is bipolar and that the given partition is equitable. Then, it follows that $\widehat{b}_{\widehat{V}_B} \coloneqq \begin{bmatrix}
    (b - A \delta')_{V_B}\\
    -(b - A \delta')_{V_B}\\
\end{bmatrix} = \begin{bmatrix}
    \zerovec_{V_B}\\
    \zerovec_{V_B}
\end{bmatrix}$, which shows that $\widehat{\MP_B}$ is an equitable partition of $\widehat{b}$. Moreover, it implies that the rows $\widehat{V}_B$ remain bipolar, which implies that the symmetry of $(\widehat{\MP},\widehat{\MQ})$ is also preserved for $F_S(A,b- A \delta', \ell - \delta',u- \delta')$. Then, we have that $\widehat{\MP}$ is an equitable partition of the right-hand side of $F_S(A,b- A \delta', \ell- \delta',u-\delta',c)$. Using the earlier derived equitability for the variable bounds, we conclude that $(\widehat{\MP},\widehat{\MQ})$ is a symmetric equitable partition of $F_S(A,b- A \delta', \ell- \delta',u-\delta',c)$.
\end{proof}

Usually, MIP solvers use \emph{variable complementation}, which maps $x_w$ to $x_w - \ell_w$ or $u_w - x_w$, in the context of reflections. Since the bounds of an integer variable are integral, these transformations have the convenient property of preserving integrality.
In \cref{thm_delta_equivalent}, we showed that by computing symmetric equitable partitions for one $\delta$ value, the obtained partition is valid for a larger set of $\delta$-values. Next, we use this result to show that computing the symmetric equitable partition for $\delta^\mathcal{C}$ can be used to derive variable complementations.
\begin{corollary}
    \label{thm_delta_equivalent_complemented}
    Let $(\widehat{\MP},\widehat{\MQ})$ be a symmetric, equitable and bound-respecting partition of $F_S(A,b - A \delta^\mathcal{C}, \ell - \delta^\mathcal{C}, u - \delta^\mathcal{C}, c)$ and let $\gamma$, $\lambda$, $\MP_U$ and $\MQ_U$ be generated by $(\widehat{\MP},\widehat{\MQ})$. Then, for $\delta^\mathcal{S}$ defined via  
    \begin{equation*}
    \delta^\mathcal{S}_w \coloneqq \begin{cases}
        \ell_w &\text{if $\ell_w$ and $u_w$ are finite, $w\in W_U$ and $\lambda_w=1$.}\\
        u_w &\text{if $\ell_w$ and $u_w$ are finite, $w\in W_U$ and $\lambda_w=-1$.}\\
        \frac{\ell_w + u_w}2 &\text{if $\ell_w$ and $u_w$ are finite and $w\in W_B$.}\\
        0 &\text{if $\ell_w = -\infty$ and $u_w = \infty$.}\\
        \ell_w &\text{if $\ell_w$ is finite and $u_w = \infty$.}\\
        u_w &\text{if $\ell_w = -\infty$ and $u_w$ is finite.}
    \end{cases},
\end{equation*} $(\widehat{\MP},\widehat{\MQ})$ is a symmetric equitable partition of $F_S(A,b - A \delta^\mathcal{S}, \ell - \delta^\mathcal{S}, u - \delta^\mathcal{S}, c)$.
\end{corollary}
\begin{proof}
    Let $\nu \in \R^{W_U}$ be such that $\nu \coloneqq \begin{cases}
        -\frac{u_{w}- \ell_w}2 &\text{if $\ell_w$ and $u_w$ are finite and $w\in W_U$}\\
        0 &\text{otherwise.}
    \end{cases}$
    
    Since the symmetric equitable partition is bound-respecting, the domain size $u_w - \ell_w$ is identical for all variables $w\in Q$ for some part $Q \in \MQ$ in the case where $u_w$ and $\ell_w$ are both finite. Thus, $\nu$ is a $\MQ_U$-equitable vector.
    Using $\frac{\ell_w + u_w}2 -\frac{u_{w}- \ell_w}2 = \ell_w$ and $\frac{\ell_w + u_w}2 + \frac{u_{w}- \ell_w}2 = u_w$ one can verify that $\delta^\mathcal{S} \coloneqq \delta^\mathcal{C} + (\Lambda \nu, \zerovec_{W_B})$ holds, and the result follows by \cref{thm_delta_equivalent}.
\end{proof}
One important detail about \cref{thm_delta_equivalent_complemented} is that the variable bounds of bipolar variables are not adjusted and may still be fractional. Later on, we will see that this is not an issue in the context of Mixed-Integer Linear Programming as the bipolar variables are eliminated from the reduced model.

\subsection{Comparison with reflection symmetries}
Next, we will show that reflection reductions generalize reflection symmetries. First let us define reflection symmetries more formally. A \emph{permutation matrix} $B$ is a square binary matrix with exactly one nonzero element with value $1$ in each row and column. It is well known that the inverse $B^{-1}$ is again a permutation matrix, which corresponds to the inverse permutation of the permutation given by $B$. A \emph{signed permutation matrix} is a square $\{-1,0,1\}$-matrix with exactly one nonzero element in each row and column. Note that each signed permutation matrix $B$ can be decomposed as $\diag(s) C$ or $C \diag(s')$ where $C$ is the nonzero support of $B$, which is a permutation matrix, and $s$ and $s'$ are suitably chosen $\{-1,1\}$ vectors.
In this section, we do not allow for arbitrary reordering of matrix columns and rows. In order to clarify the effect of the permutation symmetries, we denote each permutation explicitly using a permutation matrix.

\begin{definition}
Consider a linear program $F(A,b,\ell,u,c)$, with $A\in \R^{V\times W}$. Let $\gamma \in \{-1,1\}^V$, $\lambda \in \{-1,1\}^W$ and define $\Gamma \coloneqq\diag(\gamma)$ and $\Lambda \coloneqq \diag(\lambda)$, and let $C\in \{0,1\}^{V\times V}$ and $D\in \{0,1\}^{W\times W}$ be permutation matrices. Then, $(\gamma,\lambda,C,D)$ is a \emph{reflection symmetry of $F(A,b,\ell,u,c)$} if the following hold:
\begin{enumerate}[label=(\roman*)]
    \item $C\Gamma A \Lambda D = A$
    \item $C \Gamma b = b$
    \item $D^{-1} \lambda c = c$
    \item $-D^{-1}\mathrm{bs}(\ell,u,-\lambda)  = \ell$ and $D^{-1}\mathrm{bs}(u,\ell,\lambda)= u$
\end{enumerate}
Furthermore, we say that $(C,D)$ is a \emph{permutation symmetry} of $F(A,b,\ell,u,c)$ if $(\onevec,\onevec,C,D)$ is a reflection symmetry of $F(A,b,\ell,u,c)$.
\end{definition}

\begin{theorem}
    \label{thm_reflection_symmetry_permutation_symmetry}
    Every reflection symmetry of $F(A,b,\ell,u,c)$ induces a permutation symmetry of $F_S(A,b,\ell,u,c)$.
\end{theorem}
\begin{proof}
    Let $(\gamma,\lambda,C,D)$ be a reflection symmetry of $F(A,b,\ell,u,c)$. Then, consider the split reformulation $F_S(A,b,\ell,u,c)$. 
    We argued earlier that for the split permutation there exist permutation matrices $\widehat{C}_1$ and $\widehat{D}_1$ whose application to $F_S(A,b,\ell,u,c)$ yields the following linear program, see e.g.~\eqref{lp_centered_extended_reordered}. Here, $\widehat{C}_1$ corresponds to the permutation that swaps $v^+$ and $v^-$ if $\gamma_v = -1$ holds, and $\widehat{D}_1$ corresponds to the permutation that swaps $w^+$ and $w^-$ if $\lambda_w = -1$ holds.
    \begin{subequations}
    \label{eq_reflection_symmetry_reordered}
    \begin{align}
        \min \transpose{(\Lambda c)} \widehat{x}^\lambda - \transpose{(\Lambda c)} \widehat{x}^{-\lambda}&\\
        \begin{bmatrix}
            \Gamma A \Lambda & - \Gamma A \Lambda\\
            -\Gamma A \Lambda & \Gamma A \Lambda
        \end{bmatrix} \begin{bmatrix}
            \widehat{x}^\lambda\\
            \widehat{x}^{-\lambda}
        \end{bmatrix} &= \begin{bmatrix}
            \Gamma b\\
            - \Gamma b
        \end{bmatrix}\\
        \zerovec \leq \begin{bmatrix}
            \widehat{x}^\lambda\\
            \widehat{x}^{-\lambda}
        \end{bmatrix} & \leq \begin{bmatrix}
            \mathrm{bs(\ell,u,\lambda)}\\
            \mathrm{bs(\ell,u,-\lambda)}
        \end{bmatrix} 
    \end{align}
    \end{subequations}
    Then, we define permutation matrices $\widehat{C}_2 \coloneqq \begin{bmatrix}
        C & \zerovec\\
        \zerovec & C
    \end{bmatrix}$ and $\widehat{D}_2 \coloneqq \begin{bmatrix}
        D & \zerovec\\
        \zerovec & D
    \end{bmatrix}$ and apply them to~\eqref{eq_reflection_symmetry_reordered}. Doing so, we obtain the linear program:
    \begin{subequations}    \label{eq_reflection_symmetry_permutation_symmetry}
    \begin{align}
        \min \transpose{(D^{-1} \Lambda c)} \widehat{x}^\lambda - \transpose{(D^{-1} \Lambda c)} \widehat{x}^{-\lambda}&\\
        \begin{bmatrix}
            C \Gamma A \Lambda D & - C\Gamma A \Lambda D\\
            -C\Gamma A \Lambda D & C\Gamma A \Lambda D
        \end{bmatrix} \begin{bmatrix}
            \widehat{x}^\lambda\\
            \widehat{x}^{-\lambda}
        \end{bmatrix} &= \begin{bmatrix}
            C\Gamma b\\
            - C\Gamma b
        \end{bmatrix}\\
        \zerovec \leq \begin{bmatrix}
            \widehat{x}^\lambda\\
            \widehat{x}^{-\lambda}
        \end{bmatrix} & \leq \begin{bmatrix}
            D^{-1} \mathrm{bs(\ell,u,\lambda)}\\
            D^{-1}\mathrm{bs(\ell,u,-\lambda)}
        \end{bmatrix} 
    \end{align}
    \end{subequations}
    
    Then, since $(\gamma,\lambda,C,D)$ is a reflection symmetry, it follows by substitution that $C \Gamma A \Lambda D = A$, $C \Gamma b = b$,
    $D^{-1} \mathrm{bs}(\ell,u,\lambda) = u$, $D^{-1} \mathrm{bs}(\ell,u,-\lambda) = - \ell$. Using these substitutions,~\eqref{eq_reflection_symmetry_permutation_symmetry} is identical to $F_S(A,b,\ell,u,c,)$, without reordering the indices. Thus, the permutation matrices $\widehat{C} \coloneqq \widehat{C}_2 \widehat{C}_1$ and $\widehat{D} \coloneqq \widehat{D}_1 \widehat{D}_2$ form a permutation symmetry $(\widehat{C},\widehat{D})$ of $F_S(A,b,\ell,u,c)$.
\end{proof}

By~\cref{thm_reflection_symmetry_permutation_symmetry}, the reflection symmetries of a linear program are captured completely by the permutation symmetries of its split reformulation. 
In section 7.2 in~\cite{Grohe2014}, it is argued that for any linear program, performing the reduction from \cref{thm_lp_folding} using equitable partitions generalizes the projection to the fixed space of the symmetry group formed by the permutation symmetries.
Although we do not show it formally, the permutation symmetries that we exhibit in the proof of~\cref{thm_reflection_symmetry_permutation_symmetry} are bound-respecting, as they act similarly on the positive and negative variables of the split reformulation. As a consequence, the symmetry subgroup of the split reformulation corresponding to any permutation symmetries derived from reflection symmetries generates a bound-respecting equitable partition of the split reformulation. 
Thus, by computing bound-respecting equitable partitions of the split reformulation, our approach generalizes the projection to the fixed space of the symmetry group formed by reflections symmetries.\\

As is also remarked in~\cite{Grohe2014}, equitable partitions may also occur even if there is no symmetry. Next, we show an example where this occurs for the split reformulation of a general class of matrices that does not necessarily contain symmetries.
\begin{example}
    Let $A\in \{0,1\}^{m\times n}$ be a matrix with two non-zero entries in each row. The matrix $A$ does not necessarily contain any symmetries, and we can consider the following linear program with no objective.
    \begin{align*}
        A x &= \onevec\\
        \zerovec \leq x &\leq \onevec
    \end{align*}
    Then, centering the linear program using $\delta_w \coloneqq \frac{1}{2}$ for all $w\in W$ yields the following linear program.
    \begin{align*}
        A x &= \zerovec\\
        -\frac{1}{2} \leq x_w &\leq \frac{1}{2} & \forall w\in W
    \end{align*}
    Its split reformulation is given by 
    \begin{align*}
        A x^+ - A x^- &= \zerovec\\
        -A x^+ + A x^- &= \zerovec\\
        0\leq x^+_w &\leq \frac{1}{2} & \forall w\in W\\
        0\leq x^-_w &\leq \frac{1}{2} & \forall w\in W
    \end{align*}
    Then, since the objective and right-hand sides are all zero and the variable bounds are identical, the coarsest equitable partition of this linear program has a single bipolar variable part and a single bipolar row part, even though the original LP does not in general contain any symmetries.
\end{example}

\subsection{Detecting reflection reductions}
\label{sec_lp_algorithm}
In order to detect reflection reductions, it suffices to compute symmetric equitable partitions of the split reformulation form by \cref{thm_refl_reduction}. Thus, we can directly reuse the algorithm formulated in~\cite{Grohe2014}, which relies on the fast color refinement algorithm by Paige and Tarjan~\cite{Paige1987}. Since the split reformulation consists of four copies of the constraint matrix, we can detect symmetric equitable partitions, and thus reflection reductions, in identical asymptotic worst-case time and space complexity as in~\cite{Grohe2014}.

\begin{theorem}
\label{thm_running_time_extended} For a linear program $F(A,b,\ell,u,c)$ with $A\in \R^{V\times W}$, let $m$ be the total bitlength of the entries of $\begin{bmatrix}
    \transpose{\ell} & 0 \\
    \transpose{u} & 0 \\
    \transpose{c} & 0\\
    A & b
\end{bmatrix}$ and let $n = |V| + |W|$. For $\delta$ with total bitlength $\orderO(m)$, a reflection reduction $(\MP,\MQ,\gamma,\lambda,\delta,V_U,W_U)$ of $F(A,b,\ell,u,c)$ can be determined in $\orderO((n+m)\log n)$ time and $\orderO(n+m)$ space.
\end{theorem}
\begin{proof}
    We assume that the constraint matrix is given in a sparse format such as the CSR and/or CSC format.
    First, $F(A,b- A \delta,\ell- \delta,u-\delta,c)$ can be computed in $\orderO(n+m)$ time and space and has bitlength $\orderO(m)$, where we use that $\delta$ has bitlength of order $\orderO(m)$. Then, the split reformulation $F_S(A,b- A \delta,\ell- \delta,u-\delta,c)$ can also be computed in $\orderO(n+m)$ time and space and has bitlength of order $\orderO(m)$. 
    The initial symmetric equitable partition can be determined in and $\orderO(n+m)$ time and $\orderO(n+m)$ space by using standard hashing techniques.
    Then, computing the coarsest symmetric equitable partition of the split reformulation takes $\orderO((n+m)\log n)$ time and $\orderO(n+m)$ space~\cite[Theorem 3.1]{Grohe2014}.
    Using the symmetric equitable partition, we can directly derive a reflection reduction in $\orderO(n+m)$ time and space. Summing up all steps, we obtain a worst-case time complexity of $\orderO((n+m)\log n)$ time and a worst-case space complexity of $\orderO(n+m)$. 
\end{proof}

In \cref{thm_running_time_extended}, we assumed that the offset $\delta$ has bitlength that is of the same order as $F(A,b,\ell,u,c)$. This is a reasonable choice, as both $\delta^{\mathcal{C}}$ and $\delta^\mathcal{S}$ have encoding lengths bounded by the encoding length of $\ell$ and $u$ by their definition.\\

Although the running time is asymptotically equivalent to the `ordinary' equitable partition algorithm presented in~\cite{Grohe2014}, computing equitable partitions of the split reformulation, which has 4 times as many nonzero entries, may still be undesirable in practice. 
 To reduce the running time by a constant factor one can exploit the structure of symmetric equitable partitions. We formulate a refinement algorithm that computes $\gamma$ and $\lambda$ that runs on the original matrix but models the refinement algorithm that computes a symmetric partition on the split reformulation. This closely follows the argumentation presented in \cref{thm_initial_partition_symmetric} and \cref{thm_refinement_symmetric}. 
The core idea to achieve this constant-factor improvement is to refine unipolar part pairs at the same time, as is also done in the proof of \cref{thm_refinement_symmetric}.
The implemented algorithm is very similar to that in~\cite{Grohe2014}, who use the algorithm in~\cite{Berkholz2017} which provides a thorough description of the fast color refinement algorithm. Since the algorithm we present does not have a better asymptotic running time, we outline the main changes and adaptations compared to the description in~\cite{Grohe2014} here without proof.\\

For the data structures, we add arrays to track $\lambda$, $\gamma$ and the type (unipolar or bipolar) of each part, which takes $\orderO(n)$ space. The signs $\lambda_w$ and $\gamma_w$ are used during the computation of the column and row sums for equitability.
Then, $\lambda$ and $\gamma$ are determined whenever a bipolar part $P$ is turned into a unipolar one or at initialization. 

Initially, we compute the affine transformation given by $\delta^\mathcal{C}$. Then, we compute an initial symmetric bound-respecting partition of the split reformulation as in \cref{thm_initial_partition_symmetric}, which can be achieved in linear time by using hashing techniques.
For a row $v\in V$, we initialize it to be bipolar if $b_v = 0$ holds. Otherwise, the row is marked as unipolar and grouped with other rows based on the absolute values $|b_v|$ of the right-hand side.

For a variable $w\in W$, we initialize them it to be bipolar if $c_w = 0$ and $-\ell_w = u_w$ hold. Otherwise, we let them be unipolar, and compute signs $\lambda_w$ that so that variables whose objective and bounds can be mapped to each other by negation are put into the same initial part.

For the subsequent refinement of the constraint matrix we change the behavior slightly based on the type of each part. The changes are similar to those described in the proof of \cref{thm_refinement_symmetric}. When refined, unipolar parts retain their type, but bipolar parts may be refined into smaller parts that may be either unipolar or bipolar. When a bipolar part is refined into a unipolar part, the signs $\lambda$ or $\gamma$ are determined from its elements based on the column or row sums. In our refinement algorithm, we partition the columns (rows) of a bipolar part based on the absolute value of the sums, rather than the sums directly as in the refinement algorithm in~\cite{Grohe2014}. Then, we set the signs $\lambda_w$ so that the column sums are all positive when multiplied by $\lambda_w$. Finally, one improvement that is directly suggested by \cref{thm_bipolar_symmetric_equitable_sum} is that if $P$ is a bipolar type, that then $\widehat{M}_{P,Q}$ is an equitable block for any $Q\in \MQ$. Thus, we do not need to add bipolar parts to the search stack, since checking them for refinements is pointless.\\

There are two cases where one needs to be careful with the adapation of the color refinement algorithm to the reduced matrix size. First, if a bipolar part contains row or columns sums with identical but nonzero values, then it must be converted to a unipolar part and added to the search stack. 
Second, we note that one has to be careful with the `leave largest refined part out'-rule that is used in most color refinement algorithms and is key to the fast running time of color refinement~\cite{Hopcroft1971}. This rule pushes all but the largest part obtained by a refinement of a part to the stack of parts for which refinements need to be checked.
Since we are modeling the refinement of symmetric partitions of the split reformulation, we cannot leave out a unipolar part in the algorithm that runs on the original matrix since it represents two parts in the split reformulation, and leaving out both parts would break correctness. Luckily, this does not lead to a slower running time, since the running time is still guaranteed by \cref{thm_running_time_extended} as we maintain the invariant that our partition models a symmetric partition of the split reformulation.
\paragraph{Choosing \texorpdfstring{$\gamma$}{gamma}, \texorpdfstring{$\lambda$}{lambda} and \texorpdfstring{$\delta$}{delta}}
One detail that is important for a practical implementation is the choice of $\delta$. If $\delta$ has many nonzero entries and the matrix $A$ has large values, then the product $A \delta$ that appears in the right-hand side of the reflection reduced model can become very large and lead to numerical issues.

Recall that the $\gamma$ and $\lambda$ vectors that were generated by the symmetric equitable partition were not unique, and could be optionally negated for each part $P\in \MP_U$ and $Q\in \MQ_U$. In order to minimize the number of complemented rows, we choose $\gamma$ so that for each $P\in \MP_U$ at least $\frac{|P|}2$ rows $v\in P$ have value $\gamma_{v} = 1$. Similarly, we minimize the number of complemented variables by picking $\lambda$ such that for each $Q\in \MQ_U$ at least $\frac{|Q|}{2}$ variables $w\in Q$ have value $\lambda_w = 1$.

As a starting point for our $\delta$ vector we consider $\delta^{\mathcal{S}}$ as defined in \cref{thm_delta_equivalent_complemented}. However, this vector may still not be very sparse due to the $\ell_w$ and $u_w$ values.
By \cref{thm_delta_equivalent}, we may modify $\delta^\mathcal{S}$ to $\delta \coloneqq \delta^\mathcal{S} + (\Lambda \nu,\zerovec)$ for any $\MQ_U$-equitable vector $\nu$. In our implementation, we choose $\nu$ so that $\delta$ becomes sparse. For every part $Q\in \MQ_U$, we compute how often the value $\lambda_w \delta^\mathcal{S}_w$ appears for all $w\in Q$ and set $\nu_Q$ to be the negative of the most frequently appearing value. This way, most variables $w\in Q$ satisfy 
\begin{equation*}
    \delta_w = \delta^\mathcal{S}_w + \lambda_w \nu_w = \delta^\mathcal{S}_w + \lambda_w (- \lambda_w \delta^\mathcal{S}_w) = \delta^\mathcal{S}_w -  \delta^\mathcal{S}_w = 0.
\end{equation*}
Thus, the resulting $\delta$ vector is typically quite sparse. Furthermore, we remark that for a unipolar variable $w\in W$ with integral bounds, $\delta^\mathcal{S}_w$ is also integral. One important consequence is that if all variables in some part $Q\in \MQ_U$ have integral bounds, then $\delta_Q$ is also integral. Although we cannot modify  $\delta^\mathcal{S}_w$ for variables $w$ in bipolar parts, these values do not lead to numerical issues as fixing the corresponding values does not affect the right-hand sides since the corresponding parts have zero sums due to bipolarity.

\section{Dimension Reduction for Mixed-Integer Linear Programs}
\label{sec_dimred_milp}

Throughout this section we consider a mixed-integer linear program given by a linear program $F(A,b,\ell,u,c)$ with $A\in\R^{V\times W}$, $b\in \R^V$, $\ell, u, c \in \R^W$ and integer variables $W_I \subseteq W$. For any partition $\MQ$ of $W$ that is compatible with $W_I$, we use $\MQ_I \coloneqq \{Q \in \MQ \mid Q \subseteq W_I\}$ to denote the corresponding partition over the integer variables.\\

Let us consider the effect of DRCR in context of mixed-integer linear programs.
For the mixed-integer linear program given by $F(A,b,\ell,u,c)$ and integrality constraints $x \in \M^{W_I}$, we consider the \emph{reduced problem} $R^{\MP,\MQ}(F(A,b,\ell,u,c))$ with integrality constraints $y\in \M^{\MQ_I}$ for $y\in R^{\MP,\MQ}(F(A,b,\ell,u,c))$. 
For the linear programming relaxations,~\cref{thm_lp_folding} establishes a correspondence between the original and the reduced problem. Let us consider how these mappings interact with integrality. For a solution $x\in F(A,b,\ell,u,c)\cap \M^{W_I}$, we can observe that $y \coloneqq \transpose{\Pi_{\MQ}} x$ satisfies the integrality constraints for the reduced problem, since then $y_Q \coloneqq \sum_{w\in Q} x_w$ is integral for any $Q\in \MQ_I$ as it is the sum of integral values. Linear feasibility and optimality are guaranteed already by~\cref{thm_lp_folding}.
Thus, every mixed-integer solution for the original problem maps to a mixed-integer solution of the reduced problem. However, for the other direction integrality is not preserved, since for $y\in R^{\MP,\MQ}(F(A,b,\ell,u,c))\cap \M^{\MQ_I}$ the original solution is given for $w\in W$ with $w\in Q$ by $x_w \coloneqq \frac{1}{|Q|} y_Q$, which is not necessarily integral. In particular, this is an issue if the reduced problem contains integer solutions $y\in R^{\MP,\MQ}(F(A,b,\ell,u,c)) \cap \M^{Q_I}$ that do not map to integer solutions of the original problem.
Previous works on Orbital Shrinking~\cite{Fischetti2012,Fischetti2017,Salvagnin2013} deal with this limitation by solving a feasibility subproblem, which checks if the provided mixed-integer solution for the reduced problem can be mapped to a mixed-integer solution of the original problem. For any fixed mixed-integer $y$-solution, these subproblems can be formulated by adding the constraints $ \transpose{\Pi_\MQ} x = y$ to the original MILP. 
This subproblem is a mixed-integer linear program of reduced size with no objective, and the DRCR procedures defines a point that is feasible for its LP relaxation. 
Frequently, this subproblem is solved using techniques such as constraint programming that exploit the specific structure of the subproblem~\cite{Salvagnin2012, Salvagnin2013}.\\

In this work, we take a different approach compared to previous works on Orbital Shrinking by characterizing when each solution of the reduced problem corresponds to a solution of the original model. In terms of orbital shrinking, one can think of our approach as \emph{Exact Orbital Shrinking}, where the orbital shrinking subproblem is given by a perfect formulation of a mixed-integer linear program. In this case, the orbital shrinking subproblem can be solved via linear programming. One of the main advantages of our approach is that it guarantees a tight correspondence between the mixed-integer solutions of the original and the reduced problem, and that the subproblems can be solved easily via linear programming. This correspondence comes at a cost: the main downside of our approach is that it cannot handle all symmetries. Due to the requirement that the subproblem is a perfect formulation, our approach cannot in general handle the complete symmetry group of the MILP, and instead only handles a subgroup of the symmetry group that generates subproblems that are perfect formulations.\\

In \cref{thm_milp_folding}, we characterize when reflection reductions, which generalize the DRCR procedure, can be used to reduce the dimension of a mixed-integer linear program.
\begin{theorem}
    \label{thm_milp_folding}
    Let $(\MP,\MQ,\gamma,\lambda,\delta,V_U,W_U)$ be a reflection reduction of $F(A,b,\ell,u,c)$ such that $\delta\in \M^{W_I\cap W_U}$ holds and $W_I \cap W_U$ is compatible with $\MQ$. We formulate the reduced mixed integer program using $A'\coloneqq \transpose{\Pi_{\MP}}\Gamma A_{V_U,W_U}\Lambda \widetilde{\Pi_{\MQ}}$, $b'\coloneqq \transpose{\Pi_{\MP}}\Gamma(b_{V_U} - A_{V_U,W_U} \delta_{W_U})$, $\ell' \coloneqq \transpose{\Pi_{\MQ}} \ell^{\lambda,\delta}_{W_U}$, $u'\coloneqq  \transpose{\Pi_{\MQ}} u^{\lambda,\delta}_{W_U}$ and $c'\coloneqq \transpose{\widetilde{\Pi_\MQ}} \Lambda c_{W_U}$.
    If the following equation holds
    \begin{equation}
        \label{thm_milp_folding_condition}
        \conv(\{x\in P(A,b,\ell,u) \mid (\transpose{\Pi_{\MQ}} \Lambda x)_{\MQ_I} \in \Z^{\MQ_I}\}) = \conv(P(A,b,\ell,u)\cap \M^{W_I})
    \end{equation}
    then the following hold:
    \begin{enumerate}[label=(\roman*)]
        \item \label{thm_milp_folding_to_red} For any $x\in P(A,b,\ell,u)\cap \M^{W_I}$, then $y\coloneqq \transpose{\Pi_{\MQ}} \Lambda (x_{W_U}-\delta_{W_U}) \in P(A',b',\ell',u')\cap \M^{\MQ_{I}}$ and $\transpose{c} x = \transpose{c'} y + \transpose{c} \delta$ hold.
        \item \label{thm_milp_folding_from_red} For any $y\in P(A',b',\ell',u') \cap \M^{\MQ_{I}}$ there exists a solution $x\in P(A,b,\ell,u) \cap \M^{W_I}$ such that $\transpose{c} x \leq \transpose{c'} y + \transpose{c} \delta$ hold. 
    \end{enumerate}    
\end{theorem}
\begin{proof}
    First, consider the proof of the first point. By \cref{def_refeqpart}\ref{def_refeqpart_to_red}, it follows that $y\in P(A',b',\ell',u')$ and $\transpose{c} x = \transpose{c'} y + \transpose{c} \delta$ hold. Thus, it is sufficient to show that $y\in \M^{\MQ_I}$ holds. First, since $x\in \M^{W_I}$ and $\delta\in \M^{W_I\cap W_U}$ hold we have that $x_{W_U} - \delta_{W_U} \in \M^{W_I\cap W_U}$. Then, for any $Q\in \MQ_I$, we have that $y_Q = (\transpose{\Pi_{\MQ}} \Lambda (x_{W_U} - \delta_{W_U}))_Q = \sum_{w\in Q} \lambda_w (x_w - \delta_w)$. Since $Q\in \MQ_I$ holds, it follows that all $w\in Q$ satisfy $w\in W_I$. Thus, this implies that $x_w$ and $\delta_w$ must be integral by the given conditions. Then, since $\lambda_w$ is also integral, it follows that $y_Q$ is integral. Since this holds for any $Q\in \MQ_I$, we have $y\in \M^{\MQ_I}$.

    We proceed with the proof of the second point. Let $y\in P(A',b',\ell',u')\cap \M^{\MQ_I}$ be given. Then, by \cref{def_refeqpart}\ref{def_refeqpart_from_red} we have for $x' \coloneqq (\Lambda \widetilde{\Pi_\MQ} y, \zerovec) + \delta$ that $x'\in P(A,b,\ell,u)$ and $\transpose{c'} y +\transpose{c} \delta = \transpose{c} x'$ hold. 
    Then, we observe that for each $Q\in \MQ$, we have the following for $\transpose{\Pi_{\MQ}} \Lambda x'$: 
    \begin{align*}
        (\transpose{\Pi_{\MQ}} \Lambda x')_Q = \sum_{w\in Q} \lambda_w x'_w &= \sum_{w\in Q} \lambda_w (\frac{\lambda_w y_Q}{|Q|} + \delta_w) = \sum_{w\in Q} \frac{y_Q}{|Q|} + \lambda_w \delta_w = y_Q + \sum_{w\in Q} \lambda_w \delta_w.
    \end{align*}
    Then, for $Q\in \MQ_I$, $y_Q$ is integral by assumption, and all $w\in Q$ lie in $W_I$ by definition, which shows that $\sum_{w\in Q} \lambda_w \delta_w$ is integral since $\delta\in \M^{W_I\cap W_U}$ holds. Thus, we have that $x'\in \conv(\{x\in P(A,b,\ell,u) \mid (\transpose{\Pi_{\MQ}} \Lambda x)_{\MQ_I} \in \Z^{\MQ_I}\})$ holds. By the equality in \eqref{thm_milp_folding_condition} we have that $x'\in \conv(P(A,b,\ell,u)\cap \M^{W_I})$ holds. Since $\conv(P(A,b,\ell,u)\cap \M^{W_I})$ is a convex polyhedron, there exists some $x\in P(A,b,\ell,u)\cap \M^{W_I}$ such that $\transpose{c} x \leq \transpose{c} x'$ holds. Then $\transpose{c} x \leq \transpose{c'} y +\transpose{c} \delta$ follows, which completes the proof.
\end{proof}

\Cref{thm_milp_folding} differs from reflection reductions in a few subtle but important ways. First of all, the new and crucial condition \eqref{thm_milp_folding_condition} in \cref{thm_milp_folding} states for each part $Q\in \MQ_I$ that the integrality of all variables in $Q$ must be implied by the integrality of the sum $\sum_{w\in Q} \lambda_w x_w \in \Z$. We will later elaborate on a few different methods to detect this fact. 
Additionally, note that the partition $\MQ$ may not have parts that contain both integer and continuous variables and that $\delta_{W_I\cap W_U}$ must be integral.

In the second point of \cref{thm_milp_folding} we only show existence of a better solution in the original problem, rather than proving that a direct mapping exists, as was done in \cref{thm_lp_folding}. This is due to the fact that the mapped solution, $x'$ in the proof, does not necessarily satisfy all the integrality constraints for $W_I$, even though it is feasible for the linear relaxation. This raises an important issue; it is not yet clear how, given arbitrary $x'\in P(A',b',\ell',u')$, if one can obtain in polynomial time a solution solution $x$ to the original problem that also satisfies the integrality constraints. Intuitively, one may suspect that it is sufficient to consider the affine subspace generated by $x = \transpose{\Pi_{\MQ}} \Lambda (x_{W_U} - \delta_{W_U})$ in the original problem and recover an original solution. This approach is also taken in Orbital Shrinking~\cite{Fischetti2017}. However, we show in \cref{example_fiber_nonintegral} that it may be the case that there exists no solution $x\in P(A,b,\ell,u) \cap \M^{W_I}$ such that $x' = \transpose{\Pi_{\MQ}} \Lambda (x_{W_U} - \delta_{W_U})$ holds, even though the conditions of \cref{thm_milp_folding} and $x'\in \M^{\MQ_I}$ hold. This fact is an important limitation for Orbital Shrinking~\cite{Fischetti2017} as it may be difficult to recover integral solutions for the original integer program.

\begin{example}
\label{example_fiber_nonintegral}
    Consider the integer program given by polyhedron $P = \{(x,y)\in \R^2 \mid -x + 2 y \leq 2, ~ 2x - y \leq 2, x \geq 0, y \geq 0 \}$ and $x\in \Z, y\in \Z$ and note that $P = \conv(P\cap \Z^2) = \conv(\{(x,y) \in P \mid x+y\in \Z\})$ holds and that \cref{thm_milp_folding} is applicable to the standard representation of the integer program, where $\lambda= \gamma = \onevec$, $V_U = V$ and $W_U = W$ hold. See \Cref{fig_fiber_nonintegral_example} for the geometric representation of $P$. For the reduced integer program $P' = \{z\in \R \mid z\leq 4, z\geq 0\}\cap \Z$, $z=3$ is a feasible integer solution. However, there exists no solution $(x,y)\in P$ such that $x+y=3$, $x\in \Z$ and $y\in \Z$ hold. 
\end{example}
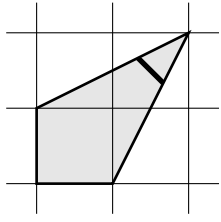
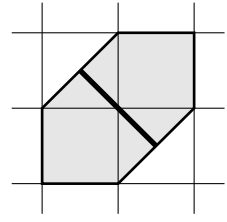
\begin{figure}[h]
    \begin{subfigure}[t]{0.45\textwidth}
    \centering
    \begin{tikzpicture}
      \draw[black, fill=black!10, line width=1pt] (0,0) -- (1,0) -- (2,2) -- (0,1) -- cycle;
        \draw[black, line width=2pt] (4/3, 5/3) -- (5/3,4/3);
      \foreach \y in {0,1,2}
      {
        \draw (-0.4,\y) -- (2.4,\y);
      }
      \foreach \x in {0,1,2}
      {
        \draw (\x,-0.4) -- (\x,2.4);
      }
    \end{tikzpicture}
    \subcaption{%
    $P$ from \cref{example_fiber_nonintegral} (thin) and its subset induced by $x+y=3$ (thick). \label{fig_fiber_nonintegral_example}
    }        
    \end{subfigure}
    \hfill
    \begin{subfigure}[t]{0.45\textwidth}
    \centering
    \begin{tikzpicture}
      \draw[black, fill=black!10, line width=1pt] (0,0) -- (1,0) -- (2,1) -- (2,2) -- (1,2) -- (0,1) -- cycle;
     \draw[black, line width=2pt] (0.5, 1.5) -- (1.5,0.5);
      \foreach \y in {0,1,2}
      {
        \draw (-0.4,\y) -- (2.4,\y);
      }
      \foreach \x in {0,1,2}
      {
        \draw (\x,-0.4) -- (\x,2.4);
      }
    \end{tikzpicture}
    \subcaption{%
    A symmetric integral polyhedron $P$ (thin) with a totally unimodular constraint matrix whose $x$ and $y$ fibers are integral, but for which the fiber defined by $x+y=2$ (thick) is not integral. \label{fig_nonintegral_affine_fiber_integral_fibers}
    }        
    \end{subfigure}
    \caption{Examples of symmetric integral polyhedra with non-integral affine fibers}
\end{figure}

Finally, \cref{thm_milp_folding} clarifies the need to prove the variant of \cref{thm_lp_folding} compared to that in~\cite{Grohe2014}. In~\cite{Grohe2014}, the mapping from the original to the reduced problem is scaled, which makes it so that integrality of the original solution does not imply integrality of the mapped solution in the reduced problem. Instead, we require the scaling as in \cref{thm_lp_folding} in \cref{thm_milp_folding} to ensure that we can map integral solutions between the original and the reduced problem. A similar scaling is also used in Orbital Shrinking~\cite{Fischetti2017}.\\

Next, we will consider several methods to detect that \eqref{thm_milp_folding_condition} is satisfied. For each method, we additionally show that a solution to the original problem can be recovered in polynomial time from a solution to the reduced problem.

First, we show in \cref{thm_integer_unit_partition} that an equitable partition where each integer variable is contained in its own part of size $1$ satisfies \eqref{thm_milp_folding_condition}. This result seems to already be known as a folklore result within the integer programming solver community\footnote{Developers of both Gurobi and FICO Xpress mentioned this idea in personal communication.}, but to the best of our knowledge it was never documented in public. This idea also extends naturally to reflection reductions.

\begin{lemma}    \label{thm_integer_unit_partition}
    If $\delta\in \M^{W_I}$ and if for each $w\in W_I$, $\{w\}\in \MQ$ holds, then \eqref{thm_milp_folding_condition} is satisfied. Additionally, for every $y\in P(A',b',\ell',u')\cap \M^{\MQ_I}$, $x' \coloneqq (\Lambda \widetilde{\Pi_\MQ} y, \zerovec)  + \delta$ lies in $P(A,b,\ell, u)\cap \M^{W_I}$.
\end{lemma}
\begin{proof}
    For the first point, note that $(\transpose{\Pi_{\MQ}} \Lambda x)_{\MQ_I} \in \Z^{\MQ_I} \iff (\Lambda x)_{W_I} \in \Z^{W_I} \iff x_{W_I} \in \Z^{W_I}$ holds, where the first equivalence follows since every $w\in W_I$ is a part of size one. The second equivalence holds since $\Lambda$ is a unimodular matrix. Thus, this directly establishes that \eqref{thm_milp_folding_condition} holds.
    
    For the second point, note that for all $\{w\} = Q\in \MQ_{I}$, $x_w=((\Lambda \widetilde{\Pi_\MQ} y, \zerovec)  + \delta)_w = \lambda_w y_Q + \delta_w$ holds since $|Q|=1$ holds. The latter expression is integral since $y_Q$, $\delta_w$ and $\lambda_w$ are all integral, which shows that $x\in \M^{W_I}$. By \cref{def_refeqpart}\ref{def_refeqpart_from_red}, $x\in P(A,b,\ell,u)$ holds, which completes the proof.
\end{proof}

It is implicit in \cref{thm_integer_unit_partition} that all bipolar parts must consist of continuous variables only by the condition that each $w\in W_I$ has $\{w\}\in \MQ$. This condition is not required in \cref{thm_milp_folding}.

The condition in \cref{thm_integer_unit_partition} that each integer variable is in its own partition class can be easily enforced during detection by modifying the initial partition $\MQ_0$ of the refinement algorithm. The condition also implies a strong condition on the rows by the requirement that the partition is equitable for the constraint matrix. In particular, this implies that for any $P\in \MP$ that $A_{v_1,w} = A_{v_2,w}$ holds for any $v_1,v_2\in P$ and $w\in W_I$. In other words, for any two rows $v_1,v_2$ that are in the same row partition, the rows must be identical in the column submatrix $A_{\star, W_I}$. Another perspective is that we are considering DRCR for a parametric set of LPs, whose right-hand sides depend on the fixed integer values for the $x_{W_I}$ variables. Since the right-hand sides for two rows in the same row-block need to be identical under all fixed integer values of $x_{W_I}$, their rows in $A_{\star,W_I}$ must be identical.\\

Although \cref{thm_integer_unit_partition} is very useful, the detection of \eqref{thm_milp_folding_condition} is mostly seen as an afterthought and is done by making the partition much finer. Thus, we investigate more sophisticated methods to detect \eqref{thm_milp_folding_condition} which do not require every integer variable to be in its own part, so that we may also aggregate integer variables. 
In the context of detecting equivalent mixed-integer formulations as we have in \eqref{thm_milp_folding_condition}, there are two works that are relevant to the current work. Recall that a matrix $A$ is said to be \emph{totally unimodular} if $\det(A')= \pm 1$ holds for every nonsingular square submatrix $A'$ of $A$. A famous result by Hoffman and Kruskal~\cite{HoffmanK56} shows that $\{x \in \R^n \mid A x \leq b, x\geq 0\}$ defines an integral polyhedron for all $b\in \Z^m$ if and only if $A$ is totally unimodular.
First of all, the author has previously investigated the phenomenon of \emph{implied integrality}~\cite{VanDerHulst2025}, which is said to occur when $\conv(P\cap \M^W) = \conv(P\cap \M^{W'})$ holds for some $W'\subseteq W$. There, implied integrality is detected by detecting totally unimodular submatrices within the constraint matrix.\\

In their work, Bader, Hildebrand, Weismantel and Zenklusen~\cite{BaderHWZ18} investigate methods to detect matrices $T\in \Z^{k\times n}$ where $k < n$ such that $\conv(\{x \in P \mid T x \in \Z^k\}) = \conv(P\cap \Z^n) $ holds. They investigate the case where $T$ is totally unimodular and formulate the notion of an \emph{affine totally unimodular decomposition} (affine TU decomposition) to detect such cases. 

\begin{definition}[Bader, Hildebrand, Weismantel \& Zenklusen\cite{BaderHWZ18}]
A matrix $A\in \Z^{m\times n}$ admits a \emph{$k$-row affine totally unimodular decomposition} if there exist matrices $S \in \Z^{m\times n}$, $U\in \Z^{m\times k}$ and $T\in \Z^{k\times n}$ such that $\begin{bmatrix}
    S\\
    T
\end{bmatrix}$ is totally unimodular and $A = S + U T$ holds.
\end{definition}

In particular, they show that affine TU decompositions can be used to remove and reformulate integrality constraints. 
\begin{proposition}[\cite{BaderHWZ18}]
    \label{thm_atu_ilp_integral}
    Let $A = S + UT \in \Z^{m\times n}$ with $T\in \{-1,0,1\}^{k\times n}$ be a $k$-row affine TU decomposition and $b\in \Z^m$ and $\ell\in (\Z\cup \{-\infty\})$, $u\in(\Z\cup \{\infty\}$ with $\ell \leq u$. Then for $P = \{x\in \R^n \mid~\ell \leq x \leq u, Ax \leq b\}$ it holds that $\conv(\{x\in P \mid Tx \in \Z^k\}) = \conv(P\cap \Z^n)$.
\end{proposition}
Note that the matrix $\transpose{\Pi_{\MQ}} \Lambda$ in \eqref{thm_milp_folding_condition} is totally unimodular as it contains at most one $\pm 1$ entry in every column. Thus, the notion of an affine TU decomposition is very suitable for our setting.

One idea which is central to the proof of \Cref{thm_atu_ilp_integral} and the detection of implied integrality, is to show that every \emph{fiber} defined by a linear subspace of $P$ given by $T x = d$ for some $d\in \Z^Q$ is (mixed)-integral. Although this is not a necessary condition for the possibility of a mixed-integer formulation (see \cref{fig_fiber_nonintegral_example}), it is practical to detect and exploit this condition. In our setting, imposing the requirement that the fibers have to be integral has the additional benefit that an integral solution to the original problem may be recovered in polynomial time, avoiding the issue in \Cref{fig_fiber_nonintegral_example}. We adapt  \cref{thm_affine_fiber_integral} from the proof of \cref{thm_atu_ilp_integral} in \cite{BaderHWZ18} to work in a slightly more general mixed-integer setting.

\begin{lemma}[Generalized from \cite{BaderHWZ18}]
\label{thm_affine_fiber_integral}
Let $P \subseteq \R^W$ be a rational polyhedron, let $W_I\subseteq W$ denote the integer variables. Let $T\in \Z^{k\times W_I}$ be any integral matrix and consider the affine fiber $H_d \coloneqq \{x \in P \mid T x_{W_I} = d\}$ for any fixed $d\in \Z^k$. If $H_d$ is $W_I$-integral for every $d\in \Z^k$, then $\conv(\{x \in P \mid T x_{W_I} \in \Z^k\}) = \conv(P \cap \M^{W_I})$ holds.
\end{lemma}
\begin{proof}
    Let $G \coloneqq \conv(\{ x\in P \mid Tx_{W_I} \in \Z^{k}\})$. First, we show that $G$ is a rational polyhedron. This is relatively easy to see, since it is the projection of the rational mixed integer program given by $G'\coloneqq \conv(\{(x,d) \in \M^{W_I} \times \Z^k \mid x \in P,~ Tx_{W_I} = d\})$. By Meyer's theorem~\cite{Meyer74}, $G'$ is a rational polyhedron. Then, since the projection of a rational polyhedron is again a polyhedron and $G = \proj_{x}(G')$, $G$ is also a rational polyhedron.

    Then the following holds:
    \begin{align*}
        G &=\conv(\{ x\in P \mid Tx_{W_I} \in \Z^k\})= \conv(\bigcup_{d\in \Z^k} H_d) \stackrel{\textup{(a)}} =  \conv(\bigcup_{d\in \Z^k} \conv(H_d\cap \M^{W_I}))\\
        &= \conv(\bigcup_{d\in \Z^k} H_d \cap \M^{W_I})
        \stackrel{\textup{(b)}}= \conv( P \cap \M^{W_I}).
    \end{align*} The equality in \textup{(a)} follows from the fact that $H_d$ is $W_I$-integral. For \textup{(b)} note that $T$ is an integral matrix, which implies for all $x\in \M^{W_I}$ that $T x_{W_I} = d$ holds for some $d\in \Z^k$. Then, the equality $\bigcup_{d\in \Z^k}H_d \cap \M^{W_I} = P \cap \M^{W_I}$ holds since any $x\in P\cap \M^{W_I}$ must be contained in $H_d\cap \M^{W_I}$ for some $d\in \Z^k$, which implies that $P\cap \M^{W_I} \subseteq \bigcup_{d \in \Z^k} H_d\cap \M^{W_I}$ holds. The reverse inclusion follows easily since $H_d \subseteq P$ holds.
\end{proof}

Then, using \cref{thm_affine_fiber_integral}, we show that one can recover solutions to the original problem from solutions to the reduced problem in polynomial time.
\begin{theorem} \label{thm_fibers_symmetry_integral} Consider the setting of \cref{thm_milp_folding} and let $\MQ'\subseteq \MQ_I$.
    If for each $d \in \Z^{\MQ'}$,  $H_d \coloneqq \{x\in P(A,b,\ell,u) \mid 
    (\transpose{\Pi_\MQ} \Lambda x)_{\MQ'} = d\}$ is an $\M^{W_I}$-integral polyhedron then \eqref{thm_milp_folding_condition} holds and $H'_{d'} \coloneqq \{y\in P(A',b',\ell',u') \mid y_{\MQ'} = d'\}$ is a $\MQ_I$-integral polyhedron for each $d'\in Z^{\MQ'}$. Moreover, for every $y\in P(A',b',\ell',u')\cap \M^{Q'}$ there exists a feasible solution $x \in H_{d^y}\cap \M^{W_I}$ with $d^y \coloneqq (y + \transpose{\Pi_\MQ} \Lambda \delta_{W_U})_{\MQ'}$
    that can be found in polynomial time, where $\transpose{c} x \leq \transpose{c'} y + \transpose{c} \delta$ holds.
\end{theorem}
\begin{proof}
    For $G=P(A,b,\ell,u)$, we have that
    \begin{equation*}
        \{x\in G \mid (\transpose{\Pi_\MQ} \Lambda x)_{\MQ'} \in \Z^{\MQ'}\} \subseteq \{x\in G \mid (\transpose{\Pi_\MQ} \Lambda x)_{\MQ_I} \in \Z^{\MQ_I}\} \subseteq G \cap \M^{W_I},
    \end{equation*}
    where the first inclusion follows since $\MQ'\subseteq \MQ_I$ holds and the latter inclusion since $(\transpose{\Pi_\MQ} \Lambda x)_{\MQ_I} \in \Z^{\MQ_I}$ holds for all $x\in \M^{W_I}$. Then, by monotonicity of the convex hull we have: 
    \begin{equation*}
        \conv(\{x\in G \mid (\transpose{\Pi_\MQ} \Lambda x)_{\MQ'} \in \Z^{\MQ'}\}) \subseteq \conv(\{x\in G \mid (\transpose{\Pi_\MQ} \Lambda x)_{\MQ_I} \in \Z^{\MQ_I}\}) \subseteq \conv(G\cap \M^{W_I}).
    \end{equation*}
    
    Since the conditions of \cref{thm_affine_fiber_integral} are satisfied for $\MQ'$, we have equality throughout, which shows that \eqref{thm_milp_folding_condition} holds.
   
   Secondly, let $y\in P(A',b',\ell',u')\cap \M^{\MQ'}$ be a feasible solution to the reduced problem. By \cref{def_refeqpart}\ref{def_refeqpart_from_red} it holds for 
    $x' \coloneqq (\Lambda \widetilde{\Pi_\MQ} y,\zerovec_{W_B}) + \delta$  that $x'\in P(A,b,\ell, u)$ and $\transpose{c} x' = \transpose{c'} y + \transpose{c} \delta $. 
    Then, we have for $Q\in \MQ'$ that:
    \begin{equation*}
        (\transpose{\Pi_\MQ} \Lambda x')_Q = \sum_{w\in Q} \lambda_w x'_w = \sum_{w\in Q} \lambda_w (\lambda_w \frac{y_Q}{|Q|} + \delta_w) = y_Q + \sum_{w\in Q} \lambda_w \delta_w = (y + \transpose{\Pi_\MQ} \Lambda \delta_{W_U})_{Q} = d^y_{Q},
    \end{equation*} where $d^y_Q$ is integral by integrality of $y_Q$, $\Lambda$ and $\delta_{W_I\cap W_U}$.
    Note that we implicitly use that $Q\subseteq W_U$ holds, which is implied by $Q\in \MQ'$.
    Since this holds for any $Q\in \MQ'$, it follows that $x'\in H_{d^y}$ holds. Since $H_{d^y}$ is $W_I$-integral and nonempty since $x'\in H_{d^y}$ holds, the linear program $\min\{\transpose{c} x \mid x \in H_{d^y}\}$ contains a solution $x\in H_{d^y}\cap \M^{W_I}$ in its minimal face such that $\transpose{c} x \leq \transpose{c} x'$. Such a solution can be found in polynomial time using the ellipsoid method \cite{Khachiyan1980}, see \cite{Grotschel1981} and \cite{Bland1981} for an explanation on how the optimal minimal face can be identified.
    Proposition 6.1 in \cite{VanDerHulst2025} provides a proof of how $x$ can be determined from the minimal face in polynomial time. Furthermore, note that $\transpose{c} x \leq \transpose{c} x' = \transpose{c'} y + \transpose{c} \delta$ holds. 

    Finally, let us show that $H'_{d'}$ is an $\MQ_I$-integral polyhedron for any $d'\in \Z^{\MQ'}$.  First, note that we obtained $P(A',b',\ell',u')$ through the affine map $\kappa : \R^{W_U} \to \R^{\MQ}$ defined as $\kappa(x_{W_U}) \coloneqq 
    \transpose{\Pi_\MQ} \Lambda (x_{W_U} - \delta{W_U})$ that was applied to $P(A,b,\ell,u)$. Thus, the preimage of $H'_{d'}$ in $P(A,b,\ell,u)$ consists of those $x\in P(A,b,\ell,u)$ such that $(\transpose{\Pi_\MQ} \Lambda (x_{W_U} - \delta_{W_U}))_{\MQ'} = d'$. Then, for $d \coloneqq d' + (\transpose{\Pi_\MQ} \Lambda \delta_{W_U})_\MQ'$ the preimage of $H'_{d'}$ is exactly $H_d$. 
    As we have $|\MQ| \leq |W_U|$, $\kappa$ is a surjective affine transformation and standard theory shows that the preimage of any face $F'$ of $H'_{d'}$ contains some minimal face $F$ of $H_d$ (see e.g. 
    \cite[Lemma 7.10]{Ziegler01}). Since $H_d$ is $\M^{W_I}$-integral, there exists some $x\in F\cap \M^{W_I}$. Then, note that $y \coloneqq \transpose{\Pi_\MQ} \Lambda (x_{W_U} - \delta_{W_U}) \in \M^{\MQ_I}$ holds, which we show in the proof of  \cref{thm_milp_folding}\ref{thm_milp_folding_to_red}. Then, since $y\in F'$ holds, it follows that any minimal face $F'$ of $H'_{d'}$ contains a $\MQ_I$-integral point, which shows that $H'_{d'}$ is $\MQ_I$-integral. 
\end{proof}

In \cref{thm_fibers_symmetry_integral}, we also show that $H'_{d}$ is integral. By \cref{thm_affine_fiber_integral}, this directly implies that $\conv(P(A',b',\ell',u') \cap \M^{\MQ'}) = \conv(P(A',b',\ell',u') \cap \M^{\MQ_I})$ holds for the reduced problem. Thus, any variables $Q\in \MQ_I \setminus \MQ'$ are implied integer in the reduced problem. 
One may hypothesize that the condition $\MQ'\subseteq \MQ_I$ is unnecessary and that the integrality of the fibers for $\MQ'$ implies the integrality for the fibers of $\MQ_I$. The example in \Cref{fig_nonintegral_affine_fiber_integral_fibers} shows that this is not the case, and that the condition $\MQ'\subseteq \MQ_I$ can indeed be helpful when the the fibers obtained by fixing the matrix with respect to $\MQ_I$ are not integral.\\

Given \cref{thm_fibers_symmetry_integral}, we are motivated to investigate methods to determine when the affine fibers $H_d$ are integral.  
Although \cref{thm_atu_ilp_integral} is useful for integer programs, we cannot apply it directly to mixed-integer programs, as they do not satisfy its conditions. Thus, we generalize affine TU decompositions to mixed-integer linear programs to deal with this limitation.

\subsection{Affine TU decompositions for mixed-integer linear programs}

We consider a mixed-integer linear program with variables $(x,y,z,\omega) \in \Z^{W^1} \times \Z^{W^2} \times \R^{W^3} \times \M^{W^4_I}$, where $W^4_I\subseteq W^4$ are the integer variables in $W^4$. We use $W_I = W^1 \cup W^2 \cup W^4_I$ to denote the integer variables and $W = W^1\cup W^2 \cup W^3 \cup W^4$ to denote all variables. 
The linear relaxation of the integer program is of the form:
\begin{align*}
    P \coloneqq \{A x + B y + C z = b,~ D x + E y + F \omega = g,~ \ell \leq \transpose{\begin{bmatrix}
        x &
        y &
        z &
        \omega
    \end{bmatrix}} \leq u \},
\end{align*} where $A \in \Z^{m_1\times W^1}$, $B\in \Z^{m_1\times W^2}$, $C\in \Z^{m_1\times W^3}$, $b\in \Z^{m_1}$, $D\in \Q^{m_2\times W^1}$, $E\in \Q^{m_2\times W^2}$, $F\in \Q^{m_2\times W^4}$ and $\ell \leq u$ hold, where $\ell$ and $u$ are rational vectors where we consider infinite entries to be integral. Any MILP can be brought into this form if we allow $m_1=0$ and $m_2=0$ to hold. 

\begin{theorem}
\label{thm_atu_milp_integral}
Suppose that for $k \leq |W^2|$, there exist matrices $S \in \Z^{m_1\times W^2}$, $U\in \Z^{m_1 \times k}$, $T\in \Z^{k \times W^2}$  and $R\in \Q^{m_2\times k}$ such that $B= S + U T$, $E=RT$ and  
$\begin{bmatrix}
    S & C\\
    T & \zerovec
\end{bmatrix}$ is totally unimodular and that $\ell,u \in \M^{W_I \cup W^3}$ holds. Then, for any $
    d = (d_{W^1}, d_T, d_{W^4_I}) \in \Z^{W^1} \times \Z^{k} \times \Z^{W^4_I}$, the affine fiber $H_d \coloneqq \{(x,y,z,\omega) \in P \mid x = d_{W^1},~ T y = d_T, ~\omega_{W^4_I} = d_{W^4_I} \}$ satisfies 
    \begin{equation*}
        H_d = \conv(H_d \cap \M^{W_I}) = \conv(H_d \cap \M^{W_I \cup W^3}) \text{ and }
    \end{equation*} 
    \begin{equation*}
        \conv(P\cap \M^{W_I}) = \conv(P\cap\M^{W_I\cup W^3})  = \conv(\{(x,y,z,\omega)\in P \mid x\in \Z^{W^1},~T y \in \Z^k,~ \omega \in \M^{W^4_I} \})
    \end{equation*} holds.
\end{theorem}
\begin{proof}
    We consider the affine fiber $H_d$ for any fixed $d\in \Z^{W^1} \times \Z^{k} \times \Z^{W^4_I}$ and rewrite it by substituting $B = S + U T$ and $D = R T$.
    
    \begin{align*}
      H_d \coloneqq& 
      \begin{aligned}[t]
          \{(x,y,z,\omega) \in \R^W \mid&~ A x + B y + C z = b,~ D x + E y + F \omega = g,~  x = d_{W^1},~ T y = d_T, \\&
       ~\omega_{W^4_I} = d_{W^4_I},~ \ell \leq \transpose{\begin{bmatrix}
        x &
        y &
        z &
        \omega
    \end{bmatrix}} \leq u \}
      \end{aligned}\\
      =& 
      \begin{aligned}[t]
          \{(x,y,z,\omega) \in \R^W \mid~&  x = d_{W^1},~ Sy + C z = b - A d_{W^1} - U d_{T},~ T y = d_T, \\
      &F \omega = g - D d_{W^1} - R d_T,~\omega_{W^4_I} = d_{W^4_I},~ \ell \leq \transpose{\begin{bmatrix}
        x &
        y &
        z &
        \omega
    \end{bmatrix}} \leq u \} 
      \end{aligned}\\
      = & \begin{aligned}[t]
      & \{x \in \R^{W^1} \mid x = d_{W^1}, ~\ell_{W^1} \leq x \leq u_{W^1} \} \times \\
      & \{\omega \in \R^{W^4} \mid F \omega = g - D d_{W^1} - R d_T, ~\omega_{W^4_I} = d_{W^4_I},~\ell_{W^4} \leq~ \omega \leq u_{W^4}\} \times \\
      &\{(y,z) \in \R^{W^2 \cup W^3} \mid Sy + C z = b - A d_{W^1} - U d_{T},~ T y = d_T, \ell_{W^2\cup W^3} \leq \transpose{\begin{bmatrix}
        y &
        z 
    \end{bmatrix}} \leq u_{W^2\cup W^3}\}
      \end{aligned}\\
    \coloneqq& H^{x}_d \times H^{\omega}_d \times H^{y,z}_d
    \end{align*}

    In the third step, we use that $H_d$ can be rewritten as the Cartesian product of the three polyhedra $H^x_d \subseteq \R^{W^1}$, $H^{y,z}_d \subseteq \R^{W^2\cup W^3}$ and $H^\omega_d \subseteq \R^{W^4}$. 
    Then since $\begin{bmatrix}
        S & C \\
        T & \zerovec
    \end{bmatrix}$ is totally unimodular and $\begin{bmatrix}
        b - A d_{W^1} - U d_{T} \\
        d_T
    \end{bmatrix}$ is integral by integrality of $b$, $U$, $A$ and $d$, and $\ell_{W^2\cup W^3}$ and $u_{W^2\cup W^3}$ are integral, $H^{y,z}_d$ is an integral polyhedron by Hoffman and Kruskal's theorem~\cite{HoffmanK56}. Thus, we have $H^{y,z}_d = \conv(H^{y,z}_d \cap \Z^{W^2 \cup W^3}) = \conv(H^{y,z}_d \cap (\Z^{W^2} \times \R^{W^3}))$, where the last equality follows since 
    $H^{y,z}_d \cap \Z^{W^2 \cup W^3} \subseteq H^{y,z}_d \cap (\Z^{W^2} \times \R^{W^3}) \subseteq H^{y,z}_d$ holds. This implies that $H_d = \conv(H_d \cap \M^{W_I}) = \conv(H_d \cap \M^{W_I \cup W^3})$ holds.

   Then, note that $C$ is totally unimodular since it is a submatrix of $\begin{bmatrix}
        S & C \\
        T & \zerovec
    \end{bmatrix}$. Since $A$, $B$ and $b$ are integral and the variable bounds $\ell_{W^3}$ and $u_{W^3}$ of $z$ are integral we have by that the equation $\conv(P\cap \M^{W_I}) = \conv(P \cap \M^{W_I \cup W^3})$ holds by~\cite[Theorem 3]{VanDerHulst2025}. Finally,  using $\widehat{T} = \begin{bmatrix}
        I_{W^1} & \zerovec & \zerovec\\
        \zerovec & T & \zerovec \\
        \zerovec & \zerovec & I_{W^3_I}
    \end{bmatrix}$ as the matrix in \cref{thm_affine_fiber_integral} and the fact that $H_d = \conv(H_d \cap \M^{W_I})$ holds, we derive that $\conv(P \cap \M^{W_I})) = \conv(\{(x,y,z,\omega)\in P \mid x\in \Z^{W^1},~T y \in \Z^k,~ \omega \in \M^{W^4_I} \})$ holds.\end{proof}

There are a few key differences between \cref{thm_atu_ilp_integral} and \cref{thm_atu_milp_integral} that make the latter more convenient to use. First of all, we do not require integrality of all of the right-hand sides, and we only need to construct an affine TU decomposition of a block in the constraint matrix rather than the complete matrix. Additionally, the $E=RT$ condition in \cref{thm_atu_milp_integral} is new. This condition broadens the scope of the theorem, as otherwise, $D=\zerovec$ would have to hold to guarantee implied integrality in the affine fiber $H_d$ in the proof. Furthermore, during the detection we infer implied integrality for the continuous variables and have $\conv(P\cap \M^{W_I}) = \conv(P\cap \M^{W_I\cup W^3})$.
\\

\Cref{thm_atu_milp_integral} can be used directly to show that the condition \eqref{thm_milp_folding_condition} in \cref{thm_milp_folding} holds. For integer variables $W_I\subseteq W$ and an equitable partition $(\MP,\MQ)$ such that either $Q\cap W_I = \emptyset$ or $Q\subseteq W_I$ holds, let $\MQ^2 = \{Q \in \MQ_I \mid |Q| > 1\}$ and let $W^2 = W_B \cup \bigcup_{Q\in \MQ^2} Q$. Then, if the conditions of  \cref{thm_atu_milp_integral} hold for $T=(\transpose{\Pi_\MQ} \Lambda)_{\MQ',\star}$ for some subset $\MQ'\subseteq \MQ^2$, this directly shows condition \eqref{thm_milp_folding_condition} in \cref{thm_milp_folding} by using \cref{thm_fibers_symmetry_integral}, and \cref{thm_fibers_symmetry_integral} shows that we can recover integral solutions for the original model.

\subsection{Example: Generalized Assignment problem}
\label{sec_example}
Consider the generalized assignment problem, where one computes an assignment of $J$ items with sizes $a_j, ~j\in J$ to $I$ knapsacks with capacity $C_i$ such that we maximize the sum of profits $\sum_{j\in J} c_j$ such that the capacities on the knapsacks are not exceeded. This can be modeled as an integer program using the binary variables $x_{ij} = \begin{cases}
    1 &\text{ if $j$ is packed into knapsack $i$}\\
    0 &\text{ otherwise}
\end{cases}$ for all $i\in I, j\in J$.

\begin{subequations}
\label{eq_mk_orig}
\begin{align}
\max\quad & \sum_{i\in I}\sum_{j\in J} c_j x_{ij} \\
\text{s.t.}\quad 
& \sum_{j\in J} a_j x_{ij} \leq C_i 
&& \forall i\in I \\
& \sum_{i\in I} x_{ij} \leq 1 
&& \forall j\in J \label{eq_mk_orig_item_packing}\\
& x_{ij} \in \{0,1\} 
&& \forall i\in I,\ \forall j\in J \label{eq_mk_orig_bounds}
\end{align}
\end{subequations}

Now, suppose that there exist identical items $j\in J$ that have the same weight and cost. Let $\mathcal{K} = \{K_1,\dotsc, K_n\}$ be a partition of $J$ into $n$ item sets that have identical cost and weight. Clearly, $n \leq |J|$ must hold, and by definition we have that $a_K = a_j$ and $c_K = c_j$ for any $j\in K$, for any $K\in \mathcal{K}$. Let the variable partition $\MQ$ be given by the parts $Q_{iK} = \{x_{ij} \mid \forall j\in K\}$ for $i\in I$, $K\in \mathcal{K}$, where identical items are placed into the same part. Let $\MP$ be a constraint partition  where all constraints \eqref{eq_mk_orig_item_packing} and \eqref{eq_mk_orig_bounds} belonging to identical items are placed into the same partition $P\in \MP$ and all other constraints are given their own partition. Then, it can be verified that $(\MP,\MQ)$ is an equitable partition of the generalized assignment problem in \eqref{eq_mk_orig}. Applying the reformulation suggested by the equitable partition we obtain a smaller formulation:
\begin{subequations}
\label{eq_mk_reduced}
\begin{align}
\max\quad & \sum_{i\in I}\sum_{K\in \mathcal{K}} c_K y_{iK} \\
\text{s.t.}\quad 
& \sum_{K\in \mathcal{K}} a_K y_{iK} \leq C_i 
&& \forall i\in I \label{eq_mk_reduced_knapsack}\\
& \sum_{i\in I} y_{iK} \leq |K| 
&& \forall K\in \mathcal{K} \\
& y_{ik} \in \{0,1,\cdots,|K|\} 
&& \forall i\in I,\ \forall K\in \mathcal{K}
\end{align}
\end{subequations}

Now, suppose that we are given any integral solution $y^\star$ to the reduced problem. For these values, we wish to recover a solution in the $x$-space. Then, we can add the constraints \eqref{eq_mk_affine_fixing}, that link the $x$ variables to $y^\star$, and hope to recover a solution for $x$.

\begin{subequations}
\label{eq_mk_affine}
\begin{align}
\max\quad & \sum_{i\in I}\sum_{j\in J} c_j x_{ij} \label{eq_mk_affine_fixing_obj}\\
\text{s.t.}\quad 
& \sum_{j\in J} a_j x_{ij} \leq C_i 
&& \forall i\in I \label{eq_mk_affine_fixing_knapsack}\\
& \sum_{i\in I} x_{ij} \leq 1 
&& \forall j\in J \\
& \sum_{j\in K} x_{ij} = y^{\star}_{iK} &&\forall i \in I, \forall K \in \mathcal{K} \label{eq_mk_affine_fixing}\\
& x_{ij} \in \{0,1\} 
&& \forall i\in I,\ \forall j\in J
\end{align}
\end{subequations}

We can substitute the equations \eqref{eq_mk_affine_fixing} into the constraints \eqref{eq_mk_affine_fixing_knapsack} and the objective \eqref{eq_mk_affine_fixing_obj}.

\begin{subequations}
\label{eq_mk_postsolve_fixed}
\begin{align}
\max\quad & \sum_{i\in I}\sum_{K\in K} c_K y^\star_{iK} \\
\text{s.t.}\quad 
& \sum_{K\in \mathcal{K}} a_K y^\star_{iK} \leq C_i 
&& \forall i\in I \label{eq_mk_postsolve_knapsack}\\
& \sum_{i\in I} x_{ij} \leq 1 
&& \forall j\in J \label{eq_mk_postsolve_usedonce}\\
& \sum_{j\in K} x_{ij} = y^{\star}_{iK} &&\forall i \in I, \forall K\in \mathcal{K} \label{eq_mk_postsolve_linking}\\
& x_{ij} \in \{0,1\} 
&& \forall i\in I,\ \forall j\in J
\end{align}
\end{subequations}
Note that in \eqref{eq_mk_postsolve_fixed}, both the objective and the knapsack constraints are not dependent on $x$ anymore. In fact, by feasibility of $y^\star$, all constraints in \eqref{eq_mk_postsolve_knapsack} are feasible by feasibility of \eqref{eq_mk_reduced_knapsack}.
Then, we observe that the remaining problem is given by an assignment problem on the $x$ variables. Since the constraints \eqref{eq_mk_postsolve_usedonce}, \eqref{eq_mk_postsolve_linking} form a totally unimodular matrix and the right-hand side given by $1$ and $y^\star$ is integral, the LP relaxation of \eqref{eq_mk_postsolve_fixed} is integral. Moreover, the solution $\bar{x}_{ij} = \frac{y^\star_{iK}}{|K|}$ for all $i\in I, K\in\mathcal{K}$ such that $j\in K$ is feasible for the LP relaxation of \eqref{eq_mk_postsolve_fixed} by \cref{thm_lp_folding}. Thus, an integral solution $x$ can be obtained by finding a vertex solution of the LP relaxation of \eqref{eq_mk_postsolve_fixed}.

The reformulation of \eqref{eq_mk_affine} into \eqref{eq_mk_postsolve_fixed} can be seen as exhibiting an affine TU decomposition. The rows of the constraints \eqref{eq_mk_postsolve_linking} correspond to the $T$-matrix, and their substitution in \eqref{eq_mk_postsolve_knapsack} corresponds to the $U$-matrix. The rows of \eqref{eq_mk_postsolve_usedonce} form the nonzeros of the $S$-matrix.
Additionally, note that the usage of the equations to cancel the nonzeros in the objective is possible due to the fact that the objective vector is equitably partitioned by $\MQ$.\\

There are a few interesting observations that can be made about the generalized assignment problem. First of all, as also implied by Salvagnin~\cite{Salvagnin2013}, the used equitable partition is not necessarily the coarsest equitable partition. Suppose for example, that for all knapsacks $i\in I$ the capacity $C_i$ is equal to some constant $C$. Then, the knapsack constraints \eqref{eq_mk_orig_item_packing} have additional permutation symmetries which permute the items between the bins. These symmetries are also captured by the coarsest equitable partition since equitable partitions generalize permutation symmetries. Consequently, for the coarsest equitable partition $(\MP,\MQ)$, $\MQ$ contains variables $x_{ij}$ belonging to multiple bins in $I$. In the case where one would reformulate \eqref{eq_mk_orig} using orbital shrinking, the orbital shrinking subproblem would then become the $\cplxNP$-hard binpacking problem, in contrast to the polynomial-time solvable assignment problem that we derive here. Salvagnin shows in~\cite{Salvagnin2013} that for the generalized assignment model, employing a constraint programming solver to solve the binpacking subproblem and additionally reformulating the symmetric knapsacks yields better performance than the reformulation with the assignment problem that we derive here. 

Secondly, we note that there is a minor way to improve the postsolve subproblem in \eqref{eq_mk_postsolve_fixed}. Let $\mathcal{K}^1 \coloneqq \{K\in \mathcal{K} \mid a_{K} = 1\}$ and $J^1 \coloneqq \{ j\in J \mid a_j = 1\}$ denote the set of items with weight $1$. For these items, it is not necessary to fix them to their corresponding $y^\star_{iK}$ value in \eqref{eq_mk_postsolve_fixed}. In particular, if we change the equalities in \eqref{eq_mk_postsolve_linking} to only be added for all $\mathcal{K}\setminus \mathcal{K}^1$, and substitute only all items with non-unit weight in \eqref{eq_mk_postsolve_knapsack}, we obtain the following reformulation.

\begin{subequations}
\label{eq_mk_unit}
\begin{align}
\max\quad & \sum_{i\in I}\sum_{J\in \mathcal{J}^1} c_j x_{ij} + \sum_{i\in I}\sum_{K\in \mathcal{K}\setminus \mathcal{K}^1} c_K y^\star_{iK}  \\
\text{s.t.}\quad 
& \sum_{j\in \mathcal{J}^1} x_{ij} \leq C_i - \sum_{K\in \mathcal{K}\setminus \mathcal{K}^1} a_K y^\star_{iK}
&& \forall i\in I \label{eq_mk_unit_knapsack}\\
& \sum_{i\in I} x_{ij} \leq 1 
&& \forall j\in J \label{eq_mk_unit_usedonce}\\
& \sum_{j\in K} x_{ij} = y^{\star}_{iK} &&\forall i \in I, \forall K\in \mathcal{K}\setminus \mathcal{K}^1 \label{eq_mk_unit_linking}\\
& x_{ij} \in \{0,1\} 
&& \forall i\in I,\ \forall j\in J
\end{align}
\end{subequations}

It is not difficult to see that \eqref{eq_mk_unit} is still a perfect formulation. The constraint matrix defined by \eqref{eq_mk_unit_knapsack}--\eqref{eq_mk_unit_linking} is totally unimodular as it is the node-edge incidence matrix of a bipartite graph and the right-hand sides are integral. The two parts of the bipartition are given by constraints \eqref{eq_mk_unit_usedonce} and (\eqref{eq_mk_unit_knapsack} plus \eqref{eq_mk_unit_linking}), respectively. The fixings for $\mathcal{K}^1$ can be left out due to the fact that the variables in $J^1$ are implied integer. Let $P$ be the linear relaxation of \eqref{eq_mk_orig}. Then total unimodularity of \eqref{eq_mk_unit} shows by affine TU decomposition that $\conv(P\cap \Z^n) = \conv(\{x\in P \mid \sum_{j\in K} x_{ij} \in \Z, ~ \forall i\in I, ~\forall K \in \mathcal{K}\setminus \mathcal{K}^1\})$ holds. From this equality it also follows that $y_{iK}$ is implied integer for all $i\in I$ and all $K\in \mathcal{K}^1$ by the other variables in \eqref{eq_mk_reduced}. One way in which this additional power from implied integrality manifests is that in contrast to \eqref{eq_mk_postsolve_fixed}, solving \eqref{eq_mk_unit} may give a stronger solution than the provided $y^\star$ solution, if $y^\star$ is suboptimal and one of the items $j\in J^1$ can still be assigned to some knapsack. This correspond to having the proper subset $\MQ' \subset \MQ_I$ in the statement of \cref{thm_fibers_symmetry_integral}.

\section{Detecting DRCR for mixed-integer linear programs}
\label{sec_algorithm}
Next, let us formulate an algorithm to detect the reductions described in \cref{thm_milp_folding}. To achieve such a reduction, we need to balance two conditions: we would like the partition $\MQ$ to be as coarse as possible, but \eqref{thm_milp_folding_condition} must also hold. These two objectives naturally conflict with each other, as a coarse partition of the integer variables requires us to find an affine TU decomposition with fewer integrality constraints. Although one could consider the coarsest possible partition and simply try to detect an affine TU decomposition and apply \cref{thm_atu_milp_integral} and \cref{thm_fibers_symmetry_integral}, it is unlikely that \eqref{thm_milp_folding_condition} holds, either by the structure of the problem or by additional symmetries. For the example of the generalized assignment problem in \cref{sec_example}, if two bins have identical capacities then the symmetries that permute the two bins imply that \eqref{thm_milp_folding_condition} is not satisfied for the coarsest equitable partition. \\

In order to detect an affine TU decomposition that we can use, we attempt to detect the structure of \cref{thm_atu_milp_integral} where $T$ is given by $(\transpose{\Pi_\MQ} \Lambda x_{W_U})_{\MQ',\star}$, where $\MQ' \subseteq  \{ Q\in \MQ_I \mid |Q| > 1\}$. Note that the columns of $T$ consist of unit vectors, which may have been negated. For $w\in W^2$ we have for its column in $T$ that  $T_{\star,w} = \lambda_w \mathbbm{e}_{Q}$ holds for $Q\in \MQ'$ such that $w\in Q$, where $\mathbbm{e}_{Q}$ is a unit vector.  
Then it follows that $B_{\star,w} = (S + U T)_{\star,w} = S_{\star,w} + U T_{\star,w} = S_{\star,w} + \lambda_w U_{\star,Q}$ holds. Since we want $S$ to be part of the totally unimodular matrix, its entries are $\pm 1$ values. Thus, if a column $B_{\star,w}$ contains large entries these must come from the $U_{\star,Q}$ vector. Furthermore, note that for all columns in the block $Q$, its entries must differ by at most $1$ from $\lambda_w U_{\star,Q}$. Thus, columns in $B_{\star,Q}$ must differ by at most $1$ up to $\pm 1$ scaling. For example, if $Q$ contains some block $(P,Q)\in \MP \times \MQ$ such that 
$B_{P,Q} = \begin{bmatrix}
    3 & 0 \\
    0 & 3
\end{bmatrix}$, then it is impossible to construct $U$ and $S$ such that we can get $B = S + U T$, since $B_{P,w_1} - B_{P,w_2} \notin \{-1,0,1\}^P$ implies that there exists no choice of $U_{\star,Q}$ such that we can construct a totally unimodular matrix $S$. \\

For a matrix $A \in \R^{V\times W}$, we define the \emph{nonternary mask} $\mathrm{ntmask}(A) \in \R^{V\times W}$ such that $\mathrm{ntmask}(A)_{v,w} \coloneqq \begin{cases}
    0 &\text{ if $|A_{v,w}| = 1$}\\
    A_{v,w} &\text{ otherwise }
\end{cases}$ holds for all $v\in V$ and $w\in W$. Then, if two columns of $B$ have the same nonternary mask up to $\pm 1$ scaling, they are candidates to belong to the same part $Q\in \MQ$. 
In order to generate equitable partitions which satisfy this condition, we change the initial partition that we pass to the color refinement algorithm to only place two integer columns $w_1$ and $w_2$ in the same initial part $Q\in \MQ_0$ if $\mathrm{ntmask}(A_{\star,w_1}) = \pm \mathrm{ntmask}(A_{\star,w_2})$ holds. Using standard hashing techniques, this can be done in linear time, see~\cite{Gemander2020} for further details. The condition that the nonternary masks are equivalent requirement is a sufficient condition to ensure that $S$ is a ternary matrix, but not a necessary one. For example, we have $\begin{bmatrix}
    3 & 2\\
    2 & 3 
\end{bmatrix} = \begin{bmatrix}
    1 & 0 \\
    0 & 1 
\end{bmatrix} + \transpose{\begin{bmatrix}
    2 & 2
\end{bmatrix}} \begin{bmatrix}
    1 & 1 
\end{bmatrix}$ where the $\transpose{\begin{bmatrix}
    S & T
\end{bmatrix}}$ matrix is given by the totally unimodular matrix  $\transpose{\begin{bmatrix}
    1 & 0 & 1 \\
    0 & 1 & 1 
\end{bmatrix}}$ but the two nonternary masks of the original columns differ. \\

Then, consider the second condition that $E=R T$ holds. For $v\in V$ and $w\in W$ we have $E_{v,w} = (R \transpose{\Pi_\MQ} \lambda)_{v,w} = R_{v,Q} \lambda_w$ for $Q$ such that $w\in Q$, which only holds if the columns in $E_{\star,Q}$ are identical up to $\pm 1$ scaling. Thus, we require that $E_{\star,w} = \pm E_{\star,w'}$ holds for all $w,w'\in Q$.\\

Although totally unimodular matrices can be detected in polynomial time~\cite{Truemper1990}, current implementations still have quintic time complexity~\cite{Walter2013}, which is too slow for practical purposes. Instead, we focus on \emph{network matrices} and their transposes, which form large subclasses of totally unimodular matrices~\cite{Schrijver86}. We use the fast column augmentation algorithms by Bixby and Wagner for network matrices~\cite{BixbyWagner1988} and the algorithm by van der Hulst and Walter for transposed network matrices~\cite{RowNetworkMatrixPaper}. For a (transposed) network matrix $M \in \R^{V\times W}$ and a column $u\in \R^V$, both algorithms are given by a procedure $\mathrm{AugmentNetwork}(M,u)$ that returns whether the augmented matrix $\begin{bmatrix}
    M & u
\end{bmatrix}$ is a (transposed) network matrix, growing the (transposed) network matrix by one column. The operation of \emph{augmenting} a column to a matrix $M$ turns $M$ into $\begin{bmatrix}
    M & u
\end{bmatrix}$.\\

Then, we are ready to describe an algorithm that detects the conditions of \cref{thm_milp_folding} by simultaneously computing a reflection reduction and an affine TU decomposition of the MILP. The algorithm is described in \cref{algo_eos} and consists of roughly five steps.\\

In the first step, we determine the row partition that is used for \cref{thm_atu_milp_integral}, by partitioning the rows according to whether they contain non-integral entries or right-hand side. Then, we determine for each of the the connected components of the submatrix formed by the continuous columns that is in the integral part whether it is (transposed) network. These components form the $C$-matrix in \cref{thm_atu_milp_integral}. This step is identical to part of the algorithm described in~\cite{VanDerHulst2025} to detect implied integrality using totally unimodular submatrices.\\

In the second step, we use the computed row partition for \cref{thm_atu_milp_integral} to compute an initial partition that we pass to the color refinement algorithm. As discussed above, we put two columns in the same partition only if they are (up to $\pm 1$ scaling) identical in the non-integral rows and if their nonternary masks are identical (up to $\pm 1$ scaling) in the integral rows. In the third step, we compute a reflection reduction given this initial partition. \\

For the fourth step, we consider the computed reflection reduction and attempt to construct a matching affine TU decomposition using \cref{thm_atu_milp_integral} that extends the computed reduction to integer variables. 
For every  $Q \in \MQ'$, where $\MQ' \coloneq \{ Q \in \MQ_I \mid |Q| > 1 \}$, we consider adding a column $u$ to the $U$ matrix of the affine TU decomposition. For our approach in \cref{algo_eos}, we simply check if $A_{v,Q} =  \alpha \lambda_{Q}$ holds for some $\alpha \in \R\setminus \{0\}$. If so, then we set $u_v = \alpha_v$. Then, we use $S_{\star,Q} = A_{\star,Q} - u (\transpose{\Pi_\MQ} \Lambda)_{\star,Q}$ This effectively cancels the $v$'th row of $A_{\star,w}$ by using $u$. If $u = \zerovec$, then we do not need the additional row of $\Pi_\MQ \Lambda$ to cancel any rows, and the additional row in the $T$ part of the affine TU decomposition can only harm detection of total unimodularity, as any submatrix of a TU matrix is TU. In this case, we consider the partition of the subset $\MQ' \setminus \{Q\}$ rather than $\MQ'$, which is sufficient by \cref{thm_fibers_symmetry_integral}, and we do not add the corresponding row to the $T$ matrix.
Finally, we ensure that for the affine TU decomposition $E=R T$ holds by checking if $S_{V_R,Q} = \zerovec$ holds, i.e. that all rows in $V_R$ are properly canceled. \\

{\footnotesize{
\begin{algorithm}[H]
    \caption{Detecting DRCR for MILP using affine TU decompositions.}
    \label{algo_eos}

    \LinesNumbered
    \TitleOfAlgo{DRCR$(A,b,\ell,u,c, W_I)$}

  \KwIn{Constraint matrix $A\in \R^{V\times W}$, right-hand side $b\in \R^V$, lower bounds $\ell\in \R^{W}$, upper bounds $u\in \R^W$, objective $c\in \R^W$ and integer variables $W_I\subseteq W$}
  \KwOut{A reflection reduction $(\MP,\MQ,\lambda,\gamma,\delta,V_U,W_U)$ such that the conditions of \cref{thm_milp_folding} hold.}
    \SetKw{Continue}{continue}
    \SetKw{Break}{break}
    $V_R\gets \{v\in V | \, b_v \notin \mathbb{Z}$ or $a_{v,w}\notin\mathbb{Z}$ for some $w\in W\}$\label{algo_eos_start}\;
    Let $\mathcal{K}$ be the set of connected components of $A_{V, W \setminus W_I}$\;
    Let $M$ denote the (transposed) network matrix that is constructed\;
    \For{$K\in\mathcal{K}$}{
        \If{$V_R\cap V_K\neq \emptyset$}{
            $V_R\gets V_R \cup V_K$\;
            \Continue
        }
       $\mathrm{isNetwork} \gets True$\;
        \For{$w\in W_K$}{
            \uIf{not $\mathrm{AugmentNetwork}(M, A_{\star,w})$}{
                $\mathrm{isNetwork} \gets False$\;
                \Break
            }\lElse{
                Augment $M$ with $A_{\star,w}$
            }       
         }
         \If{not $\mathrm{isNetwork}$}{
            $V_R \gets V_R \cup V_K$\;
            Remove from $M$ any columns of $A_{\star,W_K}$ that were augmented\;
         }
    }
    \label{algo_eos_point1}
    Let $\MQ_0$ be a partition of $W$ such that $Q\cap W_I = \emptyset$ or $Q\subseteq W_I$ holds for all $Q\in \MQ_0$. Additionally, we impose the condition that for all $Q\subseteq W_I$, $w,w'\in Q$ holds if and only if $A_{V_R,w} = \pm A_{V_R,w'}$ and $\mathrm{ntmask}(A_{V\setminus V_R,w}) = \pm \mathrm{ntmask}(A_{V\setminus V_R,w'})$ hold\;\label{algo_eos_point2}
    Let $\MP_0 \gets \{V\}$\;
    $(\MP^1,\MQ^1, \lambda, \gamma, \delta, V^1_O, W^1_O) \gets \textsc{ReflectionRefinement}(\MP_0, \MQ_0, A, b, \ell, u)$\label{algo_eos_first_refine}\;

    Sort $\MQ^1$ by descending size $|Q|$ of its parts $Q \in \MQ^1$\label{algo_eos_start_ATU}\;
    \For{$w \in W \setminus W^1_O$}{
        \lIf{not $\mathrm{AugmentNetwork}(M, A_{\star,w})$}{
            $\MQ^1 \gets \MQ^1 \cup \{w\}$
        }
        \lElse{ Augment $M$ with $A_{\star,w}$}
    }
    \For{$Q \in \MQ^1$ such that $Q\subseteq W_I$ and $|Q| > 2$}{
        Let $V_\mathrm{cancel} \coloneq \{ v\in V \mid A_{v,Q} = \alpha_v \lambda_Q \text{ for some } \alpha_v \in R\setminus \{0\}\}$ and let $u\in \R^V$ be such that $u_v = \alpha_v$ \text{ if } $v\in V_\mathrm{cancel}$ and $u_v = 0$ otherwise.\;
        
        Let $S = A_{\star,Q} - u (\transpose{\Pi}_\MQ \Lambda)_{Q}$ and let $M'\in \R^{(V\cup Q^1) \times Q}$ using $M' \gets \begin{bmatrix}
            S\\
            \zerovec_{|\MQ^1|\times Q}
        \end{bmatrix}$\;
        \If{$S_{V_R,\star} \neq \zerovec$}{
        Modify $\MQ^1$ by splitting $Q$ into $|Q|$ parts that contain its individual elements\;
            \Continue
        }

        \lIf{$V_\mathrm{cancel} \neq \emptyset$}{
            Update $M'_{\MQ^1,Q} \gets (\transpose{\Pi}_\MQ \Lambda)_Q $
        }
        $\mathrm{isNetwork} \gets True$\;
        \For{$w \in Q$}{
            \uIf{not $\mathrm{AugmentNetwork}(M, M'_{\star,w})$}{
                $\mathrm{isNetwork} \gets False$\;
                \Break
            }\lElse{
                Augment $M$ with $M'_{\star,w}$
            }
        }
        \If{not $\mathrm{isNetwork}$}{
            Remove from $M$ any column that was augmented from $M'$\;
            Modify $\MQ^1$ by splitting $Q$ into $|Q|$ parts that contain its individual elements\;
        }
    }
    \label{algo_eos_end_ATU}
    $(\MP^2,\MQ^2, \lambda, \gamma, \delta, V^2_O, W^2_O) \gets \textsc{ReflectionRefinement}(\MP^1, \MQ^1, A, b, \ell, u)$\label{algo_eos_additional_refinements}\;
    \Return $(\MP^2,\MQ^2, \lambda, \gamma, \delta, V^2_O, W^2_O)$\;
\end{algorithm}
}
}

In the fifth and final step, we consider all $Q\in \MQ$ that could not be added to the affine TU decomposition. For each such a part $Q$, we \emph{individualize} it by splitting $Q$ into $|Q|$ parts where each element is in its own part of size $1$, and run the refinement algorithm again. Note that this final refinement may refine `good' parts that we computed the affine TU decomposition for previously into several parts. This is unfortunate, as it may adjust the requirements for the $T$ matrix, and break the affine TU decomposition. An example of this phenomenon is shown in \cref{example_refinement_problematic}.
However, we show in \cref{thm_refinement_integrality_okay} that, because the new partition $\MQ$ is a refinement of the previous partition, \eqref{thm_milp_folding_condition} is still satisfied and we can still recover solutions in polynomial time.

\begin{example}
    \label{example_refinement_problematic}
    We consider a run of \cref{algo_eos} on the  mixed-integer linear program $\conv(\{x \in \R^6 \mid A x = b, x\geq 0 \}\cap \Z^6)$ where 
    \begin{equation*}
        A = \begin{bmatrix}
    2 & 2 & 2 & 2 & 3 & 3 \\
    1 & 0 & 1 & 0 & 1 & 0 \\
    0 & 1 & 0 & 1 & 1 & 0 \\
    0 & 1 & 1 & 0 & 0 & 1 \\
    1 & 0 & 0 & 1 & 0 & 1 \\
    1 & 0 & 0 & 0 & 1 & 0 \\
    0 & 1 & 0 & 0 & 1 & 0 \\
    0 & 0 & 1 & 0 & 0 & 1 \\
    0 & 0 & 0 & 1 & 0 & 1 
    \end{bmatrix} \in \Z^{9\times 6}\text{ and } b= \begin{bmatrix}
        8\\
        2\\
        2\\
        2\\
        2\\
        1\\
        1\\
        1\\
        1
    \end{bmatrix} \in \Z^9.
    \end{equation*} We use $v_i$ to indicate the $i$'th row index and $w_j$ to indicate the $j$'th column index. 
    The partition returned by \textsc{ReflectionReduction} is given by $\MP_1 = \{\{v_1\}, \{v_2,v_3,v_4,v_5\}, \{v_6,v_7,v_8,v_9\}\}$ and $\MQ_1 = \{\{w_1,w_2,w_3,w_4\}, \{w_5,w_6\}\}$. Consider $Q\coloneqq \{w_1,w_2,w_3,w_4\}$, and note that $A_{\star,Q}$ admits an affine  TU decomposition using $A_{\star,Q} = \begin{bmatrix}
        \zerovec_{1\times4}\\
        A_{V\setminus \{v_1\},Q}
    \end{bmatrix} + \begin{bmatrix}
        2 \\
        \zerovec_{8\times 1}
    \end{bmatrix} \begin{bmatrix}
        1 & 1 & 1 & 1
    \end{bmatrix}$, where it can be verified that the matrix $\begin{bmatrix}
        \zerovec_{1\times4}\\
        A_{V\setminus \{v_1\},Q}\\
        \onevec_{1\times,4}
    \end{bmatrix}$ is totally unimodular. Thus, initially, \cref{algo_eos} will reformulate the integrality constraints on $Q$. 
    Then, \cref{algo_eos} attempts to reformulate the integrality constraints for $\{w_5,w_6\}$ and fails to do so. Thus, we refine $\{w_5,w_6\}$ into $\{w_5\},\{w_6\}$, and run \textsc{ReflectionReduction} again, which results in the partition given by $\MP_2 =\{\{v_1\}, \{v_2,v_3\},\{v_4,v_5\}, \{v_6,v_7\},\{v_8,v_9\}\}$ and $\MQ_2 = \{\{w_1,w_2\},\{w_3,w_4\},\{w_5\},\{w_6\}\}$. Note for $\MQ_2$ that $Q$ is now split up into the two parts $\{w_1,w_2\}$ and $\{w_3,w_4\}$. However, this also affects the affine TU decomposition that we previously computed, as the matrix $\transpose{\Pi_\MQ}$ that represents the $T$ matrix in the affine TU decomposition is refined from $\transpose{\Pi_{\MQ_1}}$ to $\transpose{\Pi_{\MQ_2}}$, where the nonzero submatrix on $Q$ changes from $\onevec_{1\times 4}$ to $\begin{bmatrix}
        1 & 1 & 0 & 0\\
        0 & 0 & 1 & 1
    \end{bmatrix}$. Unfortunately, the matrix $\begin{bmatrix}
        0 & 0 & 0 & 0 \\
        \multicolumn{4}{c}{A_{V\setminus \{v_1\},Q}}\\
        1 & 1 & 0 & 0 \\
        0 & 0 & 1 & 1 
    \end{bmatrix}$ is not totally unimodular, which invalidates the previously computed affine TU decomposition. Thus, it may be the case that after a run of~\cref{algo_eos} that the affine TU decomposition that is computed no longer matches the equitable partition, due to the additional refinements that are computed in line \ref{algo_eos_additional_refinements} of~\cref{algo_eos}.
\end{example}

\begin{lemma}
    \label{thm_refinement_integrality_okay}
    Let $(\MP^1,\MQ^1,\gamma^1,\lambda^1,\delta,V^1_O,W^1_O)$ be a reflection reduction of $G\coloneqq P(A,b,\ell,u)$ and let it be given that for some $\widehat{\MQ^1} \subseteq \MQ^1_I$ that 
    $H_d \coloneqq \{ x \in G \mid (\transpose{\Pi_{\MQ^1}} \Lambda^1 x_{W^1_O})_{\widehat{\MQ^1}} = d \}$ is a ${W_I}$-integral polyhedron for all $d\in \Z^{\widehat{\MQ^1}}$.
    Consider any refinement $(\MP^2,\MQ^2, \lambda^2, \gamma^2,\delta,V^2_O,W^2_O)$ of the given reflection reduction where $\widehat{\MQ}^2 \subseteq \MQ^2$ corresponds to the parts that were refined from $\widehat{\MQ}^1$ such that $\bigcup_{Q\in \widehat{\MQ}^2} Q = \bigcup_{Q\in \widehat{\MQ}^1} Q$ holds. Then, \eqref{thm_milp_folding_condition} holds for the refined reduction, i.e. we have:
    \begin{equation*}
        \conv(G\cap \M^{W_I}) = \conv(\{x \in G \mid (\transpose{\Pi_{\MQ^2}} \Lambda^2 x_{W^2_O})_{{\MQ^2_I}} \in \Z^{{\MQ^2_I}}\}) = \conv(\{x \in G \mid (\transpose{\Pi_{\MQ^2}} \Lambda^2 x_{W^2_O})_{\widehat{\MQ^2}} \in \Z^{\widehat{\MQ^2}}\}),
    \end{equation*}
\end{lemma}
\begin{proof}
    First, let us show a few facts about the refined reduction $(\MP^2,\MQ^2, \lambda, \gamma,\delta,V^2_O,W^2_O)$. First of all, note that $\MP^2$ and $\MQ^2$ must refine $\MP^1$ and $\MQ^1$ respectively, and we must have $W^1_O\subseteq W^2_O$ since additional refinements can only turn bipolar variables into unipolar variables, but not the other way around. 
    Furthermore, note that additional refinements cannot change the value of $\lambda^1_w$ for $w\in W^1_O$, since the relative signs of unipolar parts are fixed. Thus, $\lambda^{1} = \lambda^{2}_{W^1_O}$ holds. By design of our algorithm, $\delta$ is constant in both reductions.
    Then, we claim that the following holds:
    \begin{equation*}
        G\cap \M^{W_I}
        \subseteq \{x \in G \mid (\transpose{\Pi_{\MQ^2}} \Lambda^2 x_{W^2_O})_{\widehat{\MQ^2}} \in \Z^{\widehat{\MQ^2}}\} \subseteq \{x \in G \mid (\transpose{\Pi_{\MQ^1}} \Lambda^1 x_{W^1_O})_{\widehat{\MQ^1}} \in \Z^{\widehat{\MQ^1}}\}.
    \end{equation*}
    Here, the first inclusion follows since $\transpose{\Pi_\MQ} \Lambda^2$ is an integral matrix. For the second inclusion, consider any $x \in \{x \in G \mid (\transpose{\Pi_{\MQ^2}} \Lambda^2 x_{W^2_O})_{\widehat{\MQ^2}} \in \Z^{\widehat{\MQ^2}}\}$.
    Then, consider any $Q \in \widehat{\MQ^1}$. Since $\MQ^2$ refines $\MQ^1$, there exist parts $Q'_i \in \widehat{\MQ^2}$ for $i=1,\cdots,k$ such that $\bigcup_{i=1}^k Q'_i = Q$ holds. Then, we can derive that:

    \begin{equation*}
    (\transpose{\Pi_{\MQ^1}} \Lambda^1 x_{W^1_O})_{Q} = \sum_{w\in Q} \lambda^1_w x_w = \sum_{i=1}^k \sum_{w\in Q'_i} \lambda^1_w x_w \stackrel{\textup{(a)}}{=} \sum_{i=1}^k \sum_{w\in Q'_i} \lambda^2_w x_w = \sum_{i=1}^k (\transpose{\Pi_{\MQ^2}} \Lambda^2 x_{W^2_O})_{Q'_i},
    \end{equation*}
    where \textup{(a)} holds since $\lambda^1 = \lambda^2_{W^1_O}$ holds, and the final expression is integral by integrality of $(\transpose{\Pi_{\MQ^2}} \Lambda^2 x_{W^2_O})_{Q'_i}$, using that $Q'_i \in \widehat{Q^2}$.
    
    Taking convex hulls, we note that by applying \cref{thm_affine_fiber_integral}, using ${W_I}$-integrality of $H_d$, we have equality throughout:
        \begin{equation*}
        \conv(G\cap \M^{W_I}) = \conv(\{x \in G \mid (\transpose{\Pi_{\MQ^2}} \Lambda^2 x_{W^2_O})_{\widehat{\MQ^2}} \in \Z^{\widehat{\MQ^2}}\}) = \conv(\{x \in G \mid (\transpose{\Pi_{\MQ^1}} \Lambda^1 x_{W^1_O})_{\widehat{\MQ^1}} \in \Z^{\widehat{\MQ^1}}\}).
    \end{equation*}
    Finally, we have the following inclusions:
        \begin{equation*}
        G\cap \M^{W_I}
         \subseteq \{x \in G \mid (\transpose{\Pi_{\MQ^2}} \Lambda^2 x_{W^1_O})_{\MQ^2_I} \in \Z^{\MQ^2_I}\} \subseteq \{x \in G \mid (\transpose{\Pi_{\MQ^2}} \Lambda^2 x_{W^2_O})_{\widehat{\MQ^2}} \in \Z^{\widehat{\MQ^2}}\},
    \end{equation*}
    taking convex hulls and noting that the first and last sets are equal, we again have equality throughout, which concludes the proof
            \begin{equation*}
        \conv(G\cap \M^{W_I})
         = \conv(\{x \in G \mid (\transpose{\Pi_{\MQ^2}} \Lambda^2 x_{W^1_O})_{\MQ^2_I} \in \Z^{\MQ^2_I}\}) = \conv(\{x \in G \mid (\transpose{\Pi_{\MQ^2}} \Lambda^2 x_{W^2_O})_{\widehat{\MQ^2}} \in \Z^{\widehat{\MQ^2}}\}),
    \end{equation*}
    
\end{proof}
It is not too difficult to see that finding a solution to the original problem can also still be done in polynomial time by simply considering the fiber $H_d$ defined where the affine subspaces are defined by $\Pi_\MQ^1 \Lambda^1 x_{W^1_O} = d$. In \cref{example_refinement_problematic}, this corresponds to adding the equality $x_1 + x_2 + x_3 + x_4 = y_{1,2} + y_{3,4}$ in the postsolve procedure rather than the two equalities $x_1 + x_2 = y_{1,2}$ and $x_3 + x_4 = y_{3,4}$ which are shown to be problematic in \cref{example_refinement_problematic}. Although the latter two equalities clearly imply the former equality, adding them to the linear program breaks the integrality guarantee.

Next, let us consider the total runtime. As for linear programs, we assume that the reflection reduction uses a $\delta$-vector whose bitlength is of the same order in the worst-case, which is reasonable as the choice for $\delta$ that is described in~\cref{thm_delta_equivalent_complemented} satisfies this assumption.
\begin{theorem}
For a mixed-integer linear program $F(A,b,\ell,u,c)$ with $A\in \R^{V\times W}$ and integer variables $W_I\subseteq W$, let $m$ be the total bitlength of the entries of $\begin{bmatrix}
    \ell & 0 \\
    u & 0 \\
    c & 0\\
    A & b
\end{bmatrix}$ and let $n = |V| + |W|$. Assuming that \textsc{ReflectionRefinement} computes a $\delta$ with total bitlength $\orderO(m)$, \cref{algo_eos} computes a  reflection reduction $(\MP,\MQ,\gamma,\lambda,\delta,V_U,W_U)$ of $F(A,b,\ell,u,c)$ such that~\eqref{thm_milp_folding_condition} holds, i.e.:
\begin{equation*}
    \conv(\{x\in P(A,b,\ell,u) \mid (\transpose{\Pi_{\MQ}} \Lambda x)_{\MQ_I} \in \Z^{\MQ_I}\}) = \conv(P(A,b,\ell,u)\cap \M^{W_I})
\end{equation*}
in $\orderO(n+m)$ space and $\orderO((n+m)\log n)$ time using network matrices or $\orderO(n^2\alpha(n,n) + m \log(n))$ time using transposed network matrices.
\end{theorem}
\begin{proof}
We do not provide a complete proof of every step here, but instead emphasize the most relevant parts of the correctness proof and the running time. First let us prove correctness. 
    In lines \ref{algo_eos_start}-\ref{algo_eos_point1}, we compute the connected components of the continuous variables that are (transposed) network, which is necessary to compute our affine TU decomposition.
    Then, we compute an initial reflection reduction in lines \ref{algo_eos_point2}-\ref{algo_eos_first_refine}, where $\MQ_1$ satisfies the following for all $Q\in \MQ_1$:
    \begin{enumerate}
        \item Either $Q\cap W_I$ or $Q\subseteq W_I$ holds.
        \item If $Q\subseteq W_I$, then $A_{V_R,w} = \pm A_{V_R,w'}$ and $\mathrm{ntmask}(A_{V\setminus V_R,w}) = \pm \mathrm{ntmask}(A_{V\setminus V_R,w'})$ hold for all $w,w'\in Q$.
    \end{enumerate}
    In lines \ref{algo_eos_start_ATU}-\ref{algo_eos_end_ATU}, we greedily construct an affine TU decomposition that satisfies the conditions of \cref{thm_atu_milp_integral} where we keep the invariant that $T$ is equal to $(\transpose{\Pi_{\MQ^1}} \Lambda)_Q$ for some $\MQ'\subseteq \MQ^1$. The construction uses condition 2. by testing if the vector $\begin{bmatrix}
        A_{V_R,w} \\
        \mathrm{ntmask}(A_{V\setminus V_R,w})
    \end{bmatrix}$, which is identical up to $\pm 1 $ scale for all $w\in Q$ by construction, can be used to augment the (transposed) network matrix in the affine TU decomposition. Our construction does not always add such a column to $u$: we determine dynamically if this is necessary by checking if $V_{\mathrm{cancel}} \neq \emptyset$ holds.
    If an additional column in the $U$-matrix is needed, we add $Q$ to $\MQ'$ and add the row $(\transpose{\Pi_{\MQ^1}} \Lambda)_Q$ to $T$. Otherwise, we leave $\MQ'$ and $T$ unchanged. By doing so for all columns, we ensure that $T = \transpose{\Pi_{\MQ'}} \Lambda$ holds.
    For any $Q\in \MQ^1$ that can not added to the affine TU decomposition, we simply split it into single-variable parts in $\MQ^1$ to satisfy the conditions of the affine TU decomposition.
    Then, it follows from \cref{thm_atu_milp_integral} and \cref{thm_fibers_symmetry_integral} that for $\MQ^1_I = \{Q \in \MQ^1 \mid Q\subseteq W_I \}$ we have:
    \begin{equation}
    \label{eq_before_refinement_condition}
    \conv(\{x\in P(A,b,\ell,u) \mid (\transpose{\Pi_{\MQ^1}} \Lambda x)_{\MQ^1_I} \in \Z^{\MQ^1_I}\}) = \conv(P(A,b,\ell,u)\cap \M^{W_I}).
    \end{equation}
    Finally, we refine $\MQ^1$ to $\MQ^2$ in \ref{algo_eos_additional_refinements} to ensure that $\MQ^2$ corresponds to a reflection reduction.
    By \eqref{eq_before_refinement_condition} and the fact that $\MQ^2$ refines $\MQ^1$,  \cref{thm_refinement_integrality_okay} shows for $\MQ^2_I \coloneqq \{Q \in \MQ^2 \mid Q\subseteq W_I \}$ that 
    \begin{equation*}
        \conv(\{x\in P(A,b,\ell,u) \mid (\transpose{\Pi_{\MQ^2}} \Lambda x)_{\MQ^1_I} \in \Z^{\MQ^2_I}\}) = \conv(P(A,b,\ell,u)\cap \M^{W_I}).
    \end{equation*}
    Then, since $(\MP^2,\MQ^2, \lambda, \gamma, \delta, V^2_O, W^2_O)$ is a reflection reduction of $F(A,b,\ell,u,c)$, it follows that the conditions of \cref{thm_milp_folding} hold. Thus, we have shown correctness of \cref{algo_eos}.\\

    Next, let us consider the running time and space-complexity of \cref{algo_eos}. First of all, $\textsc{ReflectionRefinement}$ runs in $\orderO((n+m) \log(n))$ and $\orderO(n+m)$ space by \cref{thm_running_time_extended} and is called twice. Note that we use the fact that $\delta$ has bitlength at most $\orderO(m)$.
    
    At the start of the \cref{algo_eos}, we compute the connected components of the submatrix formed by the continuous columns. This can be done in $\orderO(n)$ time and space by using a depth-first search.
    
    Computing the initial partition $\MQ_0$ can be done in $\orderO(n+m)$ time and space by using standard hashing techniques. Sorting $\MQ^1$ can be done in $\orderO(n \log(n))$ time.
    
    Finally, we note that $\mathrm{AugmentNetwork}$ is called at most once for every column. The matrix $M$ has size $(|V| + |\MQ|)  \times |W|$, which is of the order $\orderO(n) \times \orderO(n)$. Moreover, $M$ has $\orderO(m)$ nonzeros as it has fewer nonzeros than $A$, since we only add the additional row of $(\transpose{\Pi_\MQ} \Lambda)_Q$ for all $Q \in \MQ$ if it can be used to cancel one or more rows from $A_{\star,Q}$. We note that regardless of the size of the nonzero entries, one can check in constant time if a nonzero has value $\pm 1$. 
    For network matrices the column augmentation algorithm by Bixby and Wagner then runs in $\orderO(m \alpha(m,n))$ time, where $\alpha$ is the inverse Ackermann function, which grows extremely slowly and is dominated by the order $\orderO((n+m) \log(n))$ running time for the reflection reduction. The transposed network algorithm by van der Hulst and Walter runs in $\orderO(n^2 \alpha(n,n))$ time, and dominates the total runtime in case it is used. Taking the bitlength $m$ into account, we can simplify to $\orderO(n^2 \alpha(n,n) + m\log(n))$ in this case.  
\end{proof}

\section{Computational Results}
 
\label{sec_results}

In order to validate the developed algorithms, we perform experiments on linear programming and mixed-integer linear programming problems.

\Cref{algo_eos} and the reflection reduction detection algorithm are implemented in C++ and are publicly available~\cite{VanDerHulstFolding}. The implementation of~\Cref{algo_eos} uses the network matrix augmentation algorithms described in~\cite{RowNetworkMatrixPaper} and \cite{BixbyWagner1988} to recognize totally unimodular matrices, which are also publicly available~\cite{VanDerHulstMatrec}.
 All experiments are run on a AMD Ryzen 2600 CPU with 16GB of memory. Prior to their solution, the linear programming and mixed-integer linear programming instances are presolved using PaPILO 3.0.0~\cite{GGHpapilo}. In the context of our study, presolving the instances with PaPILO has two benefits. First of all, well-known presolving reductions such as those described in~\cite{AchterbergBGRW20} may provide additional opportunities for the DRCR algorithm, and provide a realistic setting for the application of DRCR. Secondly, PaPILO removes parallel rows and columns from the problem, which would also be detected by DRCR. This way, we can be sure that the reductions detected by R-DRCR and DRCR are not detected easily by other presolving techniques.
 
 For our experiments on mixed-integer linear programs we use the SCIP 10 solver~\cite{hojny2025scipoptimizationsuite100}. We note that currently parallel column reductions are not available in SCIP 10 due to technical limitations of its API. Thus, presolving once with PaPILO before handing the problem to SCIP is necessary to provide a fair comparison of DRCR for MILP with state-of-the-art presolving techniques.
\subsection{Folding reflection symmetries for Linear Programs}
\label{sec_lp_folding_results}
As a baseline, we compare our dimension reduction algorithm against DRCR as described in~\cite{Grohe2014}. The improved DRCR algorithm as presented \cref{sec_reflection_symmetry} is denoted by R-DRCR.

We use the linear programming relaxations of problems in the MIPLIB 2017 collection dataset~\cite{Gleixner2021}, which contains 1065 instances. We exclude the 10 largest instances due to memory requirements, and any instances with the (nonlinear) indicator constraints. Furthermore, we exclude instances where DRCR and R-DRCR produce an identical reduced instance. This yields a subset of 83 instances that we use for the computational experiments. The resulting linear programs are solved using SoPlex 8.0.0~\cite{hojny2025scipoptimizationsuite100} with a time limit of 1 hour.
An overview of the results can be found in \cref{tab:drcr_lp_summary}. The complete results can be found in Appendix \ref{appendix_drcr_lp}.

\begin{table}[h]
    \centering
  \caption{Performance comparison between DRCR and R-DRCR on the 83 linear programming relaxations of instances in the MIPLIB2017 collection set where R-DRCR finds additional reductions compared to DRCR. The subsets contain the instances whose solution time for the DRCR method (in seconds) lies in the indicated bracket. The times reported are shifted geometric means in seconds, with a shift of 1 second. The iterations column indicates the number of simplex iterations, and is reported as a shifted geometric mean with a shift of 10 iterations. The faster and slower columns indicate the number of instances where R-DRCR is faster/slower than DRCR.}
    \label{tab:drcr_lp_summary}
  
    \begin{tabular}{@{}r@{\;\;\extracolsep{\fill}}rrrc@{\hspace{0.5cm}}rrrr@{}}
    \toprule
    & & \multicolumn{2}{c}{DRCR} & & \multicolumn{2}{c}{R-DRCR} &  \\
    \cmidrule{3-4} \cmidrule{6-7} 
    subset & instances & time & iterations & & time & iterations & faster & slower \\
    \midrule
all & 83 & 1.27 & 1866.0 & & 0.93 & 962.0 & 40 &  43 \\
 <0.1 & 34 & 0.02 & 142.9 & & 0.02 & 60.4 & 13 & 21 \\
\bracket{0.1}{1} & 20 & 0.35 & 2907.9 & & 0.33 & 2441.6 & 9 & 11\\
\bracket{1}{10} & 9 & 3.26 & 11942.0 & & 3.00 & 7783.7 & 5 & 4 \\
>10 & 20 & 174.02 & 37205.0 & & 58.50 & 13094.6 & 13 & 7 \\
\bottomrule
    \end{tabular}
\end{table}

The R-DRCR algorithm was also tested on the well-known Netlib linear programming dataset. However, for all of the instances in this dataset, DRCR and R-DRCR produced identical reduced instances. Thus, we omit the results on Netlib in this report.

From the results in \cref{tab:drcr_lp_summary}, it is clear that R-DRCR is more effective than DRCR on the tested subset. In all categories, the running time is smaller and the number of required simplex iterations is also reduced. The running time improvement is particularly impressive for some of the more difficult linear programs that take longer than 10 seconds to solve, with the shifted geometric mean of the running times dropping from 174.0 to 58.5 seconds. Furthermore, we observe that on many of the easier instances there is a moderate reduction in running time and a clear reduction in the number of simplex iterations.

Let us highlight a few instances where reflection refinement is particularly effective.
For the combinatorial instances \texttt{pythago7824} and \texttt{pythago7825}, the linear program is even reduced to a linear program with a single variable and row. Interestingly, these instances do not contain any (reflection) symmetries.
Another impressive result is on the large instance \texttt{neos-3402454-bohle} where the solution time is reduced from 1353 to 6 seconds, where in the latter case, the majority of the running time is spent computing the reduced problem. The original problem has over 2.8 million rows, whereas the reduced problem has roughly 80 thousand rows.
\subsection{Folding symmetries in Mixed-Integer Linear Programming}
\label{sec_milp_folding_results}
Next, we evaluate the impact of \cref{algo_eos} on the solution of mixed-integer linear programs. As before, we consider the MIPLIB 2017 collection dataset~\cite{Gleixner2021}. We exclude the 10 largest instances and any instances with indicator constraints.

For our experiments, we consider two configurations of \cref{algo_eos}. We use R-DRCR-N to indicate the version that detects network matrices, and R-DRCR-T to indicate the version that detects transposed network matrices. For both configurations, we only test those instances where some reduction is found by the algorithms. R-DRCR-N find reductions on 207 of the 1065 MIPLIB 2017 collection instances, and R-DRCR-T finds reductions on 153 instances. In total, 208 of the instances admit a reduction by at least one of R-DRCR-N and R-DRCR-T.

The solving procedure is structured as follows. First, the problem is presolved using PaPILO. Then, we apply the DRCR algorithm to further reduce the presolved problem. Then, the presolved problem is solved using SCIP 10.0~\cite{hojny2025scipoptimizationsuite100} with a time limit of 1 hour. Finally, if SCIP finds a feasible integer solution to the reduced problem that is not integer when mapped to the original variables, we use SoPlex to postsolve the found integer solution to an integer solution. All running times reported include the runtime of the DRCR algorithm, SCIP's time to solve the problem, and time required to postsolve the final solution to obtain a feasible integral solution to the original problem. We compare the R-DRCR methods against the default SCIP 10 configuration, which is run on the presolved model that is output by PaPILO.
In all cases, we use the default parameters and configuration of SCIP 10.
One important detail is that, although the DRCR algorithm may also detect implied integrality in the reduced problem, we do not communicate this information to SCIP in order to measure the effect of the reduction detected by \cref{algo_eos} in isolation. 

In \cref{tab:drcr_n_mip_summary,tab:drcr_t_mip_summary} we summarize the results of R-DRCR-N and R-DRCR-T compared to the SCIP 10 baseline on instances which were solved by at least one of the methods. In \cref{tab:PDItable} we compare the performance on the instances that were not solved within the time limit by any method. Here, we use the Primal-Dual Integral measure introduced by Berthold~\cite{Berthold2013}, which measures the average gap between the primal and dual solution during the solving process. This gives an indication for the quality of the obtained primal and dual bounds.

\begin{table}[h]

    \centering
  \caption{Performance comparison between SCIP 10 and R-DRCR-N on the solvable instances of the MIPLIB2017 collection set that are affected by R-DRCR-N. The subsets contain the instances whose solution time for SCIP 10 (in seconds) lies in the indicated bracket. The time and nodes column indicate the fraction of the shifted geometric mean of the R-DRCR method compared to the SCIP 10 baseline, where we use shifts of 1 second and 1 node respectively. The faster and slower columns indicate the number of instances where R-DRCR method is faster/slower than SCIP 10.}
    \label{tab:drcr_n_mip_summary}
    \begin{tabular}{@{}r@{\;\;\extracolsep{\fill}}rrrc@{\hspace{0.5cm}}rrrr@{}}
    \toprule
    & & SCIP 10 &  \multicolumn{3}{c}{R-DRCR-N} &  \\
    \cmidrule{4-6}
    subset & instances & solved & solved & time & nodes & faster & slower \\
    \midrule
all & 107 & 100 & 103 & 0.48 & 0.69 & 67  & 40 \\
 <10 & 25 & 25 & 25 & 1.18 & 1.40 & 15 & 10 \\
\bracket{10}{100} & 24 & 24 & 23 & 0.86 & 0.98 & 12 & 12\\
\bracket{100}{1000} & 35 & 35 & 34 & 0.61 & 0.70 & 21 & 14 \\
>1000 & 23 & 16 & 21 & 0.075 & 0.25 & 19 & 4  \\
\bottomrule
    \end{tabular}
\end{table}

\begin{table}[h]
    \centering
  
  \caption{Performance comparison between SCIP 10 and R-DRCR-T on the solvable instances of the MIPLIB2017 collection set that that are affected by R-DRCR-T.}
    \label{tab:drcr_t_mip_summary}
    \begin{tabular}{@{}r@{\;\;\extracolsep{\fill}}rrrc@{\hspace{0.5cm}}rrrr@{}}
    \toprule
    & & SCIP 10 &  \multicolumn{3}{c}{R-DRCR-T} &  \\
    \cmidrule{4-6}
    subset & instances & solved & solved & time & nodes & faster & slower \\
    \midrule
all & 77 & 75 & 76 & 0.75 & 0.86 & 47 & 30 \\
 <10 & 19 & 19 & 19 & 1.08 & 0.95 & 12 & 7  \\
\bracket{10}{100} & 14 & 14 & 14 & 0.94 & 1.09 & 8 & 6 \\
\bracket{100}{1000} & 30 & 30 & 30 & 0.80 & 0.92 & 17 & 13 \\
>1000 & 14 & 12 & 13 & 0.35 & 0.49 & 10 & 4 \\
\bottomrule
    \end{tabular}
\end{table}

In \cref{tab:drcr_n_mip_summary,tab:drcr_t_mip_summary}, we observe that R-DRCR-N dominates R-DRCR-T in both the number of affected instances and the speedups on the affected instances. Of the 208 instances where some reduction was found by either of the two methods, R-DRCR-T found a smaller reduced problem on only 8 instances. One explanation for this observation is that assignment subproblems such as those in the generalized assignment problem discussed \cref{sec_example} consist of matrices that are in general network matrices but not transposed network matrices. Otherwise, both methods have similar characteristics. Although they do not seem to help with the solution of 'easy' problems that can be solved under 10 seconds, we observe large speedups for the more difficult problems.

For R-DRCR-N, seven new instances are newly solved, and four instances are no longer solved compared to SCIP 10's default. For the seven instances that are newly solved, the R-DRCR-N algorithm reduces the size of the original system significantly. For example, for the newly solved instance \texttt{neos-631710} the size of the constraint matrix is reduced from $169576\times167056$ to $9448\times7216$.

For the four instances that are no longer solved after applying R-DRCR-N, there are various potential explanations for the loss of performance. For the two instances \texttt{supportcase34} and \texttt{neos-5129192-manaia} a strong primal solution is no longer found early in the search after the reduction is applied to the problem. For \texttt{cvs08r139-94}, it seems that the default symmetry handling methods in SCIP are more effective than DRCR. One potential explanation for this is that the symmetries and the equitable partition in this instance are relatively simple: all the nontrivial variable parts consist of two binary variables that are aggregated into one integer variable.

\begin{table}[h]
    \centering
        \caption{Comparison of SCIP 10 and R-DRCR-N and R-DRCR-T on unsolved instances of MIPLIB 2017 collection set that are affected by the R-DRCR-N or R-DRCR-T. The two different rows of the table correspond to the two different instance sets for which R-DRCR-N and R-DRCR-T found one or more reductions. The `nodes' column indicates the fractional of the shifted geometric mean of R-DRCR method compared to the SCIP 10 baseline. The SCIP 10 PDI and method PDI indicate the Primal-Dual Integral values of the SCIP 10 baseline and the R-DRCR method, respectively. }
    \label{tab:PDItable}
    \begin{tabular}{lrrrr}
        \toprule
        Method & instances & nodes & SCIP 10 PDI & method PDI \\
        \midrule
        R-DRCR-N & 100 & 1.19 &$1.30\cdot 10^5$ & $1.22\cdot 10^5$  \\
        R-DRCR-T & 76 & 1.25 & $1.23\cdot 10^5$ & $1.14\cdot 10^5$ \\
        \bottomrule
    \end{tabular}
\end{table}

In \cref{tab:PDItable} we evaluate the performance of DRCR on instances that were not solved within the time limit by any method. We observe that both DRCR methods help to reduce the Primal-Dual Integral, indicating that the DRCR helps to find strong primal and dual bounds more quickly than SCIP 10. Furthermore, we note that the DRCR-methods explore more branch-and-bound nodes within the time limit. This observation can be explained by the fact that the linear programming relaxations are easier to solve due to their reduced size.

\begin{table}[h]
    \centering
        \caption{Distribution of the time taken (in seconds) to detect and postsolve for R-DRCR-N and R-DRCR-T for the affected instances. For each bracket, the number of instances that fall into it is indicated.}
    \label{tab:drcr_timings}
    \begin{tabular}{lrrrcrrrr}
        \toprule
        & \multicolumn{3}{c}{DRCR detection} & &\multicolumn{4}{c}{postsolve}\\
        \cmidrule{2-4} \cmidrule{6-9}
        Method & <0.1 & \bracket{0.1}{1} & \bracket{1}{10} & & <0.1 & 
        \bracket{0.1}{1} & \bracket{1}{10} & >10 \\ 
        \midrule
        R-DRCR-N & 146 & 46 & 15 & & 134 & 37 & 28 & 8 \\
        R-DRCR-T & 114 & 28 & 11 & & 108 & 25 & 15 & 5 \\
        \bottomrule
    \end{tabular}

\end{table}

\cref{tab:drcr_timings} provides an overview of the distribution of presolve and postsolve times. We note that for the vast majority of the instances, the presolve and postsolve procedures are cheap. For a small number of instances, the postsolve LP solution using SoPlex took longer than 10 seconds. 

One effect that we observe in a few larger instances is that DRCR helps to reduce the running time of further symmetry detection methods. For example for the instance \texttt{neos-3322547-alsek}, symmetry detection for the original model with 1001000 columns takes 378.2 seconds. R-DRCR-N reduces the model to have only 82000 columns in 6.3 seconds. Symmetry detection on the reduced model completes in only 0.24 seconds.

To summarize our results, we observe that both R-DRCR methods are effective in reducing the solution time of general mixed-integer linear programs. The R-DRCR-N algorithm dominates the R-DRCR-T algorithm. Furthermore, R-DRCR-N solves three more instances, and both methods yield a smaller PDI on the set of instances that are not solved within the time limit. 

\section{Discussion}

We will follow the structure of the paper, and first evaluate the extension to reflection symmetries for linear programming and then the extension to mixed-integer linear programs.\\

In \cref{sec_reflection_symmetry}, we extended DRCR to reflection symmetries of linear programs. There are several ways in which the current work can still be extended.
First of all, we note that it is possible to extend the detection of reflection symmetries to not just use a primal affine offset $\delta$ but to also introduce a \emph{dual offset}. Such a dual offset alters the objective by subtracting a weighted linear combination of the rows from the objective. One hopes that doing so will lead to a coarser equitable partition by having more similar entries in the objective.
Similar to the choice of $\delta$, there are infinitely many possibilities to do so. However, unlike the primal variables, the dual variables of linear programming models are typically not explicitly bounded by the user, so there is no `obvious choice' at the center of the dual variable domains like for the primal offset. Thus, one would need to come up with some method to choose a sensible dual offset. \\

The choice of $\delta$ is somewhat justified by~\cref{thm_delta_equivalent_complemented}, which shows that it corresponds to the complementation of the variables at their bounds. However, it is still unclear whether other choices of $\delta$ may lead to better results. This is true in particular for free variables and variables with half-open domains, where any offset can be used and may change the right-hand sides, and there is no obvious center of the variable domain.

More generally, we believe it would be interesting to understand in further detail how one should apply linear and affine transformations to linear programs in order to detect arbitrary symmetries. One direction which naturally extends this research would be to detect arbitrary row and column scalings that result in coarser equitable partitions, compared to the $\pm 1$ scalings that we considered here.\\

 In \cref{sec_lp_folding_results}, we performed computational experiments on the linear programming relaxations of MIPLIB 2017 instances. As expected, reflection symmetries are somewhat uncommon as was also observed by Hojny~\cite{Hojny2025-si}, as for only 81 out of 1065 MIPLIB instances R-DRCR finds more reductions than DRCR. For these instances, we observed on average a running-time reduction for R-DRCR of roughly 27\% on the affected models compared to DRCR, with most of the benefit coming from a few large models. One limitation of the current study is that we did not include a crossover procedure to obtain a basic solution in our experiments, which is crucial in the context of solving linear programming relaxations in MILP. However, we note that since we compare R-DRCR against DRCR, a fair comparison with crossover would affect the running time of both models, so the performance implications in this case are unclear. As R-DRCR still clearly improves over DRCR without crossover it can be considered a relevant improvement for the dimension reduction of linear programs. In particular, reflection reductions can be computed in the same worst-case time complexity as DRCR by \cref{thm_running_time_extended}. Also in practice, they can be computed with only minor computational overhead using the implementation details that we describe in \cref{sec_lp_algorithm}. \\

 As the computational effort for crossover is the most important limitation for applying DRCR to linear programs in practice, one direction for future research is to investigate if it is possible to exploit symmetry to speed up the crossover procedure. Since (R)-DRCR factors out symmetries, the crossover problem contains these symmetries. It is likely that these symmetries hinder the performance of the crossover procedure. In an unpublished work, Deakins, Knueven and Ostrowski~\cite{Deakins2022} formulate a method that takes advantage of equitable partitions using an iterative lifting scheme and show that their method can reduce the time spent in crossover. It would be promising to further develop symmetry-exploiting crossover methods and incorporate them with our approach.
 More generally, one could investigate if knowledge of symmetries of a linear program can be used directly to improve the performance of the simplex algorithm without modifying the analyzed (symmetric) polytope.\\

In \cref{sec_dimred_milp}, we extended DRCR to mixed-integer linear programs and highlight that affine TU decompositions may be used to aggregate integer variables. In \cref{sec_algorithm}, we formulate an algorithm that detects an equitable partition that satisfies the necessary conditions for dimension reduction of mixed-integer linear programs. The computational results in \cref{sec_milp_folding_results} indicate that the formulated algorithm is effective in reducing the running time of an MILP solver: affected instances are solved more than twice as quickly on average.\\

The algorithm we propose for MILPs may have several additional benefits. First of all, if we are able to perform dimension reduction, we may also eliminate some of the symmetries in the instance. Although we may only eliminate a subgroup of the symmetries for MILP, this reduction may help to speed up the detection of the remaining symmetries.
Second, even if the dimension of the MILP is not reduced, our algorithm still computes an affine TU decomposition and/or implied integrality that can be used by the MIP solver. Thus, the computational effort to do so does not go to waste, regardless of whether a nontrivial equitable partition is detected.

One downside of our algorithm is that it completely rejects any equitable integer block $Q\in \MQ_I$ that can not be taken into the reduced problem, by partitioning it into its individual columns. Although this is a simple approach, it may also reject blocks that have many columns that contain a network substructure and only a few that break it. One may hope that instead of a rejecting the block completely, refining the block to compute an equitable partition and then attempting to add certain subsets of it might be more effective. However, this complicates the algorithm and requires a way to decide how equitable parts must refined in case they can not be augmented to the network matrix. \\

Our detection algorithm may benefit from detecting and exploiting additional substructures. For example, if the equitable partition $\MQ$ contains some $Q\in \MQ$ such that the constraint $\sum_{w\in Q} \lambda_w x_w = b$ appears in the constraint matrix, then the corresponding variable is fixed in the reduced problem and may be ignored for the detection of the affine TU decomposition, as it is an integral constant in the reduced problem.
Additionally, it would be interesting to combine our reduction with the aggregation of non-overlapping symmetric variables from Section 7.1 in~\cite{Achterberg2016}.
Another possibility that may enhance DRCR for mixed-integer programs is to take the viewpoint that it corresponds to simultaneously performing DRCR on all affine fibers $H_d$ that result from fixing integral expressions, where the vector $d\in \Z^k$ indicates the subspace. Our current analysis only performs reduces the dimension if DRCR can be applied for all $d\in \Z^k$. It may be possible to use the structure of the polyhedron to argue for some affine fibers $H_d$ that they are irrelevant for the MIP because they are infeasible or suboptimal, and use this insight to derive a stronger size reduction that exploits these properties. \\

For our computational experiments in \cref{sec_milp_folding_results}, we remark that we did not spend any time or effort for the postsolve procedure by solving it as a linear program over the affine fibers, without utilizing the given non-basic LP-feasible point. Although this is not a bottleneck of our approach, a crossover procedure that uses the solution as an initial feasible point should be more efficient. 
If one uses affine TU decompositions to aggregate symmetric integer variables, then a postsolve step to recover integrality of the original solutions is necessary. We emphasize that in our experiments in \cref{sec_milp_folding_results} we perform crossover only once for the best integer solution that is found. Most MILP solvers do not report just the optimal solution, but also provide a solution pool that contains the best found primal solutions. For each such a solution, one needs to call a crossover procedure. Thus, the computational costs of the crossover step may become burdensome as it is repeated for every integer solution that is reported to the user. Thus, this justifies research into fast crossover procedures that can be used in conjunction with R-DRCR-N and R-DRCR-T.\\

Since we detect if the affine fibers are given by  network matrices, the affine fibers can be transformed into a network flow problem, see~\cite{Bixby1980} for further information.
In crossover, one could also exploit this network structure by using crossover algorithms that are specialized for network flow problems~\cite{Ge2025}. The authors of~\cite{Ge2025} show that exploiting the network structure and using ideas from the network simplex algorithm for crossover may help to reduce crossover time by a factor of 10 or more compared to Gurobi's default crossover procedure on large optimal transport instances. As the crossover problem that we consider is also symmetric, the methods from~\cite{Ge2025} may be further enhanced by formulating symmetry-aware crossover strategies.\\

Considering our results, we remark that R-DRCR-N was more effective on the tested instances than R-DRCR-T. One explanation for this observation is that the structure of the affine TU decomposition that we observe for the generalized assignment problem corresponds to an assignment problem, which is given by a constraint matrix that is a network matrix but not a transposed network matrix. \\

Our approach can also be interpreted as \emph{exact orbital shrinking} where the orbital shrinking subproblem is a perfect formulation. Salvagnin~\cite{Salvagnin2013} observed that orbital shrinking may be significantly more powerful than exact orbital shrinking if one aggregates a larger symmetry group, which leads to a more difficult subproblem that is not necessarily a perfect formulation. Further research could investigate how one can better navigate the trade-off between stronger bounds and the difficulty of the subproblem for orbital shrinking. A crucial component here is efficient solution strategies for the orbital shrinking subproblem, which still contains the symmetries that are factored out of the original problem. As also remarked by Salvagnin~\cite{Salvagnin2013}, one could also view orbital shrinking as a logic-based Benders' decomposition and exploit infeasible subproblems by providing cutting planes that cut off infeasible solutions.

\medskip 
\noindent
\textbf{Acknowledgements.}
I am grateful to Matthias Walter for his feedback that helped improve the quality of this work, and I would like to thank Christopher Hojny for recommending relevant literature.

The author acknowledges funding support from the Netherlands Organisation for Scientific Research (NWO) on grant number OCENW.M20.151: ``Making Mixed-Integer Programming Solvers Smarter and Faster using Network Matrices".

\bibliography{exactorbitalshrinking}

@ARTICLE{AchterbergBGRW20,
  author = {Achterberg, Tobias and Bixby, Robert E. and Gu, Zonghao and Rothberg,
	Edward and Weninger, Dieter},
  title = {Presolve Reductions in Mixed Integer Programming},
  journal = {INFORMS Journal on Computing},
  year = {2020},
  volume = {32},
  pages = {473-506},
  number = {2},
  doi = {10.1287/ijoc.2018.0857}
}

@InCollection{HoffmanK56,
  author    = {Hoffman, Alan J. and Kruskal, Joseph B.},
  booktitle = {Linear Inequalities and Related Systems},
  title     = {Integral Boundary Points of Convex Polyhedra},
  doi       = {10.1515/9781400881987-014},
  pages     = {223--246},
  publisher = {Princeton University Press},
  volume    = {38},
  comment   = {cite for integrality of polyhedra due to total unimodularity},
  year      = {1956},
  address = {Princeton, New Jersey}
}

@BOOK{Schrijver86,
  title = {Theory of Linear and Integer Programming},
  publisher = {John Wiley \& Sons, Inc.},
  year = {1986},
  author = {Schrijver, Alexander},
  address = {New York, NY, USA},
}

@ARTICLE{Meyer74,
  author = {Meyer, Robert R.},
  title = {On the existence of optimal solutions to integer and mixed-integer
	programming problems},
  journal = {Mathematical Programming},
  year = {1974},
  volume = {7},
  pages = {223--235},
  day = {01},
  doi = {10.1007/BF01585518},
  issn = {1436-4646}
}

@article{Gleixner2021,
author = {Gleixner, Ambros and et al.},
doi = {10.1007/s12532-020-00194-3},
issn = {1867-2949},
journal = {Mathematical Programming Computation},
number = {3},
pages = {443--490},
title = {MIPLIB 2017: data-driven compilation of the 6th mixed-integer programming library},
volume = {13},
year = {2021}
}

@article{BixbyWagner1988,
author = {Bixby, Robert E. and Wagner, Donald K.},
journal = {Mathematics of Operations research},
number = {1},
title = {An almost linear-time algorithm for graph realization},
volume = {13},
year = {1988},
doi = {10.1287/moor.13.1.99}
}

@article{Bixby1980,
author = {Bixby, Robert E. and Cunningham, William H.},
doi = {10.1287/moor.5.3.321},
issn = {0364-765X},
journal = {Mathematics of Operations Research},
number = {3},
pages = {321--357},
title = {Converting Linear Programs to Network Problems},
volume = {5},
year = {1980}
}

@article{Walter2013,
author = {Walter, Matthias and Truemper, Klaus},
doi = {10.1007/s12532-012-0048-x},
journal = {Mathematical Programming Computation},
number = {1},
pages = {57--73},
title = {Implementation of a unimodularity test},
volume = {5},
year = {2013}
}

@article{Truemper1990,
author = {Truemper, Klaus},
doi = {10.1016/0095-8956(90)90030-4},
issn = {00958956},
journal = {Journal of Combinatorial Theory, Series B},
number = {2},
pages = {241--281},
title = {A decomposition theory for matroids. V. Testing of matrix total unimodularity},
volume = {49},
year = {1990}
}

@TechReport{Gomory60,
  author      = {Gomory, Ralph E.},
  institution = {RAND Corporation},
  title       = {An algorithm for the mixed-integer problem},
  number      = {RM-2597},
  comment     = {cite for Gomory mixed-integer cuts},
  year        = {1960},
}

@book{Ziegler01,
  title = {Lectures on Polytopes (Graduate Texts in Mathematics)},
  publisher = {Springer},
  year = {2001},
  author = {G{\"u}nter M. Ziegler},
  address = {New York}
}

@article{BaderHWZ18,
author = {Bader, J{\"{o}}rg and Hildebrand, Robert and Weismantel, Robert and Zenklusen, Rico},
doi = {10.1007/s10107-017-1147-2},
eprint = {1508.02940},
isbn = {1010701711},
issn = {14364646},
journal = {Mathematical Programming},
number = {2},
pages = {565--584},
publisher = {Springer Berlin Heidelberg},
title = {Mixed integer reformulations of integer programs and the affine TU-dimension of a matrix},
volume = {169},
year = {2018}
}

@techreport{Achterberg2016,
address = {Berlin},
author = {Achterberg, Tobias and Bixby, Robert E. and Gu, Zonghao and Rothberg, Edward and Weninger, Dieter},
institution = {Zuse Institute Berlin},
title = {Presolve Reductions in Mixed Integer Programming},
year = {2016}
}

@article{Grohe2014,
archivePrefix = {arXiv},
arxivId = {1307.5697},
author = {Grohe, Martin and Kersting, Kristian and Mladenov, Martin and Selman, Erkal},
doi = {10.1007/978-3-662-44777-2\_42},
eprint = {1307.5697},
isbn = {9783662447765},
issn = {16113349},
journal = {Lecture Notes in Computer Science (including subseries Lecture Notes in Artificial Intelligence and Lecture Notes in Bioinformatics)},
pages = {505--516},
title = {Dimension reduction via colour refinement},
volume = {8737 LNCS},
year = {2014}
}

@misc{Berthold2018,
author = {Berthold, Timo},
title = {How to Fold a linear programming Problem},
url = {https://community.fico.com/s/blog-post/a5Q80000000Drp2EAC/fico1299},
urldate = {13 October 2025},
year = {2018}
}

@misc{GurobiOptimization2020,
author = {Gurobi Optimization},
title = {What's new in the Gurobi Optimizer 9.1?},
url = {https://www.youtube.com/watch?v=EwuEPu-nBmg},
urldate = {2025-10-13},
year = {2020}
}

@article{Khachiyan1980,
author = {Khachiyan, L.G.},
doi = {10.1016/0041-5553(80)90061-0},
issn = {00415553},
journal = {USSR Computational Mathematics and Mathematical Physics},
month = {jan},
number = {1},
pages = {53--72},
title = {Polynomial algorithms in linear programming},
volume = {20},
year = {1980}
}

@article{Grotschel1981,
author = {Gr{\"{o}}tschel, M. and Lov{\'{a}}sz, L. and Schrijver, A.},
doi = {10.1007/BF02579273},
issn = {0209-9683},
journal = {Combinatorica},
month = {jun},
number = {2},
pages = {169--197},
title = {The ellipsoid method and its consequences in combinatorial optimization},
volume = {1},
year = {1981}
}

@article{Bland1981,
author = {Bland, Robert G and Goldfarb, Donald and Todd, Michael J},
doi = {10.1287/opre.29.6.1039},
journal = {Operations Research},
number = {6},
pages = {1039--1091},
title = {The Ellipsoid Method: A Survey},
volume = {29},
year = {1981}
}

@ARTICLE{Hojny2025-si,
  title     = "Detecting and handling reflection symmetries in mixed-integer
               (nonlinear) programming and beyond",
  author    = "Hojny, Christopher",
  journal   = "Mathematical Programming Computation",
  publisher = "Springer Science and Business Media LLC",
  month     =  jul,
  year      =  2025,
  copyright = "https://creativecommons.org/licenses/by/4.0",
  language  = "en",
  doi = {10.1007/s12532-025-00289-9}
}

@article{Bolusani2024,
archivePrefix = {arXiv},
arxivId = {2402.17702},
author = {Bolusani, Suresh and Besan{\c{c}}on, Mathieu and Bestuzheva, Ksenia and Chmiela, Antonia and Dion{\'{i}}sio, Jo{\~{a}}o and Donkiewicz, Tim and van Doornmalen, Jasper and Eifler, Leon and Ghannam, Mohammed and Gleixner, Ambros and Graczyk, Christoph and Halbig, Katrin and Hedtke, Ivo and Hoen, Alexander and Hojny, Christopher and van der Hulst, Rolf and Kamp, Dominik and Koch, Thorsten and Kofler, Kevin and Lentz, Jurgen and Manns, Julian and Mexi, Gioni and M{\"{u}}hmer, Erik and Pfetsch, Marc E. and Schl{\"{o}}sser, Franziska and Serrano, Felipe and Shinano, Yuji and Turner, Mark and Vigerske, Stefan and Weninger, Dieter and Xu, Lixing},
eprint = {2402.17702},
number = {05},
pages = {1--36},
title = {The SCIP Optimization Suite 9.0},
year = {2024}
}

@article{Bodi2013,
archivePrefix = {arXiv},
arxivId = {1012.4941},
author = {B{\"{o}}di, Richard and Herr, Katrin and Joswig, Michael},
doi = {10.1007/s10107-011-0487-6},
eprint = {1012.4941},
isbn = {1010701104876},
issn = {0025-5610},
journal = {Mathematical Programming},
month = {feb},
number = {1-2},
pages = {65--90},
title = {Algorithms for highly symmetric linear and integer programs},
url = {http://link.springer.com/10.1007/s10107-011-0487-6},
volume = {137},
year = {2013}
}

@book{Gatermann2004,
archivePrefix = {arXiv},
arxivId = {math/0211450},
author = {Gatermann, Karin and Parrilo, Pablo A.},
booktitle = {Journal of Pure and Applied Algebra},
doi = {10.1016/j.jpaa.2003.12.011},
eprint = {0211450},
isbn = {4930841852},
issn = {00224049},
number = {1-3},
pages = {95--128},
primaryClass = {math},
title = {Symmetry groups, semidefinite programs, and sums of squares},
volume = {192},
year = {2004}
}

@incollection{Fischetti2012,
author = {Fischetti, Matteo and Liberti, Leo},
booktitle = {Lecture Notes in Computer Science (including subseries Lecture Notes in Artificial Intelligence and Lecture Notes in Bioinformatics)},
doi = {10.1007/978-3-642-32147-4\_6},
isbn = {9783642321467},
issn = {03029743},
pages = {48--58},
publisher = {Springer},
title = {Orbital Shrinking},
volume = {7422 LNCS},
year = {2012},
address = {Berlin, Heidelberg}
}

@article{Fischetti2017,
author = {Fischetti, Matteo and Liberti, Leo and Salvagnin, Domenico and Walsh, Toby},
doi = {10.1016/j.dam.2017.01.015},
issn = {0166218X},
journal = {Discrete Applied Mathematics},
month = {may},
pages = {109--123},
publisher = {Elsevier B.V.},
title = {Orbital shrinking: Theory and applications},
volume = {222},
year = {2017}
}

@inproceedings{Salvagnin2012,
author = {Salvagnin, Domenico and Walsh, Toby},
booktitle = {Lecture Notes in Computer Science (including subseries Lecture Notes in Artificial Intelligence and Lecture Notes in Bioinformatics)},
doi = {10.1007/978-3-642-33558-7\_46},
isbn = {9783642335570},
issn = {03029743},
pages = {633--646},
title = {A hybrid MIP/CP approach for multi-activity shift scheduling},
volume = {7514 LNCS},
year = {2012}
}

@incollection{Salvagnin2013,
author = {Salvagnin, Domenico},
booktitle = {Lecture Notes in Computer Science (including subseries Lecture Notes in Artificial Intelligence and Lecture Notes in Bioinformatics)},
doi = {10.1007/978-3-642-38171-3\_14},
isbn = {9783642381706},
issn = {03029743},
pages = {204--215},
title = {Orbital Shrinking: A New Tool for Hybrid MIP/CP Methods},
volume = {7874 LNCS},
year = {2013}
}

@article{Pfetsch2019,
author = {Pfetsch, Marc E. and Rehn, Thomas},
doi = {10.1007/s12532-018-0140-y},
issn = {1867-2949},
journal = {Mathematical Programming Computation},
month = {mar},
number = {1},
pages = {37--93},
title = {A computational comparison of symmetry handling methods for mixed integer programs},
volume = {11},
year = {2019}
}

@unpublished{Christophel2014,
archivePrefix = {arXiv},
arxivId = {1408.4017},
author = {Christophel, Philipp M and G{\"{u}}zelsoy, Menal and P{\'{o}}lik, Imre},
eprint = {1408.4017},
month = {aug},
title = {New symmetries in mixed-integer linear optimization: Symmetry heuristics and complement-based symmetries},
year = {2014}
}

@article{Liberti2014,
author = {Liberti, Leo and Ostrowski, James},
doi = {10.1007/s10898-013-0106-6},
issn = {0925-5001},
journal = {Journal of Global Optimization},
month = {oct},
number = {2},
pages = {183--194},
title = {Stabilizer-based symmetry breaking constraints for mathematical programs},
volume = {60},
year = {2014}
}

@article{Liberti2012,
author = {Liberti, Leo},
doi = {10.1007/s10107-010-0351-0},
issn = {0025-5610},
journal = {Mathematical Programming},
month = {feb},
number = {1-2},
pages = {273--304},
title = {Reformulations in mathematical programming: automatic symmetry detection and exploitation},
volume = {131},
year = {2012}
}

@article{Margot2003,
author = {Margot, Fran{\c{c}}ois},
doi = {10.1007/s10107-003-0394-6},
issn = {0025-5610},
journal = {Mathematical Programming},
month = {sep},
number = {1-3},
pages = {3--21},
title = {Exploiting orbits in symmetric ILP},
volume = {98},
year = {2003}
}

@article{Margot2001, title={Pruning by isomorphism in branch-and-cut}, volume={94}, ISSN={1436-4646}, url={http://dx.doi.org/10.1007/s10107-002-0358-2}, DOI={10.1007/s10107-002-0358-2}, number={1}, journal={Mathematical Programming}, publisher={Springer Science and Business Media LLC}, author={Margot, Fran�ois}, year={2002}, month=dec, pages={71–90} }

@article{Ostrowski2011,
author = {Ostrowski, James and Linderoth, Jeff and Rossi, Fabrizio and Smriglio, Stefano},
doi = {10.1007/s10107-009-0273-x},
issn = {0025-5610},
journal = {Mathematical Programming},
month = {jan},
number = {1},
pages = {147--178},
title = {Orbital branching},
volume = {126},
year = {2011}
}

@incollection{Ostrowski2008,
address = {Berlin, Heidelberg},
author = {Ostrowski, James and Linderoth, Jeff and Rossi, Fabrizio and Smriglio, Stefano},
booktitle = {Integer Programming and Combinatorial Optimization},
doi = {10.1007/978-3-540-68891-4\_16},
isbn = {3540688862},
issn = {03029743},
pages = {225--239},
publisher = {Springer Berlin Heidelberg},
title = {Constraint Orbital Branching},
volume = {5035 LNCS},
year = {2008}
}

@inbook{Margot2010,
address = {Berlin, Heidelberg},
author = {Margot, Fran{\c{c}}ois},
booktitle = {50 Years of Integer Programming 1958-2008: From the Early Years to the State-of-the-Art},
doi = {10.1007/978-3-540-68279-0\_17},
editor = {J{\"{u}}nger, Michael and Liebling, Thomas M and Naddef, Denis and Nemhauser, George L and Pulleyblank, William R and Reinelt, Gerhard and Rinaldi, Giovanni and Wolsey, Laurence A},
isbn = {978-3-540-68279-0},
pages = {647--686},
publisher = {Springer Berlin Heidelberg},
title = {Symmetry in Integer Linear Programming},
url = {https://doi.org/10.1007/978-3-540-68279-0\_17},
year = {2010}
}

@article{Read1977,
author = {Read, Ronald C. and Corneil, Derek G.},
doi = {10.1002/jgt.3190010410},
issn = {0364-9024},
journal = {Journal of Graph Theory},
month = {dec},
number = {4},
pages = {339--363},
title = {The graph isomorphism disease},
volume = {1},
year = {1977}
}

@misc{Darga,
author = {Darga, Paul T. and Katebi, Hadi and Liffiton, Mark and Markov, Igor L. and Sakallah, Karem},
title = {Saucy3: Fast Symmetry Discovery in Graphs},
url = {http://vlsicad.eecs.umich.edu/BK/SAUCY/},
urldate = {6 November 2025}
}

@misc{Junttila,
author = {Junttila, Tommi and Kaski, Petteri},
title = {bliss: A Tool for Computing Automorphism Groups and Canonical Labelings of Graphs},
url = {http://www.tcs.hut.fi/Software/bliss/},
urldate = {6 November 2025}
}

@misc{McKay2025,
author = {McKay, Brendan D and Piperno, Adolfo},
title = {Nauty User's Guide (version 2.7)},
url = {https://users.cecs.anu.edu.au/$\sim$bdm/nauty/},
urldate = {6 November 2025},
year = {2025}
}

@incollection{Hopcroft1971,
author = {Hopcroft, John},
booktitle = {Theory of Machines and Computations},
doi = {10.1016/B978-0-12-417750-5.50022-1},
number = {1},
pages = {189--196},
publisher = {Elsevier},
title = {An n log n algorithm for minimizing states in a finite automaton},
year = {1971},
volume = {7874 LNCS},
address = {Haifa, Israel}
}

@article{Paige1987,
author = {Paige, Robert and Tarjan, Robert E.},
doi = {10.1137/0216062},
issn = {0097-5397},
journal = {SIAM Journal on Computing},
month = {dec},
number = {6},
pages = {973--989},
title = {Three Partition Refinement Algorithms},
volume = {16},
year = {1987}
}

@article{Geyer2019,
author = {Geyer, Andrew J. and Bulutoglu, Dursun A. and Ryan, Kenneth J.},
doi = {10.1016/j.disopt.2019.01.001},
issn = {15725286},
journal = {Discrete Optimization},
pages = {93--119},
publisher = {Elsevier B.V.},
title = {Finding the symmetry group of an LP with equality constraints and its application to classifying orthogonal arrays},
url = {https://doi.org/10.1016/j.disopt.2019.01.001},
volume = {32},
year = {2019}
}

@incollection{Bachoc2012,
archivePrefix = {arXiv},
arxivId = {1007.2905},
author = {Bachoc, Christine and Gijswijt, Dion C. and Schrijver, Alexander and Vallentin, Frank},
booktitle = {Handbook of Semidefinite, Conic and polynomial Optimization},
doi = {10.1007/978-1-4614-0769-0\_9},
eprint = {1007.2905},
issn = {08848289},
pages = {219-270},
title = {Invariant Semidefinite Programs},
publisher= {Springer},
volume = {166},
year = {2012},
address = {New York}
}

@article{Permenter2020,
archivePrefix = {arXiv},
arxivId = {1608.02090},
author = {Permenter, Frank and Parrilo, Pablo A.},
doi = {10.1007/s10107-019-01372-5},
eprint = {1608.02090},
issn = {0025-5610},
journal = {Mathematical Programming},
month = {may},
number = {1},
pages = {51--84},
title = {Dimension reduction for semidefinite programs via Jordan algebras},
volume = {181},
year = {2020}
}

@article{Gemander2020,
author = {Gemander, Patrick and Chen, Wei-Kun and Weninger, Dieter and Gottwald, Leona and Gleixner, Ambros and Martin, Alexander},
doi = {10.1007/s13675-020-00129-6},
isbn = {1367502000129},
issn = {21924406},
journal = {EURO Journal on Computational Optimization},
month = {oct},
number = {3-4},
pages = {205--240},
publisher = {Elsevier},
title = {Two-row and two-column mixed-integer presolve using hashing-based pairing methods},
volume = {8},
year = {2020}
}

@article{Herr2013,
archivePrefix = {arXiv},
arxivId = {1202.0435},
author = {Herr, Katrin and Rehn, Thomas and Sch{\"{u}}rmann, Achill},
doi = {10.1016/j.orl.2013.02.007},
eprint = {1202.0435},
issn = {01676377},
journal = {Operations Research Letters},
number = {3},
pages = {298--304},
title = {Exploiting symmetry in integer convex optimization using core points},
volume = {41},
year = {2013}
}

@inproceedings{VanDerHulst2025,
address = {Cham},
author = {{van der Hulst}, Rolf and Walter, Matthias},
booktitle = {Integer Programming and Combinatorial Optimization},
isbn = {978-3-031-93112-3},
pages = {452--465},
publisher = {Springer},
title = {Implied Integrality in Mixed-Integer Optimization},
year = {2025},
doi ={10.1007/978-3-031-93112-3\_33}
}

@article{Megiddo1991,
author = {Megiddo, Nimrod},
doi = {10.1287/ijoc.3.1.63},
issn = {0899-1499},
journal = {ORSA Journal on Computing},
month = {feb},
number = {1},
pages = {63--65},
title = {On Finding Primal- and Dual-Optimal Bases},
volume = {3},
year = {1991}
}

@article{Ge2025,
archivePrefix = {arXiv},
arxivId = {2102.09420},
author = {Ge, Dongdong and Wang, Chengwenjian and Xiong, Zikai and Ye, Yinyu},
doi = {10.1287/ijoc.2022.0291},
eprint = {2102.09420},
issn = {1091-9856},
journal = {INFORMS Journal on Computing},
month = {mar},
number = {December},
publisher = {http://pubsonline.informs.org/journal/ijoc},
title = {From an Interior Point to a Corner Point: Smart Crossover},
year = {2025}
}

@article{RowNetworkMatrixPaper,
archivePrefix = {arXiv},
arxivId = {2408.12869},
author = {{van der Hulst}, Rolf and Walter, Matthias},
eprint = {2408.12869},
month = {jun},
pages = {1--44},
title = {A Row-wise Algorithm for Graph Realization},
url = {http://arxiv.org/abs/2408.12869},
volume = {1},
year = {2025}
}

@article{Berkholz2017,
archivePrefix = {arXiv},
arxivId = {1509.08251},
author = {Berkholz, Christoph and Bonsma, Paul and Grohe, Martin},
doi = {10.1007/s00224-016-9686-0},
eprint = {1509.08251},
issn = {14330490},
journal = {Theory of Computing Systems},
number = {4},
pages = {581--614},
title = {Tight Lower and Upper Bounds for the Complexity of Canonical Colour Refinement},
volume = {60},
year = {2017}
}

@misc{hojny2025scipoptimizationsuite100,
      title={The SCIP Optimization Suite 10.0}, 
      author={Christopher Hojny and Mathieu Besançon and Ksenia Bestuzheva and Sander Borst and João Dionísio and Johannes Ehls and Leon Eifler and Mohammed Ghannam and Ambros Gleixner and Adrian Göß and Alexander Hoen and Jacob von Holly-Ponientzietz and Rolf {van der Hulst} and Dominik Kamp and Thorsten Koch and Kevin Kofler and Jurgen Lentz and Marco Lübbecke and Stephen J. Maher and Paul Matti Meinhold and Gioni Mexi and Til Mohr and Erik Mühmer and Krunal Kishor Patel and Marc E. Pfetsch and Sebastian Pokutta and Chantal Reinartz Groba and Felipe Serrano and Yuji Shinano and Mark Turner and Stefan Vigerske and Matthias Walter and Dieter Weninger and Liding Xu},
      year={2025},
      eprint={2511.18580},
      archivePrefix={arXiv},
      primaryClass={math.OC},
      url={https://arxiv.org/abs/2511.18580}, 
}

@article{GGHpapilo,
author    = {Ambros Gleixner and Leona Gottwald and Alexander Hoen},
publisher = {INFORMS Journal on Computing},
title     = {{PaPILO}: A Parallel Presolving Library for Integer and Linear Programming with Multiprecision Support},
year      = {2023},
journal   = {INFORMS Journal on Computing},
doi       = {10.1287/ijoc.2022.0171.cd},
url =       {https://github.com/INFORMSJoC/2022.0171},
}

@article{Berthold2013,
title = {Measuring the impact of primal heuristics},
journal = {Operations Research Letters},
volume = {41},
number = {6},
pages = {611-614},
year = {2013},
issn = {0167-6377},
doi = {https://doi.org/10.1016/j.orl.2013.08.007},
author = {Timo Berthold}
}

@unpublished{Deakins2022,
author = {Deakins, Ethan and Knueven, Bernard and Ostrowski, Jim},
title = {Orbital Crossover},
year = {2022},
url = {https://optimization-online.org/2022/12/orbital-crossover/}
}

@misc{VanDerHulstMatrec,
  author = {{van der Hulst}, Rolf},
  title = {MATREC: Matrix recognition algorithms},
  year = {2024},
  publisher = {GitHub},
  journal = {GitHub repository},
  howpublished = {\url{https://github.com/rolfvdhulst/matrec}},
  url = {https://github.com/rolfvdhulst/matrec}
}

@misc{VanDerHulstFolding,
  author = {{van der Hulst}, Rolf},
  title = {Folding Mixed-Integer Linear Programs and Reflection Symmetries},
  year = {2026},
  publisher = {GitHub},
  journal = {GitHub repository},
  howpublished = {\url{https://github.com/rolfvdhulst/folding}},
  url = {https://github.com/rolfvdhulst/folding}
}

\begin{appendices}
\newpage
\section{Experimental results}
\label{appendix_drcr_lp}
\tiny{
\begin{longtable}{l|r|r|r|r|r|r|r|r|r|r}
    \caption{Comparison of DRCR and R-DRCR on the linear relaxation of MIPLIB2017 instances.}
    \label{tab:complete_lp_drcr_results} \\
     & \multicolumn{2}{c|}{Original} & \multicolumn{4}{c|}{DRCR} & \multicolumn{4}{c}{R-DRCR} \\
    \hline
instance & rows & columns & time & iters & rows & columns & time & iters & rows & columns \\
\hline
academictimetablebig           & 146304 & 154686 & 155.04 & 36162 & 130415 & 121047 & 259.45 & 38111 & 121759 & 111881\\
academictimetablesmall         & 17258 & 25550 & 0.59 & 2969 & 13727 & 9096 & 0.88 & 3030 & 13442 & 8845\\
bab3                           & 22500 & 393401 & 1085.09 & 353260 & 22500 & 393401 & 1133.28 & 350161 & 22578 & 393401\\
brazil3                        & 3876 & 13002 & 7.46 & 24655 & 3278 & 5604 & 3.25 & 15098 & 2842 & 3642\\
bts4-cta                       & 34301 & 74259 & 0.62 & 5578 & 34013 & 74115 & 0.63 & 5579 & 34013 & 74115\\
cryptanalysiskb128n5obj14      & 68508 & 33549 & 11.53 & 20225 & 50252 & 24417 & 12.77 & 21498 & 50220 & 24401\\
cryptanalysiskb128n5obj16      & 68508 & 33549 & 12.34 & 21756 & 50252 & 24417 & 10.28 & 19876 & 50220 & 24401\\
eva1aprime5x5opt               & 11737 & 1520 & 12.01 & 28688 & 9623 & 1520 & 4.24 & 17575 & 7270 & 1141\\
eva1aprime6x6opt               & 34464 & 3106 & 78.44 & 63347 & 23945 & 3106 & 34.74 & 50149 & 19227 & 2359\\
fastxgemm-n2r6s0t2             & 4462 & 784 & 0.00 & 7 & 24 & 10 & 0.00 & 3 & 20 & 8\\
fastxgemm-n2r7s4t1             & 5180 & 904 & 0.00 & 7 & 24 & 10 & 0.00 & 3 & 20 & 8\\
fastxgemm-n3r21s3t6            & 158132 & 18684 & 0.04 & 7 & 24 & 10 & 0.04 & 3 & 20 & 8\\
fastxgemm-n3r22s4t6            & 165590 & 19539 & 0.05 & 7 & 24 & 10 & 0.04 & 3 & 20 & 8\\
fastxgemm-n3r23s5t6            & 173048 & 20394 & 0.05 & 7 & 24 & 10 & 0.04 & 3 & 20 & 8\\
fhnw-binpack4-48               & 4480 & 3710 & 0.04 & 3062 & 4480 & 3710 & 0.02 & 0 & 3885 & 3083\\
fhnw-binpack4-4                & 620 & 520 & 0.00 & 549 & 620 & 520 & 0.00 & 0 & 542 & 431\\
fiball                         & 2387 & 32899 & 0.11 & 2570 & 878 & 4948 & 0.06 & 1831 & 878 & 4878\\
graphdraw-grafo2               & 203455 & 9258 & 0.33 & 1087 & 40068 & 2813 & 0.45 & 1515 & 40067 & 2813\\
highschool1-aigio              & 67773 & 299378 & 3602.72 & 203703 & 54501 & 154250 & 3600.71 & 229931 & 54482 & 154227\\
ic97\_tension                   & 202 & 334 & 0.00 & 27 & 201 & 333 & 0.00 & 31 & 379 & 332\\
icir97\_potential               & 1657 & 1945 & 0.01 & 687 & 1657 & 1945 & 0.01 & 586 & 2698 & 1637\\
icir97\_tension                 & 880 & 1119 & 0.01 & 16 & 820 & 1057 & 0.01 & 18 & 1574 & 1039\\
in                             & 1504669 & 1444973 & 3606.62 & 144868 & 1504365 & 1444707 & 3607.59 & 130656 & 1503062 & 1443561\\
irish-electricity              & 70342 & 38611 & 292.93 & 85415 & 66825 & 36122 & 210.96 & 72675 & 60894 & 32221\\
kottenpark09                   & 119513 & 1315854 & 3603.47 & 53892 & 118541 & 1299889 & 3604.04 & 66638 & 118538 & 1299793\\
lectsched-1                    & 9206 & 9350 & 0.04 & 0 & 9108 & 9252 & 0.05 & 0 & 17935 & 9123\\
lectsched-2                    & 4894 & 5013 & 0.02 & 0 & 4852 & 4971 & 0.03 & 0 & 9557 & 4903\\
lectsched-3                    & 8185 & 8321 & 0.03 & 0 & 8104 & 8240 & 0.04 & 0 & 15979 & 8137\\
lectsched-4-obj                & 3102 & 3187 & 0.00 & 115 & 2950 & 3035 & 0.02 & 116 & 5285 & 2989\\
lectsched-5-obj                & 9453 & 9582 & 0.05 & 421 & 8822 & 8951 & 0.06 & 432 & 15848 & 8833\\
liu                            & 2178 & 1154 & 0.02 & 560 & 2178 & 1154 & 0.00 & 1 & 1089 & 563\\
maxgasflow                     & 5786 & 5974 & 0.11 & 4463 & 5473 & 5658 & 0.11 & 4015 & 5219 & 5366\\
neos-1324574                   & 5763 & 5220 & 0.01 & 364 & 1247 & 1093 & 0.01 & 523 & 1036 & 910\\
neos-1330346                   & 4177 & 2628 & 0.18 & 2628 & 4177 & 2628 & 0.06 & 1340 & 2794 & 1747\\
neos-2974461-ibar              & 209853 & 210174 & 3600.76 & 112419 & 209853 & 210174 & 3601.85 & 158103 & 208249 & 209304\\
neos-3009394-lami              & 1677 & 1353 & 0.00 & 169 & 419 & 341 & 0.00 & 93 & 220 & 179\\
neos-3068746-nene              & 4409 & 4520 & 0.24 & 2482 & 4228 & 4415 & 0.21 & 2178 & 4123 & 4356\\
neos-3211096-shag              & 5818 & 4271 & 0.00 & 35 & 250 & 163 & 0.00 & 19 & 193 & 82\\
neos-3355323-arnon             & 11010 & 10128 & 0.07 & 3564 & 8106 & 7224 & 0.01 & 21 & 131 & 62\\
neos-3402294-bobin             & 32932 & 780 & 0.69 & 978 & 32932 & 780 & 0.04 & 96 & 4259 & 132\\
neos-3402454-bohle             & 2881228 & 2496 & 1353.86 & 25314 & 2881228 & 2496 & 5.78 & 796 & 81413 & 990\\
neos-3656078-kumeu             & 9614 & 10255 & 0.14 & 4269 & 8266 & 8831 & 0.24 & 5338 & 10850 & 8477\\
neos-4306827-ravan             & 120429 & 67539 & 0.28 & 2680 & 29549 & 14519 & 0.28 & 1415 & 25145 & 10790\\
neos-4338804-snowy             & 1470 & 1323 & 0.01 & 340 & 1470 & 1323 & 0.00 & 0 & 1048 & 888\\
neos-4754521-awarau            & 128319 & 16140 & 14.31 & 0 & 117152 & 15099 & 10.49 & 0 & 84952 & 15086\\
neos-4797081-pakoka            & 6402 & 11777 & 49.80 & 21872 & 6402 & 11777 & 44.89 & 21164 & 6664 & 11777\\
neos-5075914-elvire            & 2629 & 2602 & 0.05 & 1310 & 2531 & 2504 & 0.05 & 1360 & 2532 & 2388\\
neos-5076235-embley            & 6336 & 23618 & 0.64 & 4391 & 6238 & 23569 & 0.61 & 4005 & 5783 & 22239\\
neos-5078479-escaut            & 2783 & 2602 & 0.05 & 1340 & 2765 & 2584 & 0.05 & 1386 & 2912 & 2559\\
neos-5079731-flyers            & 6335 & 23618 & 0.52 & 3968 & 6237 & 23569 & 0.71 & 5123 & 5691 & 21973\\
neos-5093327-huahum            & 5988 & 19252 & 0.34 & 4276 & 5922 & 19220 & 0.44 & 4446 & 5808 & 18992\\
neos-5100895-inster            & 4802 & 14000 & 0.20 & 2325 & 4746 & 13972 & 0.24 & 2278 & 4460 & 13004\\
neos-5102383-irwell            & 6629 & 24500 & 0.51 & 3877 & 6531 & 24451 & 1.07 & 6802 & 5972 & 22559\\
neos-5107597-kakapo            & 3249 & 3045 & 0.01 & 200 & 3249 & 3045 & 0.00 & 158 & 3513 & 1606\\
neos-5115478-kaveri            & 3249 & 3045 & 0.01 & 198 & 3249 & 3045 & 0.01 & 140 & 3513 & 1606\\
neos-5125849-lopori            & 397 & 8074 & 0.02 & 132 & 327 & 6170 & 0.02 & 77 & 205 & 3470\\
neos-5129192-manaia            & 530672 & 161799 & 2.31 & 1115 & 530672 & 161799 & 2.08 & 40 & 529939 & 159533\\
neos-780889                    & 25866 & 112634 & 1.60 & 8074 & 17032 & 74453 & 6.93 & 13524 & 16873 & 73785\\
neos-826650                    & 2245 & 5024 & 0.00 & 94 & 137 & 157 & 0.00 & 83 & 128 & 144\\
neos-983171                    & 6669 & 7269 & 2.22 & 14068 & 3556 & 3871 & 3.15 & 17258 & 3556 & 3870\\
neos9                          & 31600 & 81408 & 0.03 & 82 & 332 & 672 & 0.04 & 38 & 166 & 336\\
nh97\_potential                 & 958 & 1100 & 0.00 & 227 & 958 & 1100 & 0.00 & 130 & 564 & 354\\
ns1905797                      & 51600 & 18180 & 0.39 & 1097 & 38715 & 13635 & 0.19 & 669 & 25830 & 9090\\
ns4-pr6                        & 1589 & 4991 & 0.00 & 261 & 395 & 1146 & 0.00 & 243 & 384 & 1108\\
nu120-pr12                     & 1515 & 4444 & 0.02 & 1127 & 1363 & 3939 & 0.04 & 1402 & 1360 & 3927\\
nu120-pr9                      & 1823 & 6780 & 0.50 & 4785 & 1803 & 6702 & 0.50 & 4991 & 1803 & 6702\\
nu25-pr12                      & 1516 & 4445 & 0.03 & 1289 & 1347 & 3889 & 0.04 & 1343 & 1344 & 3877\\
nu4-pr9                        & 1823 & 6780 & 0.76 & 7556 & 1803 & 6702 & 0.64 & 6286 & 1803 & 6702\\
p0201                          & 110 & 183 & 0.00 & 44 & 80 & 102 & 0.00 & 26 & 50 & 52\\
physiciansched6-1              & 38870 & 26248 & 5.60 & 15821 & 38104 & 25478 & 7.92 & 16836 & 38139 & 25452\\
physiciansched6-2              & 31064 & 17093 & 10.48 & 25047 & 30830 & 16963 & 9.72 & 22787 & 30850 & 16943\\
piperout-27                    & 17225 & 11403 & 1.37 & 11423 & 17225 & 11403 & 0.81 & 10283 & 17226 & 11403\\
pythago7824                    & 7326 & 3740 & 34.37 & 39996 & 7326 & 3740 & 0.00 & 0 & 1 & 1\\
pythago7825                    & 7336 & 3745 & 33.88 & 39996 & 7336 & 3745 & 0.00 & 0 & 1 & 1\\
supportcase33                  & 9862 & 13641 & 0.21 & 549 & 9862 & 13641 & 0.29 & 555 & 9901 & 13631\\
supportcase37                  & 39192 & 9674 & 3.44 & 15069 & 36075 & 8800 & 3.62 & 15156 & 36084 & 8800\\
tokyometro                     & 6835 & 3606 & 0.06 & 1792 & 6791 & 3575 & 0.07 & 1852 & 6919 & 3567\\
transportmoment                & 7312 & 7268 & 0.60 & 8299 & 6989 & 6946 & 0.55 & 7757 & 6719 & 6646\\
uccase10                       & 140361 & 73892 & 145.33 & 94940 & 79074 & 39006 & 12.57 & 30401 & 26022 & 12518\\
uccase12                       & 88635 & 40242 & 9.46 & 35027 & 60606 & 26822 & 2.53 & 14668 & 36541 & 15758\\
uccase7                        & 31928 & 21289 & 90.57 & 47420 & 31093 & 20786 & 83.19 & 44585 & 31581 & 20786\\
unitcal\_7                      & 41890 & 21979 & 2.56 & 16471 & 32001 & 16779 & 2.15 & 15560 & 33834 & 16779\\
woodlands09                    & 36150 & 227094 & 3600.49 & 165913 & 26988 & 110657 & 3600.55 & 364827 & 26988 & 110657
\end{longtable}
}
\end{appendices}

\end{document}